\documentclass[11pt,reqno]{amsart}

\usepackage{lipsum}
\usepackage{amsmath}
\usepackage[utf8]{inputenc}
\usepackage{amssymb}
\usepackage{ragged2e}
\usepackage{xfrac}
\usepackage[dvipsnames,table]{xcolor}

\usepackage{amsthm}
\usepackage{amsfonts}
\usepackage{amssymb}
\usepackage{tikz}
\usepackage{color}
\usepackage{float}
\usepackage[makeroom]{cancel}
\usepackage[centertableaux]{ytableau}
\usepackage{mathdots}
\usepackage{multicol}
\usepackage{mathtools}
\usepackage{stmaryrd}
\usepackage{multirow}
\usepackage{pdflscape}
\usepackage[foot]{amsaddr}

\usepackage[left=2.5cm,right=2.5cm,top=2.5cm,bottom=2.5cm]{geometry}

\DeclareMathSymbol{\shortminus}{\mathbin}{AMSa}{"39}
\numberwithin{equation}{section}

\usepackage[colorlinks=true,citecolor=RoyalBlue,linkcolor=OrangeRed,breaklinks=false]{hyperref}

\newtheorem{teo}{Theorem}[section]
\newtheorem{cor}[teo]{Corollary}
\newtheorem{prop}[teo]{Proposition}
\newtheorem{lema}[teo]{Lemma}

\theoremstyle{definition}

\newtheorem{defin}[teo]{Definition}

\theoremstyle{remark}
\newtheorem{obs}[teo]{Remark}
\newtheorem{ex}[teo]{Example}

\title{A shifted Berenstein--Kirillov group and the cactus group}
\author{Inês Rodrigues}
\address{CEAFEL, University of Lisbon}
\email{\href{mailto:imarrodrigues@fc.ul.pt}{imarrodrigues@fc.ul.pt}}

\definecolor{lgray}{rgb}{0.65, 0.65, 0.65}
\definecolor{lblue}{rgb}{0.65, 0.65, 0.65}
\definecolor{llblue}{rgb}{0.85, 0.85, 0.85}

\begin{document}
\begin{abstract}
The Bender--Knuth involutions on semistandard Young tableaux are known to coincide with the tableau switching on horizontal border strips of two adjacent letters, together with the swapping of those letters. Motivated by this coincidence and using the shifted tableau switching due to Choi, Nam and Oh (2019), we consider a shifted version of the Bender--Knuth involutions and define a shifted version of the Berenstein--Kirillov group (1995). Similarly to the classical case, the shifted version of the Berenstein--Kirillov group also acts on the straight-shaped shifted tableau crystals introduced by Gillespie, Levinson and Purbhoo (2020), via partial Schützenberger involutions, thus coinciding with the action of the cactus group on the same crystal, due to the author. Following the works of Halacheva (2016, 2020), and Chmutov, Glick and Pylyavskyy (2020), on the relation between the actions of the Berenstein--Kirillov group and the cactus group on a crystal of straight-shaped Young tableaux, we also show that the shifted Berenstein--Kirillov group is isomorphic to a quotient of the cactus group. Not all the known relations that hold in the classic Berenstein--Kirillov group need to be satisfied by the shifted Bender--Knuth involutions, but the ones implying the relations of the cactus group are verified. Hence, we have an alternative presentation for the cactus group in terms of the shifted Bender--Knuth involutions. We also use the shifted growth diagrams due to Thomas and Yong (2016) to provide an alternative proof concerning the mentioned cactus group action.
\end{abstract}

\maketitle
\ytableausetup{smalltableaux}	

\section{Introduction}
The Bender--Knuth moves $t_i$ are well known involutions on semistandard Young tableaux 
\cite{BeKn72}, that act on adjacent letters $i$ and $i+1$ by interchanging their multiplicity, while leaving the other letters unchanged. The tableau switching, introduced by Benkart, Sottile and Stroomer 
\cite{BSS96}, is an involution on pairs of semistandard Young tableaux $(S,T)$, with $T$ extending $S$, that moves one through the other, obtaining a pair that is component-wise Knuth equivalent to $(T,S)$. 
The tableau switching on horizontal border strips of two adjacent letters $i$ and $i+1$, together with a swapping of the labels $i$ and $i+1$, is known to coincide with the classic Bender--Knuth involution $t_i$
\cite{BSS96,PV10}.
Berenstein and Kirillov 
\cite{BK95} studied explicit relations satisfied by the involutions $t_i$ 
\cite[Corollary 1.1]{BK95}, and introduced the Berenstein--Kirillov group $\mathcal{BK}$ (also known as Gelfand--Tsetlin group), the free group generated by the classic Bender--Knuth involutions $t_i$, for $i \in \mathbb{Z}_{>0}$, subject to the relations they satisfy on semistandard Young tableaux of any shape 
\cite{BK16,BK95,CGP16}. This group is well-defined although an explicit and comprehensive set of relations is not known. Some of the relations that are held by the  $t_i$ are listed in
\cite{BK16,BK95,Kir01}, and \cite[Theorem 1.6]{CGP16}, and they are recalled in Section \ref{subsec:bk}.

Chmutov, Glick and Pylyavskyy \cite{CGP16} studied, using semistandard growth diagrams, the relation between the group $\mathcal{BK}_n$, the subgroup of $\mathcal{BK}$ generated by $t_1, \ldots, t_{n-1}$, and the cactus group $J_n$ 
\cite{HenKam06}, concluding that $\mathcal{BK}_n$ is isomorphic to a quotient of $J_n$. Halacheva has remarked 
\cite[Remark 3.9]{Hala20} that this may also be concluded by noting that the action of the cactus group $J_n$ 
\cite[Section 10.2]{Hala16} agrees with the one of $\mathcal{BK}_n$ on $\mathfrak{gl}_n$-crystals of straight-shaped Young tableaux filled in $[n] = \{1, \ldots, n\}$. Considering the alternative set of generators $q_1, \ldots, q_{n-1}$ for $\mathcal{BK}_n$, where each $q_i := t_1 (t_2 t_1) \cdots (t_i t_{i-1} \cdots t_1)$ acts on a straight-shaped Young tableau via the partial Schützenberger involution, or evacuation, restricted to the alphabet $\{1,\dots,i+1\}$ \cite[Theorem 2.1]{BK95}. Chmutov, Glick and Pylyavskyy also refine their results concerning the cactus group quotient in \cite[Theorem 1.8]{CGP16} by showing precise implications between the cactus-type relations, satisfied by generators of $\mathcal{BK}_n$, and a subset of known relations \eqref{eq:bkrelations} and \eqref{eq:bkrelationextra} in the $\mathcal{BK}_n$, thereby yielding a presentation of the cactus group in terms Bender--Knuth generators.

Motivated by the tableau switching characterization of the Bender--Knuth involutions on semistandard Young tableaux
\cite{BSS96}, we introduce a shifted version of the Bender--Knuth involutions, here denoted $\mathsf{t}_i$, for shifted semistandard tableaux in a shifted tableau crystal due to Gillespie, Levinson and Purbhoo \cite{GLP17}, using the shifted tableau switching introduced by Choi, Nam and Oh 
\cite{CNO17}. Alternatively, we may use the type $C$ infusion on shifted standard tableaux due to Thomas and Yong 
\cite{TY09} together with the semistandardization of Pechenik and Yong
\cite{PY17}. We observe that genomic Bender--Knuth involutions have also been defined in a similar way on genomic tableaux, by Pechenik and Yong
\cite{PY17}.
The shifted Bender--Knuth involutions we present differ from the operators introduced by Stembridge 
\cite[Section 6]{Stem90}, which are not compatible with the canonical form requirement for the shifted tableau crystals considered here (see Remark \ref{rmk:stembridge}). Using the shifted Bender--Knuth involutions $\mathsf{t}_i$ as generators, we define a shifted analogue of the Berenstein--Kirillov group, denoted $\mathcal{SBK}$, with $\mathcal{SBK}_n$ being defined analogously. 

Following \cite{BK95}, the elements $\mathsf{q}_i:=\mathsf{t}_1 (\mathsf{t}_2 \mathsf{t}_1) \cdots (\mathsf{t}_i \mathsf{t}_{i-1} \cdots \mathsf{t}_1)$, for $1 \leq i \leq n-1$, also constitute an alternative set of generators for $\mathcal{SBK}_n$.  Similarly to the $\mathcal{BK}_n$ group, each generator $\mathsf{q}_i$ acts on a straight-shaped shifted semistandard tableau, via the shifted Schützenberger involution restricted to the primed alphabet $\{1', 1, \dots, i', i\}$. Thereby, as in the classic case \cite {CGP16,Hala16,Hala20},  the actions of the cactus group $J_n$  \cite[Theorem 5.7
]{Ro20b} (here Theorem \ref{teo:cact_evaci}) and of $\mathcal{SBK}_n$ agree on a straight-shaped shifted tableau crystal \cite{GLP17}. Thus, the shifted Berenstein--Kirillov group is isomorphic to a quotient of the cactus group (Theorem \ref{teo:cact_sbk}).

The shifted Bender--Knuth operators $\mathsf{t}_i$ also satisfy the $\mathcal{BK}$-type relations  \eqref{eq:bkrelations} and \eqref{eq:bkrelationextra}. Those are the relations satisfied by the generators $t_i$ in $\mathcal{BK}$ which are equivalent to the ones of the cactus group, as shown in 
\cite[Theorem 1.8]{CGP16} (here Theorem \ref{thm:rel_cact_bk}). Thus, we also have, similarly to the classic case \cite{CGP16}, another presentation of the cactus group via the shifted Bender--Knuth moves.

Not all known relations that hold in $\mathcal{BK}$ need to be satisfied by the shifted Bender-Knuth involutions, namely the relation $(\mathsf{t}_1 \mathsf{t}_2)^6 =1$ \cite[Proposition 1.3]{BK95} (\eqref{eq:bkrelations_special}) does not need to hold in $\mathcal{SBK}$ (see Example \ref{ex:t1t26}). As observed in
\cite[Remark 9]{CGP16}, the relation $(t_1 t_2)^6=1$ \eqref{eq:bkrelations_special} in $\mathcal{BK}$ does not follow from any cactus group relation. In fact, it is equivalent to the braid relations of the symmetric group $\mathfrak{S}_n$, satisfied by the type $A$ crystal reflection operators $\varsigma_i$, due to Lascoux and Schützenberger \cite{LaSchu81}, and rediscovered by Kashiwara \cite[Theorem 7.2.2]{Kash94}. These operators are elements of $\mathcal{BK}$ \cite[Proposition 1.4]{BK95}, and $\varsigma_i$ acts on a $\mathfrak{gl}_n$-crystal as a middle reflection of each $i$-string, which agrees with the partial Schützenberger involution restricted to the alphabet $\{i,i+1\}$, for $1 \leq  i \leq n-1$. 

The shifted crystal reflection operators $\sigma_i$, for $1 \leq i \leq n-1$ \cite[Definition 4.3]{Ro20b} are  also elements of $\mathcal{SBK}_n$, and  $\sigma_i$ acts on a shifted tableau crystal as a double reflection of each $\{i,i'\}$-coloured connected component, which agrees with the shifted Schützenberger involution restricted to the primed alphabet $\{i,i+1\}'$.  A relation of the type $(\mathsf{t}_1 \mathsf{t}_2)^{2 m} =1$ holds in $\mathcal{SBK}_n $ if and only if the relation $(\sigma_i\sigma_{i+1})^{m}=1$ does, where $m$ is a positive integer (see Proposition \ref{prop:braid_t6}). However, unlike type $A$ crystals, the shifted crystal reflection operators do not define an action of the symmetric group, thus none of the aforesaid relations holds for $m = 3$. It is not known whether some $m > 3$ exists such that the said relation holds \cite[Appendix A]{Ro20b}. It is an open question to find explicit relations in $\mathcal{SBK}$, beyond those listed in Proposition \ref{prop:rels_SBK}, that do not follow from the cactus group relations. Further relations for $\mathcal{SBK}$ seem to be intimately related with further relations satisfied by the shifted crystal reflection operators.

The proof in \cite[Theorem 5.7]{Ro20b} concerning a cactus group action on a shifted tableau crystal relies on the formulation of the Schützenberger involution as the unique set involution on a shifted tableau crystal satisfying certain conditions in terms of the shifted crystal operators 
\cite[Proposition 4.1]{Ro20b} (see Proposition \ref{prop:Schu}). Thus, the partial Schützenberger involutions, corresponding to the restrictions of the Schützenberger involutions to all primed subintervals of $[n]$, are also described in an analogous way 
\cite[Lemma 5.4]{Ro20b}, similarly to what is done in \cite[Definition 5.17]{HaKaRyWe20}.
Those set involutions on a shifted tableau crystal coincide with the \emph{shifted reversal} map, or the \emph{evacuation} on straight-shaped tableaux (Section \ref{ss:evac}), and its restrictions, and thus are regarded as explicit involutions on shifted tableaux. Sticking to this algorithmic formulation, we may use type $C$ growth diagrams, introduced by Thomas and Yong 
\cite{TY16}, together with the semistandardization process due to Pechenik and Yong 
\cite{PY17}, to obtain an alternative proof that the cactus group acts on a shifted tableau crystal via the restrictions of the reversal involution.

The type $C$ growth diagrams, for shifted standard tableaux, were introduced by Thomas and Yong
\cite{TY16}, together with generalizations for other cominuscule posets, and they generalize the classic growth diagrams for standard Young tableaux due to Fomin
\cite{Sta99}. These diagrams consist of saturated chains of shifted shapes encoding the shifted \textit{jeu de taquin} for shifted standard tableaux. Thus, they define type $C$ infusion, as well as the shifted promotion, evacuation and reversal, and the adequate restrictions. Like the classic growth diagrams 
\cite[Proposition A1.2.7]{Sta99}, the shifted ones may be computed via local growth rules \cite[Theorem 2.1]{TY16}. The symmetry of those rules shows that the type $C$ infusion, evacuation and reversal are involutions. 
Unlike the case for type $A$, shifted semistandard tableaux, being filled in a primed alphabet, are not encoded by a sequence of strict shapes and thus we do not have a semistandard-like growth diagrams as in 
\cite{CGP16}. However, the shifted semistandardization due to Pechenik and Yong
\cite{PY17} allows us to extend these notions for semistandard shifted tableaux. Thus, we are able to obtain an alternative proof, in Section \ref{sec:gd}, for the cactus group action on a shifted tableau crystal \cite[Theorem 5.7]{Ro20b} (here Theorem \ref{thm:cactusaction}), relying on the combinatorial description of the shifted reversal.

This paper is organized as follows. Section \ref{sec:background} provides the basic definitions and algorithms on shifted tableaux, in particular, the reversal and evacuation, as well as the main concepts regarding the shifted tableau switching 
\cite{CNO17}. We also emphasize that the shifted tableau switching agrees with the type $C$ infusion map on standard shifted tableaux, and thus the result on semistandard shifted tableaux may be recovered using the semistandardization map \cite{PY17}.
In Section \ref{sec:crystal}, we briefly recall the basic notions of the crystal-like structure on shifted tableaux, due to Gillespie, Levinson, and Purbhoo 
\cite{GLP17}, and an action of the cactus group 
\cite{Ro20a, Ro20b}, due to the author, on that crystal. In Section \ref{secBK}, we introduce the shifted Bender--Knuth involutions (Definition \ref{def:sbk_mov}), using the shifted tableau switching. Then, as in the classic case, we use those shifted Bender--Knuth involutions to define a shifted Berenstein--Kirillov group. Proposition \ref{prop:rels_SBK} shows that the known relations \eqref{eq:bkrelations} and \eqref{eq:bkrelationextra} satisfied by the classical Bender--Knuth involutions also hold among the shifted counterparts, with the exception of the relation $(\mathsf{t}_1\mathsf{t}_2)^6=1$. We then prove the main result (Theorem \ref{teo:cact_sbk}) which states that the shifted Berenstein--Kirillov group is isomorphic to a quotient of the cactus group and exhibit in \eqref{eq:cactus_alt_prst} an alternative presentation for the cactus group in terms of the shifted Bender--Knuth moves. 
In Section \ref{sec:gd}, we recall the notion of growth diagrams for shifted standard tableaux, as well as the local growth rules \cite{TY09}. Using the semistandardization \cite{PY17}, we are able to recover the shifted \textit{jeu de taquin}, type $C$ infusion, evacuation and reversal, as well as their restrictions, to semistandard shifted tableaux. We then provide an alternative proof of \cite[Theorem 5.7]{Ro20b}, using growth diagrams, that the cactus group acts on a shifted tableau crystal via the partial Schützenberger involutions.

An extended abstract of part of this work was accepted for publication in a proceedings volume of the Séminaire Lotharingien de Combinatoire.

\section{Background}\label{sec:background}
A \emph{strict partition} is a sequence $\lambda = (\lambda_1 > \cdots > \lambda_k)$ of distinct positive integers displayed in strictly decreasing order. The entries $\lambda_i$ are called the \emph{parts} of $\lambda$ and the \emph{length} of $\lambda$, denoted $\ell(\lambda)$, is the number of non-zero parts of $\lambda$. We denote by $|\lambda| := \lambda_1 + \cdots + \lambda_k$ the \emph{sum} of the parts of $\lambda$. A strict partition $\lambda$ is identified with its \emph{shifted shape} $S(\lambda)$ which consists of boxes placed in $\ell(\lambda)$ rows, with the $i$-th row having $\lambda_i$ boxes and being shifted $i-1$ units to the right. We use the English (or matrix) notation. The boxes in $\{(1,j), (2,j+1), (3,j+2), \ldots\}$ form a \emph{diagonal}, for $j \geq 1$. If $j=1$ it is called the \emph{main diagonal}. 
Given strict partitions $\lambda$ and $\mu$ such that $S(\mu) \subseteq S(\lambda)$, we write $\mu \subseteq \lambda$ and define the \emph{skew shifted shape} of $\lambda/\mu$ as $ S(\lambda/\mu) = S(\lambda) \setminus S(\mu)$ (see Figure \ref{fig:lambdamu}). Shapes of the form $\lambda/\emptyset$ are called \emph{straight} (or \emph{normal}). Any shifted shape $\lambda$ lies naturally in the ambient triangle of the \emph{shifted staircase shape} $\delta = (\lambda_1, \lambda_1-1, \ldots, 1)$. We define the \emph{complement} of $\lambda$ to be the strict partition $\lambda^{\vee}$ whose set of parts is the complement of the set of parts of $\lambda$ in $\{\lambda_1, \lambda_1-1, \ldots, 1\}$. In particular, $\emptyset^{\vee} = \delta$ (see Figure \ref{fig:lambdamu}).

\begin{figure}[h]
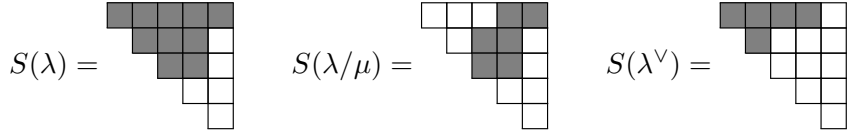

\begin{center}
$S(\lambda)=
	\ytableausetup{smalltableaux}	
	\begin{ytableau}
    *(gray)  & *(gray)  &*(gray)  &*(gray)  &*(gray)  \\
    \none & *(gray)  & *(gray)  & *(gray) &   \\
    \none & \none & *(gray) & *(gray)  & \\
    \none & \none & \none & & \\
    \none & \none & \none & \none &
  \end{ytableau}
  \qquad
  S(\lambda/\mu)=
	\ytableausetup{smalltableaux}	
	\begin{ytableau}
    {}  & {} & {} &*(gray)  &*(gray)  \\
    \none & {}  & *(gray)  & *(gray) &   \\
    \none & \none & *(gray) & *(gray) & \\
    \none & \none & \none & & \\
    \none & \none & \none & \none &
  \end{ytableau}
  \qquad
 S(\lambda^{\vee})=
	\ytableausetup{smalltableaux}	
	\begin{ytableau}
    *(gray)  & *(gray)  &*(gray)  &*(gray)  &   \\
    \none & *(gray)  &    &  &   \\
    \none & \none &   &  & \\
    \none & \none & \none & & \\
    \none & \none & \none & \none &
  \end{ytableau}
  $
 \end{center}
 \caption{The shapes of $\lambda$, $\lambda/\mu$ and $\lambda^{\vee}$, shaded in gray, for $\lambda= (5,3,2)$ and $\mu = (3,1)$.}
 \label{fig:lambdamu}
 \end{figure}

We consider $[n] := \{1 \!<\! \cdots \!<\! n\}$ and define the \emph{primed} alphabet  $[n]'$ to be $\{1'\!<\! 1 \!<\! \cdots \!<\! n' \!<\! n\}$. Following the notation in \cite{CNO17}, we will write $\mathbf{i}$ when referring to the letters $i$ and $i'$ without specifying whether they are primed. Given a string $w = w_1 \cdots w_m$ in the alphabet $[n]'$, the \emph{canonical form} \cite[Definition 2.1]{GLP17} of $w$ is the string obtained from $w$ by replacing the leftmost $\mathbf{i}$, if it exists, with $i$, for all $1 \leq i \leq n$. Two strings $w$ and $v$ are said to be \emph{equivalent} if they have the same canonical form. A \emph{word} $\hat{w}$ is defined to be an equivalence class $\hat{w}$ of the strings equivalent to $w$ \cite[Definition 2.2]{GLP17}. If $w$ is in canonical form, then it is said to be the \emph{canonical representative} of $\hat{w}$, while the other strings are called the \emph{representatives} of $\hat{w}$. Whenever there is no risk of ambiguity, we refer to $\hat{w}$ by its canonical representative $w$. The \emph{weight} of a word $w$ is $\mathsf{wt}(w) = (wt_1, \ldots, wt_n)$, where $wt_i$ is equal to the total number of $i$ and $i'$ in $w$. We remark that the weight of a word does not depend on the choice of representative, as the number of $i$ and $i'$ is the same for all representatives, for $i \in [n]$.
	
\begin{defin}
Given strict partitions $\lambda$ and $\mu$ such that $\mu \subseteq \lambda$, a \emph{shifted semistandard tableau} $T$ of shape $\lambda / \mu$ is a filling of $S(\lambda/\mu)$ with letters in $[n]'$ such that the entries are weakly increasing in each row and in each column, and there is at most one $i$ per column and one $i'$ per row, for any $i \geq 1$.
\end{defin}

The \emph{reading word} $w(T)$ of a shifted tableau is obtained by concatenating its rows, going from bottom to top. The \emph{weight} of $T$ is defined as $\mathsf{wt}(T) := \mathsf{wt}(w(T))$. A word or a shifted tableau are said to be \emph{standard} if their weight is $(1, \ldots, 1)$.

\begin{ex}
The following is a shifted semistandard tableau, with its reading word and weight:
$$T=\begin{ytableau}
{} & 1 & 1 & 2' & 2\\
\none & 2 & 3'\\
\none & \none & 3
\end{ytableau} \qquad w(T) = 323'112'2 \qquad \mathsf{wt}(T)=(2,3,2).$$
\end{ex}

We say that a tableau $T$ is in \emph{canonical form} if its reading word is in canonical form and, in that case, it is identified with its set of \emph{representatives}, that are obtained by possibly priming the entry corresponding to the first $i$ in $w(T)$, for all $i$. We denote the set of shifted semistandard tableaux of shape $\lambda/\mu$, on the alphabet $[n]'$, in canonical form, by $\mathsf{ShST}(\lambda/\mu,n)$.	

\begin{ex}
The tableau of the previous example is in canonical form, as the first occurrences of each letter is unprimed. Some of its representatives are listed below. Their reading words are representatives of the class of $w(T)$.
$$\begin{ytableau}
{} & 1 & 1 & 2' & 2\\
\none & 2 & 3'\\
\none & \none & 3
\end{ytableau} \qquad
\begin{ytableau}
{} & 1' & 1 & 2' & 2\\
\none & 2 & 3'\\
\none & \none & 3
\end{ytableau} \qquad
\begin{ytableau}
{} & 1 & 1 & 2' & 2\\
\none & 2' & 3'\\
\none & \none & 3
\end{ytableau} \qquad
\begin{ytableau}
{} & 1 & 1 & 2' & 2\\
\none & 2' & 3'\\
\none & \none & 3'
\end{ytableau} \qquad \ldots$$
\end{ex}

In the remaining of the article, we will consider the \emph{symmetric group} $\mathfrak{S}_n$ to be the Coxeter group generated by $\theta_1, \ldots, \theta_{n-1}$, subject to the relations
\begin{equation}
\theta_i^2 = 1, \qquad \theta_i \theta_j = \theta_j \theta_i,\,\text{for}\, |i-j|>1, \qquad (\theta_i \theta_{i+1})^3 = 1,\,\text{for}\, 1\leq i \leq n-2.
\end{equation}

The elements of $\mathfrak{S}_n$ are explicitly described by the permutations of $[n]$, and its generators $\theta_i$ are described by the simple transpositions\footnote{We use the cycle notation.} $(i,i+1)$, for $1 \leq i \leq n-1$. A permutation $\tau \in \mathfrak{S}_n$ acts naturally on a vector of $\mathbb{Z}^n$ as $\tau (v_1, \ldots, v_n) := (v_{\tau^{-1} (1)}, \ldots, v_{\tau^{-1}(n)})$, and on letters of the marked alphabet $\mathbf{x} \in [n]'$ as
\begin{equation}\label{eq:theta_pri}
\tau (\mathbf{x}) := \begin{cases}
\tau(x) &\text{if}\; \mathbf{x}=x\\
\tau(x)' &\text{if}\; \mathbf{x}=x'
\end{cases}.
\end{equation}
According to this action, given $\tau \in \mathfrak{S}_n$ and a word in the alphabet $[n]'$, we define $\tau (w_1 \cdots w_k)$ as the word $\tau(w_1) \cdots \tau(w_k)$, after canonicalizing, for $w_i \in [n]'$. Similarly, the action of $\tau$ is extended to fillings $T$ in $[n]'$ of a shifted shape (in particular, this includes shifted semistandard tableaux), defining $\tau(T)$ by the action of $\tau$ on the word of $T$. Given $1 \leq i < j \leq j$, we denote by $\theta_{i,j}$ the longest permutation in $\mathfrak{S}_{\{i, \ldots,j\}}$ embedded in $\mathfrak{S}_n$, i.e., $\theta_{i,j} = \theta_i (\theta_{i+1} \theta_i) \cdots (\theta_{j-1} \cdots \theta_i)$. In particular, $\theta_{1,n}$ is the longest permutation in $\mathfrak{S}_n$, also knows as the order reversing permutation.
	
\subsection{Shifted \textit{jeu de taquin}, evacuation and reversal}\label{ss:evac}
The shifted \emph{jeu de taquin} \cite{Sag87,Wor84} is defined similarly to the one for ordinary Young tableaux. A skew shape $S(\lambda/\mu)$ is said to be a \emph{border strip} if it contains no subset of the form $\{(i,j),(i+1,j+1)\}$ and a \emph{double border strip} if it contains no subset of the form $\{(i,j),(i+1,j+1),(i+2,j+2)\}$. 

\begin{defin}
Let $T \in \mathsf{ShST}(\lambda/\mu,n)$ and let $i \in [n]$. The tableau obtained from $T$ considering only the letters $i$ and $i'$ is called the \emph{$i$-border strip} of $T$, and is denoted by $T^i$.
\end{defin}

Given strict partitions $\nu \subseteq \mu \subseteq \lambda$, we say that $\lambda/\mu$ \emph{extends} $\mu/\nu$, and, in this case, we define
$$(\mu/\nu) \sqcup (\lambda/\mu) := \lambda/\nu.$$
Given $S$ and $T$ shifted semistandard tableaux, we say that $T$ \emph{extends} $S$ if the shape of $T$ extends the shape of $S$. In this case, we denote by $S \sqcup T$ the union of $S$ and $T$, obtained by overlapping the two tableaux, which is not necessarily a valid semistandard tableau. A shifted semistandard tableau $T$ filled in $[n]'$ is clearly a disjoint union of its $i$-border strips, for $i \in [n]$.

\begin{ex}
Considering $T = \begin{ytableau}
{} & 1 & 1 & 2' & 2\\
\none & 2 & 3'\\
\none & \none & 3
\end{ytableau}$, we have 

$$T = \begin{ytableau}
{} & 1 & 1 & {} & {}\\
\none & {} & {}\\
\none & \none & {}
\end{ytableau} \sqcup
\begin{ytableau}
{} & {} & {} & 2' & 2\\
\none & 2 & {}\\
\none & \none & {}
\end{ytableau} \sqcup
\begin{ytableau}
{} & {} & {} & {} & {}\\
\none & {} & 3'\\
\none & \none & 3
\end{ytableau}
 = T^1 \sqcup T^2 \sqcup T^3.$$
\end{ex}

A single box $b$ is said to be an \emph{inner corner} of a shape $\lambda/\mu$ if $\lambda/\mu$ extends $b$, and an \emph{outer corner} if $b$ extends $\lambda/\mu$.

\begin{defin}[{\cite[Section 6.4]{Wor84}}]\label{def:sh_jdt}
Let $T\in \mathsf{ShST}(\lambda/\mu,n)$.  An \emph{inner jeu de taquin slide} is the process in which an empty inner corner of the skew shape of $T$ is chosen and then either the entry to its right or the one below it is chosen to slide into the empty square, maintaining semistandardness. The process is then repeated on the obtained new empty square until it is an outer corner. An \emph{outer jeu de taquin slide} is the reverse process, starting with an outer corner. This process has an exception to the sliding rules when the empty box of an inner or outer slide enters in the diagonal. If an inner slide moves a box with $a'$ to the left into the diagonal and then moves a box with $a$ up from the diagonal, to the right of it, the former becomes unprimed (and vice versa for the corresponding outer slide), as illustrated by the following slide:

\begin{center}
\includegraphics[scale=0.4]{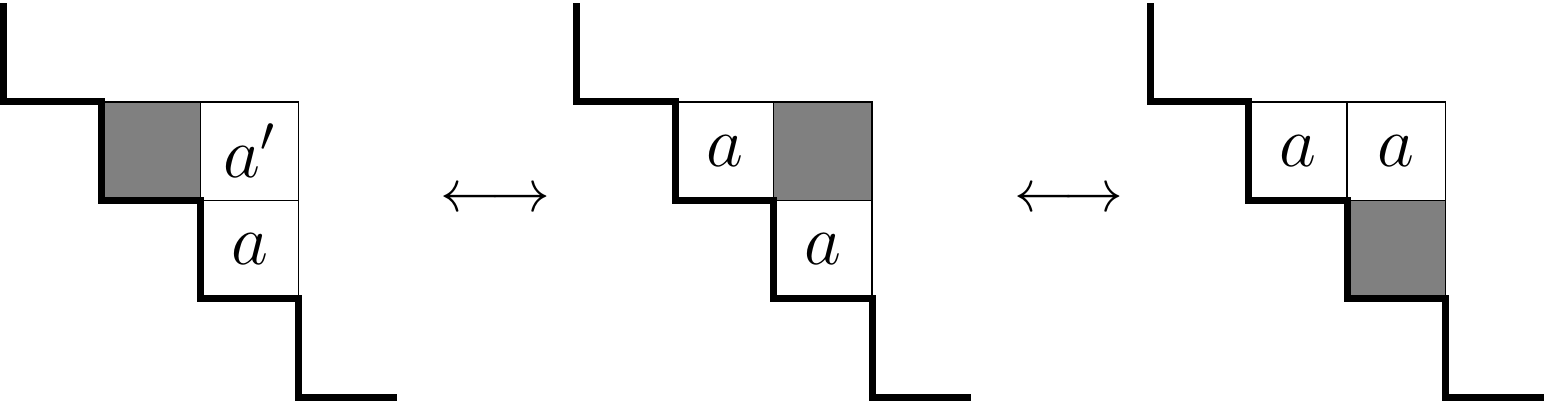}
\end{center}
\end{defin}

If $T$ is not in the canonical form, there is another exception to consider illustrated below (observe that result is in the same canonical class of the former case):

\begin{center}
\includegraphics[scale=0.4]{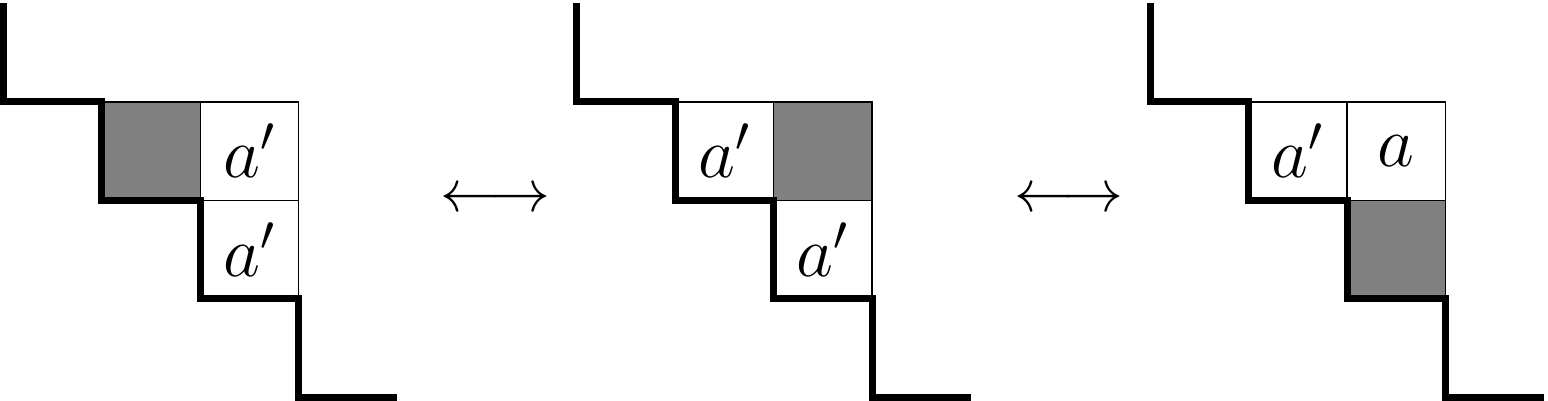}
\end{center}

The \emph{rectification} $\mathsf{rect}(T)$ of $T$ is the tableau obtained by applying any sequence of inner slides until a straight shape is obtained (it is known that any chosen sequence of slides produces the same straight-shaped tableau \cite[Theorem 11.1]{Sag87}). The \emph{rectification} of a word $w$ is the word of the rectification of any tableau with reading word $w$. Two tableaux are said to be \emph{shifted jeu de taquin equivalent} if they have the same rectification. An operator on shifted tableaux that commutes with the shifted \textit{jeu de taquin} is called \emph{coplactic}.

The \emph{standardization} of a word $w$, denoted $\mathsf{std}(w)$, is obtained by replacing the letters of any representative of $w$ with $1, \ldots, \ell(w)$, where $\ell(w)$ denotes the lenght of $w$, from least to greatest, reading right to left for primed entries, and left to right for unprimed entries \cite[Definition 2.8]{GLP17}. 

\begin{ex}\label{ex:std}
Let $T =\begin{ytableau}
1 & 1 & 2' & 2\\
\none & 2 & 2\\
\none & \none & 3
\end{ytableau}$, with reading word $w=322112'2$ and $\ell(w)=7$. Then,

$$\mathsf{std}(T) = \begin{ytableau}
1 & 2 & 3 & 6\\
\none & 4 & 5\\
\none & \none & 7
\end{ytableau}.
$$
\end{ex}

This process does not depend on the choice of the representative. The standardization of a shifted tableau $T$ is defined as the tableau with the same shape as $T$ with reading word $\mathsf{std}(w(T))$.

\begin{lema}[{\cite[Lemma 3.5]{GLP17}}]\label{standard}
If $s$ is a standard word in $[m]$, with $m =a_1+\cdots+a_k$, then there is at most one word $w$ of weight $(a_1, \dots , a_k)$ with standardization $\mathsf{std} (w) = s$.
\end{lema}
 
As a consequence, any shifted semistandard tableau is completely determined (up to canonical form) by its shape, weight and standardization. Thus, given a standard tableau $T$ of shape $\lambda/\mu$ and a composition $\nu$ (i.e., a vector of non-negative integers) such that $|\nu| = |\lambda| - |\mu|$, there exists at most one semistandard tableau with the same shape of $T$ and weight $\nu$. The process to obtain it, if it exists, is known as \emph{shifted semistandardization} and was introduced by Pechenik and Yong \cite[Section 9.1]{PY17}. Let $\nu$ be a composition and define, for $k = 1, \ldots, \ell(\nu)$,
\begin{equation}\label{eq:cal_p_nu}
\mathcal{P}_k (\nu) := \{1 + \sum\limits_{i<k} \nu_i, 2 + \sum\limits_{i<k} \nu_i, \ldots, \sum\limits_{i\leq k} \nu_i \}.
\end{equation} 
That is, $\mathcal{P}_1 = \{1, \ldots, \nu_1\}$, $\mathcal{P}_2 =\{\nu_1 +1, \ldots, \nu_1 + \nu_2\}$, etc. Note that each $\mathcal{P}_k(\nu)$ has cardinality $\nu_k$. 

\begin{defin}[{\cite[Section 9.1]{PY17}}]\label{def:sstd}
Given a shifted standard tableau $T$, its \emph{semistandardization (with respect to $\nu$)}, denoted $\mathsf{sstd}_{\nu} (T)$, is given by the following process:
	\begin{enumerate}
	\item Replace each letter $i$ with $k_i$, for the unique $k$ such that $i \in \mathcal{P}_k(\nu)$.
	\item Then, replace each $k_i$ with $k'$, if there exists a $k_j$ south-west of $k_i$ with $i < j$, or with $k$, otherwise.
	\item If the obtained filling is a semistandard tableau, then $\nu$ is said to be \emph{admissible} for $T$ and $\mathsf{sstd}_{\nu}(T)$ is set to be that tableau. Otherwise, $\mathsf{sstd}_{\nu} (T)$ is said to be undefined. 
	\end{enumerate}
\end{defin} 
	 
Note that, if $\nu$ is admissible for $T$, then $\mathsf{wt}(\mathsf{sstd}_{\nu}(T)) = \nu$. Moreover, if $T \in \mathsf{ShST}(\lambda/\mu,n)$ has weight $\nu$, then $\nu$ is admissible for $\mathsf{std}(T)$ and $\mathsf{sstd}_{\nu} (\mathsf{std}(T)) = T$ \cite[Lemma 9.5]{PY17}. A shifted tableau in these conditions is said to be \emph{$\nu$-Pieri filled}. As a consequence, $\mathsf{std}$ defines a bijection between the set of shifted semistandard tableaux of shape $\lambda/\mu$ and weight $\nu$ and the set of $\nu$-Pieri filled shifted semistandard tableaux of the same shape, whose inverse is given by $\mathsf{sstd}_{\nu}$ \cite[Theorem 9.6]{PY17}.

\begin{ex}
Let $T= \begin{ytableau}
1 & 2 & 3 & 6\\
\none & 4 & 5\\
\none & \none & 7
\end{ytableau}$ be a shifted standard tableau and let $\nu = (2,4,1)$. We have:
$$\mathcal{P}_1 (\nu) = \{1,2\} \qquad \mathcal{P}_2(\nu) = \{3,4,5,6\} \qquad \mathcal{P}_3 (\nu) = \{7\}.$$
Then, the semistandardization of $T$ with respect to $\nu$ is obtained as follows:

$$\begin{ytableau}
1 & 2 & 3 & 6\\
\none & 4 & 5\\
\none & \none & 7
\end{ytableau}
\longrightarrow
\begin{ytableau}
1_1 & 1_2 & 2_3 & 2_6\\
\none & 2_4 & 2_5\\
\none & \none & 3_7
\end{ytableau}
\longrightarrow
\begin{ytableau}
1 & 1 & 2' & 2\\
\none & 2 & 2\\
\none & \none & 3
\end{ytableau}.$$
\end{ex}

Given $\nu$ a strict partition, there exists a unique (up to canonical form) shifted tableau of shape and weight equal to $\nu$. This is known as the \emph{Yamanouchi tableau} $Y_{\nu}$, and its $i$-th row consists only of unprimed $i$'s. A word $w$ on the alphabet $[n]'$ with weight $\nu$, a strict partition, is said to be \emph{ballot} (or \emph{lattice}, or \emph{Yamanouchi}) if its rectification is $w(Y_\nu)$. A shifted semistandard tableau $T$ of weight $\nu$ is said to be \emph{Littlewood--Richardson--Stembridge} (LRS) if $\mathsf{rect}(T)=Y_{\nu}$, or, equivalently, if its reading word is \emph{ballot}, with weight $\nu$. The \emph{shifted Littlewood--Richardson coefficient} $f_{\mu \nu}^{\lambda}$ is defined as the number of LRS tableaux of shape $\lambda/\mu$ and weight $\nu$, if $|\lambda| = |\mu| + |\nu|$ (for this and other formulations, see \cite{Stem89, Wor84}).

\begin{defin}[\cite{Sag87}]
Two words $w$ and $v$ on an alphabet $[n]'$ are said to be \emph{shifted Knuth equivalent}, denoted $w \equiv_k v$, if one can be obtained from the other by applying a sequence of the following Knuth moves on adjacent letters
	\begin{enumerate}
	\item $bac \longleftrightarrow bca$ if, under the standardization ordering, $a < b < c$.
	\item $acb \longleftrightarrow cab$ if, under the standardization ordering, $a < b < c$.
	\item $ab \longleftrightarrow ba$ if these are the first two letters.
	\item $aa \longleftrightarrow aa'$ if these are the first two letters.
	\end{enumerate}
\end{defin}

Two shifted semistandard tableaux are shifted Knuth equivalent if their reading words are shifted Knuth equivalent \cite[Theorem 12.2]{Sag87}\label{jdtknuth}, or, equivalently, if they have the same rectification \cite[Theorem 6.4.17]{Wor84} or if their words have the same Worley--Sagan insertion tableau \cite{Sag87}. Two shifted semistandard tableaux are \emph{shifted dual equivalent} (or coplactic equivalent) if they have the same shape after applying any sequence (including the empty one) of shifted \textit{jeu de taquin} slides to both. Equivalently, two tableaux are shifted dual equivalent if their words have the same recording tableau under the shifted Robinson--Schensted \cite{Haim92,Sag87,Wor84}.  In particular, any two tableaux of the same straight shape are shifted dual equivalent \cite[Corollary 2.5]{Haim92}.

Given $T\in \mathsf{ShST}(\lambda/\mu, n)$, its \emph{complement} in $[n]'$ is the tableau $\mathsf{c}_n (T)$ obtained by reflecting $T$ along the anti-diagonal in the shifted stair shape $\delta = (\lambda_1, \lambda_1 -1, \ldots, 1)$, i.e., sending each box in $(i,j)$ to $(\lambda_1 - j +1, \lambda_1-i+1)$, followed by replacing each unprimed entry $i$ with $\theta_{1,n} (i)'$ and each primed entry $i'$ with $\theta_{1,n}(i)$, where, we recall, $\theta_{1,n}$ denotes the longest permutation in $\mathfrak{S}_n$. Hence, if $T$ is of shape $\lambda/\mu$, then $\mathsf{c}_n (T)$ is of shape $ \mu^{\vee} / \lambda^{\vee}$, and if $\mathsf{wt}(T) = (wt_1, \ldots, wt_n)$, then $\mathsf{wt}(\mathsf{c}_n(T)) = \theta_{1,n}(\mathsf{wt}(T)) = (wt_n, \ldots, wt_1)$. The following result is due to Haiman.

\begin{teo}[{\cite[Theorem 2.13]{Haim92}}]\label{teo:haiman}
Given $T\in \mathsf{ShST}(\lambda/\mu,n)$\footnote{The result also holds for ordinary Young tableaux.}, there exists a unique tableau $T^e$ that is shifted Knuth equivalent to $\mathsf{c}_n (T)$ and shifted dual equivalent to $T$.
\end{teo} 
This unique tableau is known as the \emph{shifted reversal} of $T$. If $T$ is straight-shaped, then this is known as the \emph{shifted evacuation} and denoted $\mathsf{evac}(T)$.

\begin{prop}[{\cite[Definition 7.1.5, Lemma 7.1.6]{Wor84}}]
Given $T \in \mathsf{ShST}(\nu,n)$, its \emph{(shifted) evacuation} is defined as $\mathsf{evac} (T) := \mathsf{rect} (\mathsf{c}_n (T))$. Then, $\mathsf{evac}(T)$ has the same shape as $T$ and it is shifted Knuth equivalent to $\mathsf{c}_n (T)$. 
\end{prop}

Since the operator $\mathsf{c}_n$ preserves shifted Knuth equivalence \cite[Lemma 7.1.4]{Wor84}, the \emph{shifted reversal} operator may be defined as the coplactic extension of evacuation, in the sense that, we may first rectify $T$, then apply the evacuation operator, and then perform outer \textit{jeu de taquin} slides, in the reverse order defined by the previous rectification, to get a tableau $T^e$ with the same shape of $T$. 
From \cite[Corollaries 2.5, 2.8 and 2.9]{Haim92}, this tableau $T^e$ is shifted dual equivalent to $T$, besides being shifted Knuth equivalent to $\mathsf{c}_n (T)$. In particular, $\mathsf{evac}(T) = T^e$ for $T$ a straight-shaped tableau. 
This process can be re-written with the aid of the shifted tableau switching (Proposition \ref{prop:rev_sw}) to be introduced in the next section. In Section \ref{sec:gd}, we also present a growth diagram to compute the reversal (Proposition \ref{prop:rev_gr}).

\begin{ex}
Consider the following tableau in $\mathsf{ShST}(\nu,3)$, with $\nu = (4,2,1)$:
$$T = \begin{ytableau}
1 & 1 & 1 & 1\\
\none & 2 & 2\\
\none & \none & 3
\end{ytableau}.$$
To obtain $\mathsf{evac}(T)$ we first compute $\mathsf{c}_3 (T)$ and then rectify it:

\begin{align*}
T=\begin{ytableau}
1 & 1 & 1 & 1\\
\none & 2 & 2\\
\none & \none & 3
\end{ytableau} &\xrightarrow{\mathsf{c}_3}
\mathsf{c}_3 (T) = \begin{ytableau}
{} & {} & *(lblue) & 3'\\
\none & 1 & 2' & 3'\\
\none & \none & 2 & 3'\\
\none & \none & \none & 3
\end{ytableau} \rightarrow
\begin{ytableau}
{} & *(lblue)& 2' & 3'\\
\none & 1 & 2 & 3'\\
\none & \none & 3 & 3
\end{ytableau}
\rightarrow
\begin{ytableau}
*(lblue) & 1 & 2' & 3'\\
\none & 2 & 3' & 3\\
\none & \none & 3
\end{ytableau}
\rightarrow
\begin{ytableau}
1 & 2' & 3' & 3\\
\none & 2 & 3'\\
\none & \none & 3
\end{ytableau} = \mathsf{evac}(T).
\end{align*}
\end{ex}

\begin{ex}\label{ex:reversal}
Consider the following tableau in $\mathsf{ShST}(\lambda/\mu,3)$, with $\lambda = (6,5,3,1)$ and $\mu=(4,2)$:

$$
T= \begin{ytableau}
{} & {} & {} & 1' & 1\\
\none & {} & 1 & 1\\
\none & \none & 2 & 2\\
\none & \none & \none & 3
\end{ytableau}.
$$
To compute the reversal $T^e$, we first rectify $T$, recording in reverse order the outer corners resulting of the sequence of inner \textit{jeu de taquin} slides. Then, we compute the evacuation of the obtained straight-shaped tableau and perform outer \textit{jeu de taquin} slides defined by the outer corners of the previous sequence, from the smallest to the largest.

$$\begin{ytableau}
{} & {} & {} & 1' & 1\\
\none & {} & 1 & 1\\
\none & \none & 2 & 2\\
\none & \none & \none & 3
\end{ytableau}
\overset{\mathsf{rect}}\longrightarrow
\begin{ytableau}
1 & 1 & 1 & 1 & {\color{lgray} \bullet_3}\\
\none & 2 & 2 & {\color{lgray} \bullet_1}\\
\none & \none & 3 & {\color{lgray} \bullet_2}\\
\none & \none & \none & {\color{lgray} \bullet_4}
\end{ytableau}
\xrightarrow{\mathsf{evac}}
\begin{ytableau}
1 & 2' & 3' & 3 & {\color{lgray} \bullet_3}\\
\none & 2 & 3' & {\color{lgray} \bullet_1}\\
\none & \none & 3 & {\color{lgray} \bullet_2}\\
\none & \none & \none & {\color{lgray} \bullet_4}
\end{ytableau}\longrightarrow
\begin{ytableau}
{} & {} & {} & 2' & 3'\\
\none & {} & 1 & 3'\\
\none & \none & 2 & 3'\\
\none & \none & \none & 3
\end{ytableau} = T^e.
$$
\end{ex}

\subsection{Shifted tableau switching and type $C$ infusion}\label{sec:switching}
In this section we recall the \emph{shifted tableau switching} algorithm for shifted semistandard tableau due to Choi, Nam and Oh \cite{CNO17}, which will be used later in Section \ref{secBK} to introduce a shifted version of the Bender--Knuth involutions. We also recall how the evacuation and reversal algorithms of Section \ref{ss:evac} can be formulated using the shifted tableau switching.

The tableau switching algorithm for type $A$ is an involution that, given a pair of tableaux $(S,T)$, with $S$ extending $T$, moves one trough another, using switches similar to the \textit{jeu de taquin} slides, regarding the boxes in $S$ as inner corners, and keeping semistandardness, whitin each of the alphabets, in the intermediate steps \cite{BSS96}. Any chosen sequence of those switches produces the same final result 
\cite[Theorem 2.2]{BSS96}. This is not the case for the \emph{shifted tableau switching}, which must be performed following a determined sequence of switches, similarly to the type $A$ \emph{infusion map} 
\cite{TY09,TY16}. As observed in \cite[Remark 8.1]{CNO17}, the resulting pair obtained by the shifted tableau switching can be recovered alternatively, using the type $C$ infusion map of Thomas and Yong \cite{TY09} on a pair of standardized tableaux, followed by the \emph{semistandardization} of Pechenik and Yong \cite{PY17}. The infusion map on type $A$ standard tableaux is a special case of the tableau switching process \cite{BSS96}, in which the order to perform the switches is determined by the entries of the standardization of the inner-most tableau. Unlike the case for ordinary Young tableaux, the shifted tableau switching process comprehends a determined sequence of switches to be performed, which agrees with the one prescribed by the type $C$ infusion map on shifted standard tableaux (Proposition \ref{prop:sw_infu}). Furthermore, it is compatible with standardization \cite[Remark 3.8]{CNO17}. This will be illustrated in Example \ref{ex:inf_sstd}.

We begin by recalling the definitions of the shifted tableau switching for pairs $(A,B)$ of border strip shifted tableaux, with $B$ extending $A$, and for pairs of shifted semistandard tableaux $(S,T)$, with $T$ extending $S$. We omit most of the details and proofs, and refer to \cite{CNO17}. Recall that $\mathbf{i}$ denotes either the letters $i$ or $i' \in [n]'$.
	
\begin{defin}[{\cite[Definition 3.1]{CNO17}}] Let $S(\lambda/\mu)$ be a double border strip, i.e., a shape that does not contain a subset of the form $\{(i,j),(i+1,j+1),(i+2,j+2)\}$. A \emph{shifted perforated $\mathbf{a}$-tableau} in $\lambda/\mu$ is a filling of some of the boxes of $S(\lambda/\mu)$ with letters $a$, $a'$ $\in [n]'$ such that no $a'$-boxes are south-east to any $a$-boxes, there is at most one $a$ per column and one $a'$ per row, and the main diagonal has at most one $\mathbf{a}$.
\end{defin}

The \emph{shape} of $A$, a perforated $\mathbf{a}$-tableau in a double border strip $S(\lambda/\mu)$, consists of the $\mathbf{a}$-filled boxes of $S(\lambda/\mu)$, and is denoted by $sh(A)$. Given a perforated $\mathbf{a}$-tableau $A$ and a perforated $\mathbf{b}$-tableau $B$, the pair $(A,B)$ is said to be a \emph{shifted perforated $(\mathbf{a},\mathbf{b})$-pair} of shape $\lambda/\mu$ if $S(\lambda/\mu)$ is the disjoint union of $sh(A)$ and $sh(B)$. In this case, we denote by $A\sqcup B$ the filling obtained by overlapping $A$ and $B$. 

\begin{ex}
The following are shifted perforated $\mathbf{1}$- and $\mathbf{2}$-tableaux, that form a shifted perforated $(\mathbf{1,2})$-pair of shape $(6,4,3)/(3,1)$:
$$A = \begin{ytableau}
{} & {} & {} & *(lightgray) 1' & *(lightgray) 1 & {} \\
\none & {} & *(lightgray) 1' & {} & {}\\
\none & \none & *(lightgray) 1 & {} & *(lightgray) 1
\end{ytableau}
\qquad
B = \begin{ytableau}
{} & {} & {} & {} & {} & 2' \\
\none & {} & {} & 2' & 2 \\
\none & \none & {} & 2 & {}
\end{ytableau}
\qquad
A \sqcup B = \begin{ytableau}
{} & {} & {} & *(lightgray) 1' & *(lightgray) 1 & 2' \\
\none & {} & *(lightgray) 1' & 2' & 2\\
\none & \none & *(lightgray) 1 & 2 & *(lightgray) 1
\end{ytableau}.
$$
\end{ex}

If $(A,B)$ is a shifted perforated $(\mathbf{a,b})$-pair, one can interchange an $\mathbf{a}$-box with a $\mathbf{b}$-box in $A \sqcup B$ subject to the following moves, called \emph{(shifted) switches}, illustrated in Figure \ref{fig:switches}.

\begin{figure}[h]
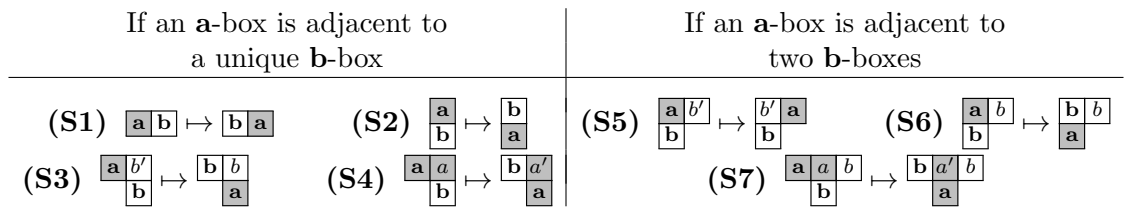

\begin{center}
\begin{tabular}{c|c}
If an $\mathbf{a}$-box is adjacent to & If an $\mathbf{a}$-box is adjacent to\\
a unique $\mathbf{b}$-box  &  two $\mathbf{b}$-boxes  \\
\hline
\\[-0.5em]
\textbf{(S1)}\; $
	\ytableausetup{smalltableaux}	
	\begin{ytableau}
    *(lightgray)\mathbf{a} & \mathbf{b}
  \end{ytableau} \mapsto
  \begin{ytableau}
  \mathbf{b} & *(lightgray)\mathbf{a}  \end{ytableau}$
  \hspace{2em}
  \textbf{(S2)}\; $\begin{ytableau}
    *(lightgray)\mathbf{a}\\
    \mathbf{b}
  \end{ytableau} \mapsto
  \begin{ytableau}
  \mathbf{b}\\
  *(lightgray)\mathbf{a}  \end{ytableau}$ & \textbf{(S5)}\; $\begin{ytableau}
    *(lightgray)\mathbf{a} & b'\\
    \mathbf{b}
  \end{ytableau} \mapsto
  \begin{ytableau}
  b' & *(lightgray)\mathbf{a} \\ \mathbf{b}  \end{ytableau}$ 
  \hspace{2em}
  \textbf{(S6)}\; $\begin{ytableau}
    *(lightgray)\mathbf{a} & b\\
    \mathbf{b}
  \end{ytableau} \mapsto
  \begin{ytableau}
  \mathbf{b} &  b \\*(lightgray)\mathbf{a}  \end{ytableau}$\\
  \\[-1em]

\textbf{(S3)}\; $\begin{ytableau}
    *(lightgray)\mathbf{a} & {b'}\\
    \none & \mathbf{b}
  \end{ytableau} \mapsto
  \begin{ytableau}
  \mathbf{b} &  b \\ \none & *(lightgray)\mathbf{a}  \end{ytableau}$
 \hspace{2em}
  \textbf{(S4)} \;$\begin{ytableau}
    *(lightgray)\mathbf{a} & *(lightgray)a\\
    \none & \mathbf{b}
  \end{ytableau} \mapsto
  \begin{ytableau}
  \mathbf{b} & *(lightgray){a'} \\
  \none & *(lightgray)\mathbf{a}  \end{ytableau}$ &
    \textbf{(S7)} \;$\begin{ytableau}
    *(lightgray)\mathbf{a} & *(lightgray){a} & b\\
    \none & \mathbf{b}
  \end{ytableau} \mapsto
  \begin{ytableau}
  \mathbf{b} & *(lightgray){a'} & b\\
  \none & *(lightgray)\mathbf{a}  \end{ytableau}$ 
\\
\end{tabular}
\end{center}
\caption{The shifted switches \cite[Section 3]{CNO17}.}
\label{fig:switches}
\end{figure}

Note that these switches correspond to the shifted \textit{jeu de taquin}, regarding the $\mathbf{a}$-boxes as empty inner corners. The switches \textbf{(S3)}, \textbf{(S4)}, \textbf{(S7)} are called the \emph{diagonal switches} and can only be performed when $\mathbf{a}$ and $\mathbf{b}$ are in the main diagonal. A $\mathbf{a}$-box is said to be \emph{fully switched} if it can't be switched with any $\mathbf{b}$-boxes, and that $A \sqcup B$ if \emph{fully switched} if every $\mathbf{a}$-box is fully switched.

\begin{obs}\label{rmk:sw_jdt}
With the exception of \textbf{(S4)} and \textbf{(S7)}, the shifted switches in Figure \ref{fig:switches} correspond to shifted \textit{jeu de taquin} moves, regarding the $\mathbf{a}$-boxes as empty corners.
\end{obs}

\begin{defin}[Shifted switching process \cite{CNO17}]
Let $T = A \sqcup B$ be a perforated $\mathbf{(a,b)}$-pair and suppose that $T := A \sqcup B$ is not fully switched. The \emph{shifted switching process} from $T$ to $\varsigma^m(T)$, with $m$ the least integer such that $\varsigma^m(T)$ is fully switched, is obtained as follows: choose the rightmost $a$-box in $A$ that is a neighbour to the north or west of a $\mathbf{b}$-box, if it exists, otherwise, choose the bottommost $a'$-box in the same conditions, and then apply the adequate switch among \textbf{(S1)}-\textbf{(S7)}, obtaining $\varsigma(T)$. The process is repeated until $\varsigma^m (T)$ is fully switched, where $\varsigma^i (T) := \varsigma (\varsigma^{i-1} (T))$, for $i\geq 2$. We then set $\mathsf{SP}_1(A,B) := \varsigma^m (T)^b$ and $\mathsf{SP}_2(A,B) := \varsigma^m (T)^a$, the tableaux obtained from $\varsigma^m (T)$ considering only the letters $\{b',b\}$ and $\{a',b\}$ respectively, and define
$$\mathsf{SP}(A,B) := (\mathsf{SP}_1 (A,B), \mathsf{SP}_2 (A,B)).$$
This process is depicted by the algorithm in Figure \ref{fig:SP}.
\end{defin}

\begin{figure}[h]
\includegraphics[scale=1]{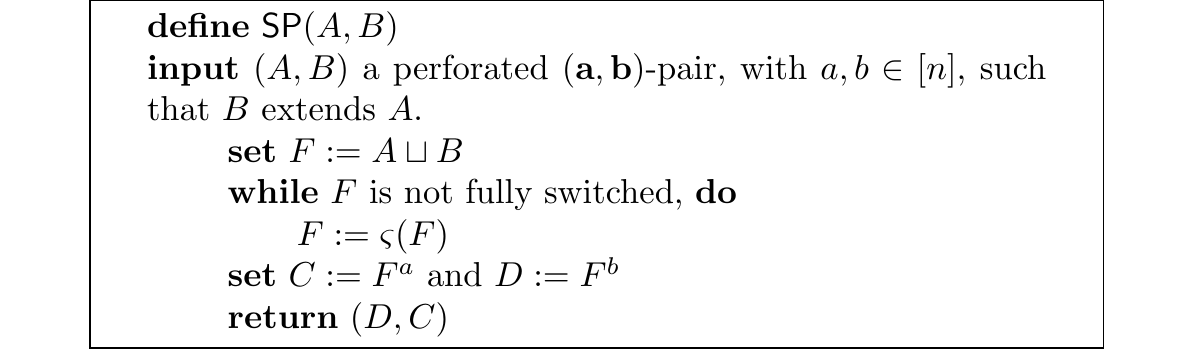}
\caption{Shifted tableau switching for shifted perforated $(\mathbf{a,b})$-pairs \cite[Algorithm 1]{CNO17}.}
\label{fig:SP}
\end{figure}

\begin{ex}
Consider the shifted perforated $(\mathbf{1,2})$-pair of the previous example, which is not fully switched.
$$A \sqcup B = \begin{ytableau}
{} & {} & {} & *(lightgray) 1' & *(lightgray) 1 & 2' \\
\none & {} & *(lightgray) 1' & 2' & 2\\
\none & \none & *(lightgray) 1 & 2 & *(lightgray) 1
\end{ytableau}.$$
The leftmost box filled with 1 (unprimed) is in position $(1,5)$, and it is adjacent to two $\mathbf{2}$-boxes. Hence, we apply the $\mathbf{(S5)}$ switch, obtaining:

$$A \sqcup B = \begin{ytableau}
{} & {} & {} & *(lightgray) 1' & *(lightgray) 1 & 2' \\
\none & {} & *(lightgray) 1' & 2' & 2\\
\none & \none & *(lightgray) 1 & 2 & *(lightgray) 1
\end{ytableau} \xrightarrow{\mathbf{(S5)}} \begin{ytableau}
{} & {} & {} & *(lightgray) 1' & 2' & *(lightgray) 1 \\
\none & {} & *(lightgray) 1' & 2' & 2\\
\none & \none & *(lightgray) 1 & 2 & *(lightgray) 1
\end{ytableau} = \varsigma(A \sqcup B).$$
\end{ex}

This $1$-box is now fully switched. Continuing the shifted switching process, until all $\mathbf{1}$-boxes are fully switched, we obtain:

\begin{align*}
\varsigma(A \sqcup B) = \begin{ytableau}
{} & {} & {} & *(lightgray) 1' & 2' & *(lightgray) 1 \\
\none & {} & *(lightgray) 1' & 2' & 2\\
\none & \none & *(lightgray) 1 & 2 & *(lightgray) 1
\end{ytableau}
&\xrightarrow{\mathbf{(S1)}}
\begin{ytableau}
{} & {} & {} & *(lightgray) 1' & 2' & *(lightgray) 1 \\
\none & {} & *(lightgray) 1' & 2' & 2\\
\none & \none & 2& *(lightgray) 1 & *(lightgray) 1
\end{ytableau}
\xrightarrow{\mathbf{(S5)}}
\begin{ytableau}
{} & {} & {} & *(lightgray) 1' & 2' & *(lightgray) 1 \\
\none & {} & 2' & *(lightgray) 1' & 2\\
\none & \none & 2& *(lightgray) 1 & *(lightgray) 1
\end{ytableau}\\
&\xrightarrow{\mathbf{(S1)}}
\begin{ytableau}
{} & {} & {} & *(lightgray) 1' & 2' & *(lightgray) 1 \\
\none & {} & 2' & 2 & *(lightgray) 1' \\
\none & \none & 2& *(lightgray) 1 & *(lightgray) 1
\end{ytableau}
\xrightarrow{\mathbf{(S5)}}
\begin{ytableau}
{} & {} & {} & 2' & *(lightgray) 1' & *(lightgray) 1 \\
\none & {} & 2' & 2 & *(lightgray) 1' \\
\none & \none & 2& *(lightgray) 1 & *(lightgray) 1
\end{ytableau} = \varsigma^5 (A \sqcup B).
\end{align*}

\begin{obs}
Unlike the tableau switching for Young tableaux \cite{BSS96}, the shifted version depends on the order in which the $\mathbf{a}$-boxes are chosen \cite[Remark 3.7 (i)]{CNO17arxiv}. For instance, if one applies \textbf{(S6)} (corresponding to choose the box with $2'$) instead of \textbf{(S1)} (corresponding to the box with $1$, i.e., the rightmost $1$-box), the obtained filling is not a valid $(\mathbf{1},\mathbf{2})$-pair, as the second row is not weakly increasing:
$$\begin{ytableau}
{} & {} & *(lblue)1' & 2\\
\none & *(lblue)1 & 2
\end{ytableau} \overset{\textbf{(S6)}}\longrightarrow
\begin{ytableau}
{} & {} & 2 & 2\\
\none & *(lblue)1 & *(lblue)1'
\end{ytableau}.$$
\end{obs}

This process is well defined and it is an involution 
\cite[Theorem 3.5]{CNO17}. It may be extended to pairs of shifted semistandard tableaux $(S,T)$, with $T$ extending $S$.
The result is denoted by $\mathsf{SW}(S,T) := (\mathsf{SW}_1(S,T), \mathsf{SW}_2(S,T))$, where $\mathsf{SW}_1 (S,T) = T'$ and $\mathsf{SW}_2(S,T)=S'$ as depicted in Figure \ref{fig:SW}. 

The shifted tableau switching $\mathsf{SW}$ for pairs of shifted semistandard tableaux is also well defined \cite[Theorem 3.6]{CNO17} and it is an involution 
\cite[Theorem 4.3]{CNO17}. If $S$ is straight-shaped, then $\mathsf{SW}_1(S,T) = \mathsf{rect}(T)$. Similar to the type $A$ case \cite{Az18,BSS96}, if $T$ is a LRS tableau, then so it is $\mathsf{SW}_2(S,T)$, for any straight-shaped shifted $S$ extended by $T$ \cite[Theorem 4.3]{CNO17}. Thus, considering $S := Y_{\mu}$, we have a bijection that sends $T$, a LRS tableau of shape $\lambda/\mu$ and weight $\nu$, to $\mathsf{SW}_2 (Y_{\mu},T)$, a LRS tableau of shape $\lambda/\nu$ and weight $\nu$, giving the symmetry $f^{\lambda}_{\mu\nu} = f^{\lambda}_{\nu\mu}$.

\begin{figure}[h]
\includegraphics[scale=1]{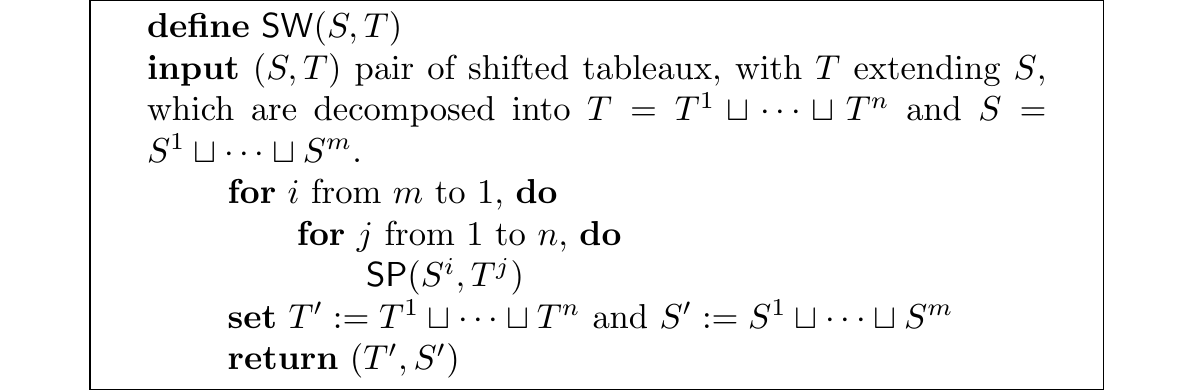}
\caption{Shifted tableau switching for pairs of shifted tableaux \cite[Algorithm 2]{CNO17}.}
\label{fig:SW}
\end{figure}

This shifted tableau switching is compatible with standardization \cite[Remark 3.8]{CNO17}, i.e., 
\begin{equation}\label{eq:sw_std}
\begin{split}
\mathsf{SW} (\mathsf{std}(S),T) &= (id \times \mathsf{std}) \circ \mathsf{SW}(S,T)\\
\mathsf{SW} (S, \mathsf{std}(T)) &= (\mathsf{std} \times id) \circ \mathsf{SW}(S,T)
\end{split}
\end{equation}
where $id \times \mathsf{std}$ denotes the usual Cartesian product of maps, i.e., $(id \times \mathsf{std})(S,T) = (S, \mathsf{std}(T))$. Moreover, since the switches may be regarded as \textit{jeu de taquin} slides, the pair $\mathsf{SW}(S,T)$ is component-wise shifted Knuth equivalent to $(T,S)$, for any pair of shifted semistandard tableaux $(S,T)$, with $T$ extending $S$. Moreover, rewriting \cite[Corollaries 2.8 and 2.9]{Haim92} in terms of the shifted tableau switching yields the following.

\begin{prop}[{\cite[Proposition 3.2]{CNOrev}}]\label{prop:sw_dual}
Let $S$ and $T$ be shifted semistandard tableaux in the same dual equivalence class (in particular, $S$ and $T$ have the same shape). Let $W$ be a semistandard shifted tableau. Then,
	\begin{enumerate}
	\item If $S$ and $T$ extend $W$, then $\mathsf{SW}_2(W,S) = \mathsf{SW}_2(W,T)$, and $\mathsf{SW}_1(W,S)$ is shifted dual equivalent to $\mathsf{SW}_1(W,T)$.
	\item If $W$ extends $S$ and $T$, then $\mathsf{SW}_1(S,W) = \mathsf{SW}_1(T,W)$, and $\mathsf{SW}_2(S,W)$ is shifted dual equivalent to $\mathsf{SW}_2(T,W)$.
	\end{enumerate}
\end{prop}
 
\begin{cor}\label{cor:sw_rev_com}
Let $S$ and $T$ be shifted semistandard tableaux such that $T$ extends $S$. Then,
	\begin{enumerate}
	\item $(\mathsf{SW}_1 (S,T))^e = \mathsf{SW}_1 (S,T^e)$.
	\item $(\mathsf{SW}_2 (S,T))^e = \mathsf{SW}_2 (S^e,T)$.
	\end{enumerate}
\end{cor}

\begin{proof}
By definition, $T$ is shifted dual equivalent to $T^e$, hence, by Proposition \ref{prop:sw_dual}, we have that $\mathsf{SW}_1 (S,T)$ is shifted dual equivalent to $\mathsf{SW}_1 (S,T^e)$. The shifted tableau switching algorithm ensures that $\mathsf{SW}_1 (S,T^e)$ is shifted Knuth equivalent to $T^e$, and since the operator $\mathsf{c}_n$ is coplactic, we have
$$\mathsf{SW}_1 (S,T^e) \equiv_k T^e \equiv_k \mathsf{c}_n (T) \equiv_k \mathsf{c}_n (\mathsf{SW}_1 (S,T)).$$
Since $\mathsf{SW}_1 (S,T^e)$ is shifted dual equivalent to $\mathsf{SW}_1 (S,T)$ and shifted dual equivalent to $\mathsf{c}_n (\mathsf{SW}_1 (S,T))$, we have that $(\mathsf{SW}_1 (S,T))^e = \mathsf{SW}_2 (S,T^e)$. The proof for the second statement is similar.
\end{proof}

The following result ensures that the shifted tableau switching is also compatible with canonical form.
\begin{prop}
Let $S$, $S'$, $T$ and $T'$ be shifted semistandard tableaux filled in $[n]'$, not necessarily in canonical form, such that $T$ extends $S$ and $T'$ extends $S'$. Suppose that $S$ and $S'$ have the same canonical form, and so do $T$ and $T'$.
Then,
\begin{enumerate}
\item $\mathsf{SW}_1(S,T)$, $\mathsf{SW}_1(S',T)$, $\mathsf{SW}_1(S,T')$ and $\mathsf{SW}_1(S',T')$ have the same canonical form.
\item $\mathsf{SW}_2(S,T)$, $\mathsf{SW}_2(S',T)$, $\mathsf{SW}_2(S,T')$ and $\mathsf{SW}_2(S',T')$ have the same canonical form.
\end{enumerate}
\end{prop}

\begin{proof}
It suffices to show that for each $i \in [n]$, the southwesternmost occurrence of $\mathbf{i}$ maintains its relative position. This is verified for each switch \textbf{(S1)}-\textbf{(S7)}. Moreover, the switching algorithm states that one must start with the rightmost unprimed $i$ that has neighbours to its south or east, and then proceeding to the lowest primed $i'$. Hence, the switching path is going from right to left, and then from bottom to top. Therefore, the lowest and leftmost $\mathbf{i}$ is either the last unprimed $i$ or the first primed $i'$, leaving the switching order unchanged.
\end{proof}

\begin{ex}
Consider the following pair of shifted semistandard tableau $(S,T)$, with $T$ (in gray background) extending $S$ (in white background):
$$(S,T) = 
\begin{ytableau}
1 & 1 & 2' & *(lblue) 1 & *(lblue) 2'\\
\none & 2 & *(lblue) 1 & *(lblue)2\\
\none & \none & *(lblue)2 & *(lblue)3
\end{ytableau}$$
To apply the shifted tableau switching $\mathsf{SW}$ to $(S,T)$, we first compute $\mathsf{SP}(S^2, T^1)$:

$$(S,T) = \begin{ytableau}
1 & 1 & 2' & *(lblue) 1 & *(lblue) 2'\\
\none & 2 & *(lblue) 1 & *(lblue)2\\
\none & \none & *(lblue)2 & *(lblue)3
\end{ytableau} \xrightarrow{\textbf{(S1)}}
\begin{ytableau}
1 & 1 & 2' & *(lblue) 1 & *(lblue) 2'\\
\none & *(lblue) 1 & 2 & *(lblue)2\\
\none & \none & *(lblue)2 & *(lblue)3
\end{ytableau} \xrightarrow{\textbf{(S1)}}
\begin{ytableau}
1 & 1 & *(lblue) 1 & 2' & *(lblue) 2'\\
\none & *(lblue) 1 & 2 & *(lblue)2\\
\none & \none & *(lblue)2 & *(lblue)3
\end{ytableau}.
$$
Then, we compute $\mathsf{SP}(S^2,T^2)$:
$$\begin{ytableau}
1 & 1 & *(lblue) 1 & *(lblue) 2' & 2'\\
\none & *(lblue) 1 & 2 & *(lblue)2\\
\none & \none & *(lblue)2 & *(lblue)3
\end{ytableau} \xrightarrow{\textbf{(S6)}}
\begin{ytableau}
1 & 1 & *(lblue) 1 & *(lblue) 2' & 2'\\
\none & *(lblue) 1 & *(lblue)2  & *(lblue)2\\
\none & \none & 2 & *(lblue)3
\end{ytableau} \xrightarrow{\textbf{(S5)}}
\begin{ytableau}
1 & 1 & *(lblue) 1 & *(lblue) 2' & 2'\\
\none & *(lblue) 1 & *(lblue)2  & *(lblue)2\\
\none & \none & 2 & *(lblue)3
\end{ytableau}
$$
Continuing the process, we have:
\begin{align*}
\begin{ytableau}
1 & 1 & *(lblue) 1 & *(lblue) 2' & 2'\\
\none & *(lblue) 1 & *(lblue)2  & *(lblue)2\\
\none & \none & 2 & *(lblue)3
\end{ytableau} & \xrightarrow{\textbf{(S1)}} 
\begin{ytableau}
1 & 1 & *(lblue) 1 & *(lblue) 2' & 2'\\
\none & *(lblue) 1 & *(lblue)2  & *(lblue)2\\
\none & \none & *(lblue)3 & 2
\end{ytableau} \xrightarrow{\textbf{(S7)}} 
\begin{ytableau}
*(lblue)1 & 1' & *(lblue) 1 & *(lblue) 2' & 2'\\
\none & 1 & *(lblue)2  & *(lblue)2\\
\none & \none & *(lblue)3 & 2
\end{ytableau} \xrightarrow{\textbf{(S1)}} 
\begin{ytableau}
*(lblue)1 & *(lblue) 1 & 1' & *(lblue) 2' & 2'\\
\none & 1 & *(lblue)2  & *(lblue)2\\
\none & \none & *(lblue)3 & 2
\end{ytableau} \xrightarrow{\textbf{(S1)}}
\begin{ytableau}
*(lblue)1 & *(lblue) 1 & 1' & *(lblue) 2' & 2'\\
\none & *(lblue)2 & 1 & *(lblue)2\\
\none & \none & *(lblue)3 & 2
\end{ytableau}\\ 
&\xrightarrow{\textbf{(S1)}}
\begin{ytableau}
*(lblue)1 & *(lblue) 1 & 1' & *(lblue) 2' & 2'\\
\none & *(lblue)2 & *(lblue)2 & 1\\
\none & \none & *(lblue)3 & 2
\end{ytableau} \xrightarrow{\textbf{(S5)}}
\begin{ytableau}
*(lblue)1 & *(lblue) 1 & *(lblue) 2' & 1' & 2'\\
\none & *(lblue)2 & *(lblue)2 & 1\\
\none & \none & *(lblue)3 & 2
\end{ytableau} = \mathsf{SW}(S,T).
\end{align*}
\end{ex}

The algorithm to compute the reversal of a shifted tableau $T \in \mathsf{ShST}(\lambda/\mu,n)$ may be described using the shifted tableau switching \cite{CNOrev}. 

\begin{prop}[{\cite[Definition 4.5]{CNOrev}}]\label{prop:rev_sw}
Let $T \in \mathsf{ShST}(\lambda/\mu,n)$ and let $U$ and $W$ be shifted standard tableaux of shape $\mu$. Let $W' :=  \mathsf{SW}_2(W,T)$ and $U' :=  \mathsf{SW}_2(U,T)$. Then,
$$\mathsf{SW}(\mathsf{evac}(\mathsf{rect}(T)), W') = \mathsf{SW}(\mathsf{evac}(\mathsf{rect}(T)), U'),$$
and we have
$$\mathsf{SW}(\mathsf{evac}(\mathsf{rect}(T)), W') = (W, T^e).$$ 

\end{prop}

\begin{proof}
Since they have the same straight shape, $\mathsf{rect}(T)$ is shifted dual equivalent to $\mathsf{evac} (\mathsf{rect}(T))$. Thus, by Proposition \ref{prop:sw_dual}, $\mathsf{SW}_2(\mathsf{evac}(\mathsf{rect}(T)), U')$ is dual equivalent to 
$$\mathsf{SW}_2(\mathsf{rect}(T), U') = \mathsf{SW}_2 \big(\mathsf{SW}_1(U,T),\mathsf{SW}_2(U,T)\big) = \mathsf{SW}_2 (\mathsf{SW}(U,T)) = T.$$ 
Furthermore, since $\mathsf{SW}(\mathsf{evac}(\mathsf{rect}(T)),U')$ is component-wise shifted Knuth equivalent to $(U', \mathsf{evac}(\mathsf{rect} (T)))$, we have
$$\mathsf{SW}_2(\mathsf{evac}(\mathsf{rect}(T)), U') \equiv_k \mathsf{evac} (\mathsf{rect}(T)) \equiv_k \mathsf{c}_n (\mathsf{rect}(T)) \equiv_k \mathsf{c}_n (T).$$
The result then follows from the uniqueness of Theorem \ref{teo:haiman}.
\end{proof}

\begin{ex}
To illustrate this procedure, we use the same tableau in Example \ref{ex:reversal}, filling the inner shape $\mu$ with a standard tableau $U$. We note that, since $U = U^1 \sqcup \cdots \sqcup U^{|\mu|}$ is standard, then each $U^i$ consists of a single box filled with (unprimed) $i$. Thus, the switches \textbf{(S4)} and \textbf{(S7)} will not be used during the shifted tableau switching process.
$$\begin{ytableau}
{} & {} & {} & 1' & 1\\
\none & {} & 1 & 1\\
\none & \none & 2 & 2\\
\none & \none & \none & 3
\end{ytableau}
\longrightarrow
\begin{ytableau}
*(lblue)1 & *(lblue)2 & *(lblue)3 & 1' & 1\\
\none & *(lblue)2 & 1 & 1\\
\none & \none & 2 & 2\\
\none & \none & \none & 3
\end{ytableau}
\overset{\mathsf{SW}}\longrightarrow
\begin{ytableau}
1 & 1 & 1 & 1 & *(lblue)3\\
\none & 2 & 2 & *(lblue)1\\
\none & \none & 3 & *(lblue)2\\
\none & \none & \none & *(lblue)4
\end{ytableau}
\overset{\mathsf{evac} \times id}\longrightarrow
\begin{ytableau}
1 & 2' & 3' & 3 & *(lblue)3\\
\none & 2 & 3' & *(lblue)1\\
\none & \none & 3 & *(lblue)2\\
\none & \none & \none & *(lblue)4
\end{ytableau}
\overset{\mathsf{SW}}\longrightarrow
\begin{ytableau}
*(lblue)1 & *(lblue)2 & *(lblue)3 & 2' & 3'\\
\none & *(lblue)4 & 1 & 3'\\
\none & \none & 2 & 3'\\
\none & \none & \none & 3
\end{ytableau}= (U, T^e).
$$
\end{ex}

As remarked before, the shifted tableau switching process could be obtained by first standardizing the involved tableaux, apply the \emph{type $C$ infusion involution} \cite{TY09}, and then the \emph{shifted semistandardization} process \cite{PY17}. This is due to the shifted tableau switching being compatible with standardization and the fact that, on shifted standard tableaux, the order in which the shifted are performed (see Figure \ref{fig:SW}) agrees with the one determined by the type $C$ infusion map (see Lemma \ref{lem:sw_inf_std} below).

\begin{defin}[\cite{TY09}]\label{def:infusionC}
Let $(S,T)$ be a pair of shifted standard tableaux, of shapes $\mu/\nu$ and $\lambda/\mu$ (thus, $T$ extends $S$), respectively. The \emph{type $C$ infusion} of the pair $(S,T)$, denoted by $\mathsf{infusion}(S,T) := (\mathsf{infusion}_1 (S,T), \mathsf{infusion}_2 (S,T))$, is the pair $(X,Y)$ of standard tableaux of shapes $\gamma/\nu$ and $\lambda/\gamma$, for some strict partition $\gamma$ with $|\gamma|=|\lambda|-|\mu|$, obtained in the following way: 
	\begin{enumerate}
	\item Let $m$ be the largest entry of $S$. Then, its box is a inner corner for $\lambda/\mu$, and we perform \textit{jeu de taquin} on $T$ starting with that inner corner, until an outer corner is obtained. Place $m$ on that outer corner and never move it again for the duration of the process.
	\item Repeat the last step for the remaining entries of $S$, going from the largest to the smallest.
	\item Then, $X$ is tableau obtained after performing all the shifted \textit{jeu de taquin} slides on $T$ determined by the entries of $S$, and $Y$ is the tableau obtained by placing the entries of $S$ on the resulting outer corners.
	\end{enumerate}
\end{defin}
The shifted tableau $\mathsf{infusion}_1 (S,T)$ is then the result of applying shifted \textit{jeu de taquin} inner slides to $T$ (determined by $S$), and $\mathsf{infusion}_2 (S,T)$ encodes the order in which those slides were performed. In particular, if $S$ has straight-shape, then $\mathsf{infusion}_1 (S,T) = \mathsf{rect}(T)$. 

If $(S,T)$ is a pair of shifted standard tableaux, then there are no repeated entries, nor primed ones, thus the algorithm to compute $\mathsf{SW}(S,T)$ requires only the switches $\mathbf{(S1)}$ and $\mathbf{(S2)}$ of Figure \ref{fig:switches}. This switches correspond to shifted \textit{jeu de taquin} slides in a shifted standard tableau, as the exceptional slide (see Definition \ref{def:sh_jdt}) cannot occur. Moreover, the algorithm for the shifted tableau switching in Figure \ref{fig:SW} states that the shifted switches must be performed from the largest entry of $S$ to the smallest, which agrees with the order defined by the type $C$ infusion (Definition \ref{def:infusionC}). Thus, we have the following.

\begin{lema}\label{lem:sw_inf_std}
Let $(S,T)$ be a pair of shifted standard tableaux, with $T$ extending $S$. Then, $\mathsf{SW}(S,T) = \mathsf{infusion} (S,T)$.
\end{lema}

\begin{ex}\label{ex:infu}
Consider the following pair of shifted standard tableaux
$$(S,T) = \begin{ytableau}
1 & 2 & 3 & *(lblue) 2 & *(lblue) 3\\
\none & 4 & *(lblue) 1 & *(lblue)5\\
\none & \none & *(lblue)4 & *(lblue)6
\end{ytableau}.$$
To compute $\mathsf{infusion}(S,T)$ we start with the largest entry of $S$, and regarding its box as inner corner, perform \textit{jeu de taquin} slides:
$$\begin{ytableau}
1 & 2 & 3 & *(lblue) 2 & *(lblue) 3\\
\none & 4 & *(lblue) 1 & *(lblue)5\\
\none & \none & *(lblue)4 & *(lblue)6
\end{ytableau} \longrightarrow
\begin{ytableau}
1 & 2 & 3 & *(lblue) 2 & *(lblue) 3\\
\none & *(lblue) 1 & 4 & *(lblue)5\\
\none & \none & *(lblue)4 & *(lblue)6
\end{ytableau} \longrightarrow
\begin{ytableau}
1 & 2 & 3 & *(lblue) 2 & *(lblue) 3\\
\none & *(lblue) 1 & *(lblue)4 & *(lblue)5\\
\none & \none & 4 & *(lblue)6
\end{ytableau} \longrightarrow
\begin{ytableau}
1 & 2 & 3 & *(lblue) 2 & *(lblue) 3\\
\none & *(lblue) 1 & *(lblue)4 & *(lblue)5\\
\none & \none & *(lblue)6 & 4 
\end{ytableau}$$

Continuing with the next largest entries of $S$, we obtain:
\begin{align*}
\begin{ytableau}
1 & 2 & 3 & *(lblue) 2 & *(lblue) 3\\
\none & *(lblue) 1 & *(lblue)4 & *(lblue)5\\
\none & \none & *(lblue)6 & 4 
\end{ytableau} &\longrightarrow 
\begin{ytableau}
1 & 2 & *(lblue) 2 & 3 & *(lblue) 3\\
\none & *(lblue) 1 & *(lblue)4 & *(lblue)5\\
\none & \none & *(lblue)6 & 4 
\end{ytableau} \longrightarrow
\begin{ytableau}
1 & 2 & *(lblue) 2 & *(lblue)3 & 3\\
\none & *(lblue) 1 & *(lblue)4 & *(lblue)5\\
\none & \none & *(lblue)6 & 4 
\end{ytableau} \longrightarrow
\begin{ytableau}
1 & 2 & *(lblue) 2 & *(lblue)3 &  3\\
\none & *(lblue) 1 & *(lblue)4 & *(lblue)5\\
\none & \none & *(lblue)6 & 4 
\end{ytableau} \longrightarrow
\begin{ytableau}
1 & *(lblue)1 & *(lblue) 2 & *(lblue)3 & 3\\
\none &  2 & *(lblue)4 & *(lblue)5\\
\none & \none & *(lblue)6 & 4 
\end{ytableau}\\
&\longrightarrow
\begin{ytableau}
1 & *(lblue)1 & *(lblue) 2 & *(lblue)3 & 3\\
\none & *(lblue)4 & 2 & *(lblue)5\\
\none & \none & *(lblue)6 & 4 
\end{ytableau} \longrightarrow
\begin{ytableau}
1 & *(lblue)1 & *(lblue) 2 & *(lblue)3 & 3\\
\none & *(lblue)4 & *(lblue)5 & 2\\
\none & \none & *(lblue)6 & 4 
\end{ytableau} \longrightarrow
\begin{ytableau}
1 & *(lblue)1 & *(lblue) 2 & *(lblue)3 & 3\\
\none & *(lblue)4 & *(lblue)5 & 2\\
\none & \none & *(lblue)6 & 4 
\end{ytableau} \longrightarrow 
\begin{ytableau}
*(lblue)1 & 1 & *(lblue) 2 & *(lblue)3 &  3\\
\none & *(lblue)4 & *(lblue)5 & 2\\
\none & \none & *(lblue)6 & 4 
\end{ytableau}\\ 
&\longrightarrow
\begin{ytableau}
*(lblue)1 & *(lblue)2 & 1 & *(lblue)3 & 3\\
\none & *(lblue)4 & *(lblue)5 & 2\\
\none & \none & *(lblue)6 & 4 
\end{ytableau} \longrightarrow
\begin{ytableau}
*(lblue)1 & *(lblue)2 & *(lblue) 3 & 1 & 3\\
\none & *(lblue)4 & *(lblue)5 & 2\\
\none & \none & *(lblue)6 &4 
\end{ytableau} = \mathsf{infusion}(S,T).
\end{align*}
\end{ex}

\begin{prop}\label{prop:sw_infu}
Let $(S,T)$ be a pair of shifted semistandard tableaux, with $T$ extending $S$, and such that $\mathsf{wt}(T)=\nu_T$ and $\mathsf{wt}(S) = \nu_S$. Then,
$$\mathsf{SW}(S,T) = (\mathsf{sstd}_{\nu_T} \times \mathsf{sstd}_{\nu_S}) \circ \mathsf{infusion} (\mathsf{std}(S),\mathsf{std}(T)).$$
\end{prop}

\begin{proof}
Since $\mathsf{SW}_1(S,T)$ and $\mathsf{SW}_2(S,T)$ have weights $\nu_{T}$ and $\nu_{S}$, respectively, then by \cite[Lemma 9.5]{PY17} we have
$$(\mathsf{sstd}_{\nu_T} \times \mathsf{sstd}_{\nu_S}) \circ (\mathsf{std} \times \mathsf{std}) ( \mathsf{SW}(S,T)) = \mathsf{SW}(S,T).$$ 
By Lemma \ref{lem:sw_inf_std} and \eqref{eq:sw_std}, we have
\begin{align*}
(\mathsf{sstd}_{\nu_T} \times &\mathsf{sstd}_{\nu_S}) \circ \mathsf{infusion} (\mathsf{std}(S),\mathsf{std}(T)) =\\ 
&=(\mathsf{sstd}_{\nu_T} \times \mathsf{sstd}_{\nu_S}) \circ \mathsf{SW} (\mathsf{std}(S),\mathsf{std}(T))\\
&= (\mathsf{sstd}_{\nu_T} \times \mathsf{sstd}_{\nu_S}) \circ (id \times \mathsf{std}) \circ \mathsf{SW} (S,\mathsf{std}(T))\\
&=(\mathsf{sstd}_{\nu_T} \times \mathsf{sstd}_{\nu_S}) \circ (id \times \mathsf{std}) \circ (\mathsf{std} \times id) \circ \mathsf{SW} (S,T)\\
&= (\mathsf{sstd}_{\nu_T} \times \mathsf{sstd}_{\nu_S}) \circ (\mathsf{std} \times \mathsf{std}) \circ \mathsf{SW} (S,T)\\
&= \mathsf{SW}(S,T).
\end{align*}
\end{proof}

\begin{ex}\label{ex:inf_sstd}
We illustrate the process with the shifted tableau pair $(S,T)$ from a previous example:
$$(S,T) = 
\begin{ytableau}
1 & 1 & 2' & *(lblue) 1 & *(lblue) 2'\\
\none & 2 & *(lblue) 1 & *(lblue)2\\
\none & \none & *(lblue)2 & *(lblue)3
\end{ytableau} \longrightarrow 
\begin{ytableau}
1 & 2 & 3 & *(lblue) 2 & *(lblue) 3\\
\none & 4 & *(lblue) 1 & *(lblue)5\\
\none & \none & *(lblue)4 & *(lblue)6
\end{ytableau} = (\mathsf{std} (S), \mathsf{std}(T)).$$
From Example \ref{ex:infu}, we have
$$\begin{ytableau}
1 & 2 & 3 & *(lblue) 2 & *(lblue) 3\\
\none & 4 & *(lblue) 1 & *(lblue)5\\
\none & \none & *(lblue)4 & *(lblue)6
\end{ytableau} \overset{\mathsf{infusion}}\longrightarrow
\begin{ytableau}
*(lblue)1 & *(lblue)2 & *(lblue) 3 & 1 & 3\\
\none & *(lblue)4 & *(lblue)5 & 2\\
\none & \none & *(lblue)6 &4 
\end{ytableau}.$$
Since we have $\mathsf{wt} (T) = (2,3,1)$ and $\mathsf{wt}(S) = (2,2)$, we now apply the semistandardization process with respect to this compositions, respectively:

$$
\begin{ytableau}
*(lblue)1 & *(lblue) 2 & *(lblue) 3 & 1 & 3\\
\none & *(lblue)4 & *(lblue)5 & 2\\
\none & \none & *(lblue)6 & 4
\end{ytableau} \longrightarrow
\begin{ytableau}
*(lblue)1_1 & *(lblue) 1_2 & *(lblue) 2_3 & 1_1 & 2_3\\
\none & *(lblue)2_4 & *(lblue)2_5 & 1_2\\
\none & \none & *(lblue)3_6 & 2_4
\end{ytableau} \longrightarrow
\begin{ytableau}
*(lblue)1 & *(lblue) 1 & *(lblue) 2' & 1' & 2'\\
\none & *(lblue)2 & *(lblue)2 & 1\\
\none & \none & *(lblue)3 & 2
\end{ytableau} = \mathsf{SW}(S,T).$$
\end{ex}

The authors in \cite{CNO17} present another algorithm for tableaux of straight shape, that coincides with the shifted evacuation (Section \ref{ss:evac}), using the shifted tableau switching\footnote{The authors use the terminology \emph{shifted generalized evacuation} for this algorithm.}. We consider the auxiliary alphabet $-[n]' := \{-n' \!<\! -n \!<\! \cdots \!<\! -1' \!<\! -1\}$ and $-[n]' \sqcup [n]' := \{-n' \!<\! -n \!<\! \cdots \!<\! -1' \!<\! -1 \!<\! 1' \!<\! 1 \!<\! \cdots \!<\! n' \!<\! n\}$. Given $T \in \mathsf{ShST}(\lambda/\mu,n)$ and $k \in [n]$, we define $\mathsf{neg}_k (T)$ to be the tableau (filled in $-[n]' \sqcup [n]'$) obtained from $T$ by replacing each $k$ with $-k$ and each $k'$ with $-k'$, leaving the remaining letters unchanged. If $T$ is a shifted tableau filled in $-[n]'$, we define $\mathsf{d}_n (T)$ to be the tableau (filled in $[n]'$) obtained from $T$ by replacing each $-i$ with $\theta_{1,n} (i)$ and each $-i'$ with $\theta_{1,n}(i')$, that is,
\begin{equation}\label{eq:d_neg}
\mathsf{d}_n (T) = \theta_{1,n} \mathsf{neg}_1^{-1} \cdots \mathsf{neg}_n^{-1} (T).
\end{equation}
Consider the algorithm presented in Figure \ref{fig:SGE}, defined on the alphabet $-[n]' \sqcup [n]'$ (we note that the use of negative entries ensure that after those will not move again after being fully switched). This algorithm coincides with the shifted evacuation for straight-shaped tableaux \cite[Theorem 5.6]{CNO17}.

\begin{figure}[h!]
\includegraphics[scale=1]{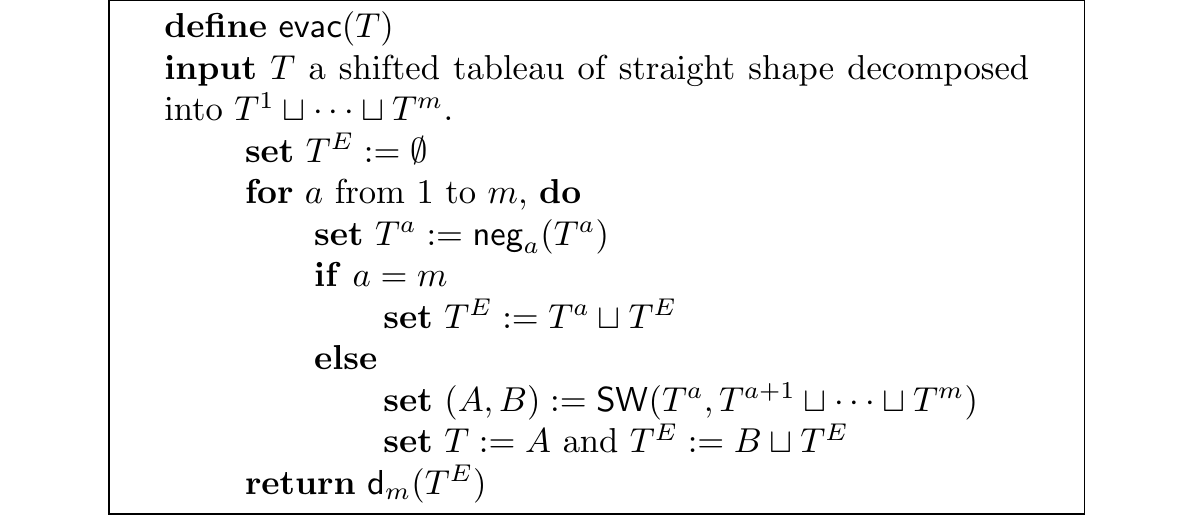}
\caption{The shifted evacuation algorithm \cite[Algorithm 4]{CNO17}.}
\label{fig:SGE}
\end{figure}

Given $T \in \mathsf{ShST}(\nu,n)$, the algorithm in Figure \ref{fig:SGE} may be easily modified to obtain a restriction $\mathsf{evac}_k$ to the alphabet $\{1, \ldots, k\}'$, for $k \leq n$, by applying $\mathsf{evac}$ to $T^1 \sqcup \cdots \sqcup T^k$ and maintaining $T^{k+1} \sqcup \cdots \sqcup T^n$ unchanged. This is depicted in Figure \ref{fig:SGE_k}. It is clear that $\mathsf{evac}_n = \mathsf{evac}$.

\begin{figure}[h]
\includegraphics[scale=1]{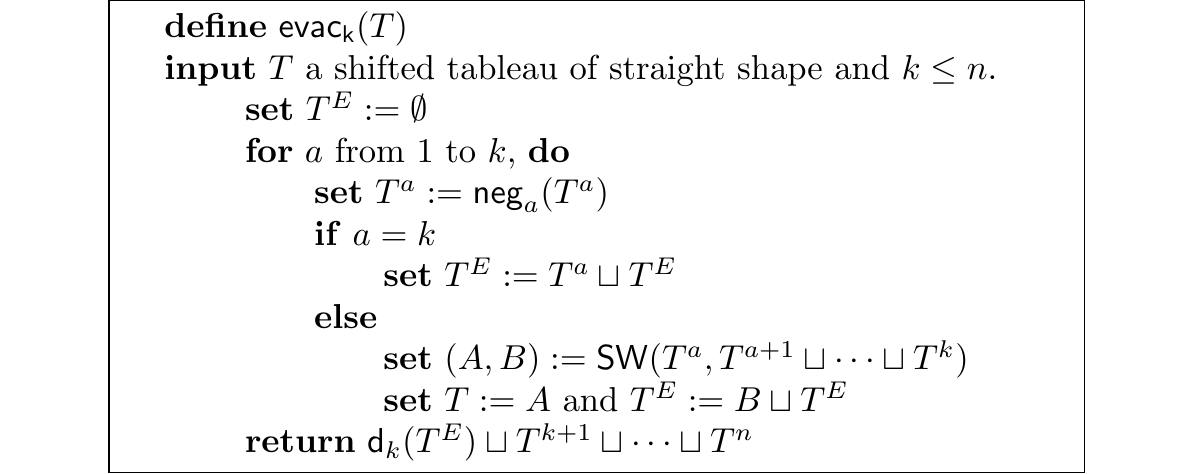}
\caption{The shifted evacuation algorithm, restricted to the letters $[1,k]'$.}
\label{fig:SGE_k}
\end{figure}

\begin{ex}
Let $T=\begin{ytableau}
1 & 1 & 2' & 2 & 3\\
\none & 2 & 2 & 3'\\
\none & \none & 3
\end{ytableau}$. Computing $\mathsf{evac}(T)$ with the Algorithm in Figure \ref{fig:SGE}, we have

\begin{align*}
T = \begin{ytableau}
1 & 1 & 2' & 2 & 3\\
\none & 2 & 2 & 3'\\
\none & \none & 3
\end{ytableau} &\xrightarrow{\mathsf{neg}_1} \begin{ytableau}
*(lgray)\shortminus 1 & *(lgray) \shortminus 1 & 2' & 2 & 3\\
\none & 2 & 2 & 3'\\
\none & \none & 3
\end{ytableau} \xrightarrow{\mathbf{(S5)}}
\begin{ytableau}
*(lgray) \shortminus 1 & 2' & *(lgray) \shortminus 1 & 2 & 3\\
\none & 2 & 2 & 3'\\
\none & \none & 3
\end{ytableau} \xrightarrow{\mathbf{(S6)}}
\begin{ytableau}
*(lgray) \shortminus 1 & 2' & 2 & 2 & 3\\
\none & 2 & *(lgray) \shortminus 1 & 3'\\
\none & \none & 3
\end{ytableau} \xrightarrow{\mathbf{(S3)}}
\begin{ytableau}
2 & 2 & 2 & 2 & 3\\
\none & *(lgray) \shortminus 1 & *(lgray) \shortminus 1 & 3'\\
\none & \none & 3
\end{ytableau}\\
&\xrightarrow{\mathbf{(S5)}}
\begin{ytableau}
2 & 2 & 2 & 2 & 3\\
\none & *(lgray) \shortminus 1 &  3' & *(lgray) \shortminus 1\\
\none & \none & 3
\end{ytableau} \xrightarrow{\mathbf{(S3)}}
\begin{ytableau}
2 & 2 & 2 & 2 & 3\\
\none & 3 &  3 & *(lgray) \shortminus 1\\
\none & \none & *(lgray) \shortminus 1
\end{ytableau} \xrightarrow{\mathsf{neg}_2}
\begin{ytableau}
*(lgray) \shortminus 2 & *(lgray) \shortminus 2 & *(lgray) \shortminus 2 & *(lgray) \shortminus 2 & 3\\
\none & 3 &  3 & \shortminus 1\\
\none & \none & \shortminus 1
\end{ytableau}
\xrightarrow{\mathbf{(S1)}}
\begin{ytableau}
*(lgray) \shortminus 2 & *(lgray) \shortminus 2 & *(lgray) \shortminus 2 & 3 & *(lgray) \shortminus 2\\
\none & 3 &  3 & \shortminus 1\\
\none & \none & \shortminus 1
\end{ytableau}\\
&\xrightarrow{\mathbf{(S6)}}
\begin{ytableau}
*(lgray) \shortminus 2 & *(lgray) \shortminus 2 & 3 & 3 & *(lgray) \shortminus 2\\
\none & 3 &  *(lgray) \shortminus 2 & \shortminus 1\\
\none & \none & \shortminus 1
\end{ytableau}
\xrightarrow{\mathbf{(S7)}}
\begin{ytableau}
3 & *(lgray) \shortminus 2' & 3 & 3 & *(lgray) \shortminus 2\\
\none & *(lgray) \shortminus 2 &  *(lgray) \shortminus 2 & \shortminus 1\\
\none & \none & \shortminus 1
\end{ytableau}
\xrightarrow{\mathbf{(S1)}}
\begin{ytableau}
3 & 3 & *(lgray) \shortminus 2' & 3 & *(lgray) \shortminus 2\\
\none & *(lgray) \shortminus 2 &  *(lgray) \shortminus 2 & \shortminus 1\\
\none & \none & \shortminus 1
\end{ytableau}
\xrightarrow{\mathbf{(S1)}}
\begin{ytableau}
3 & 3 & 3 & *(lgray) \shortminus 2' & *(lgray) \shortminus 2\\
\none & *(lgray) \shortminus 2 &  *(lgray) \shortminus 2 & \shortminus 1\\
\none & \none & \shortminus 1
\end{ytableau}\\
&\xrightarrow{\mathsf{neg}_3}
\begin{ytableau}
*(lgray) \shortminus 3 & *(lgray) \shortminus 3 & *(lgray) \shortminus 3 &  \shortminus 2' & \shortminus 2\\
\none & \shortminus 2 & \shortminus 2 & \shortminus 1\\
\none & \none & \shortminus 1
\end{ytableau}
\xrightarrow{\mathsf{d}_3}
\begin{ytableau}
1 & 1 & 1 & 2' & 2\\
\none & 2 & 2 & 3\\
\none & \none & 3
\end{ytableau} = \mathsf{evac}(T).
\end{align*}
\end{ex}

Similarly to the case for ordinary Young tableaux \cite[Section 2.2, (5)]{PV10} \cite[Section 5]{BSS96}, the shifted evacuation algorithms, in Figures \ref{fig:SGE} and \ref{fig:SGE_k}, may be easily extended to skew shapes, by removing in both algorithms the requirement for the input to have a straight shape. We denote these operators by $\widetilde{\mathsf{evac}}$ and $\widetilde{\mathsf{evac}}_k$. However, we note that, similarly to the ordinary Young tableaux case \cite[Section 5]{BSS96}, the involution $\widetilde{\mathsf{evac}}$ is different from the reversal (Section \ref{ss:evac}), as in general, given $T \in \mathsf{ShST}(\lambda/\mu,n)$, we have $\widetilde{\mathsf{evac}} (T) \neq T^e$, since $\widetilde{\mathsf{evac}} (T)$ does not need to be shifted Knuth equivalent to $\mathsf{c}_n (T)$ or to $\mathsf{evac}(\mathsf{rect}(T))$.

\begin{ex}\label{ex:tilevac}
Let $T = \begin{ytableau}
{} & {} & {} & 1' & 1\\
\none & {} & 1 & 1\\
\none & \none & 2 & 2\\
\none & \none & \none & 3
\end{ytableau}$. Then, $\widetilde{\mathsf{evac}} (T) =  \begin{ytableau}
{} & {} & {} & 2 & 3\\
\none & {} & 1 & 3'\\
\none & \none & 2 & 3'\\
\none & \none & \none & 3
\end{ytableau} \neq T^e$ (see Example \ref{ex:reversal}).
\end{ex}

\section{Shifted tableau crystal structure and a cactus group action}\label{sec:crystal}
Gillespie, Levinson and Purbhoo \cite{GLP17} introduced a crystal-like structure on $\mathsf{ShST}(\lambda/\mu,n)$. This shifted tableau crystal is not a crystal for any known quantized enveloping algebra, unlike the one in \cite{AsOg18, GHPS18}, which is a crystal for the quantum queer Lie superalgebra.

\subsection{Shifted tableau crystals}
A shifted tableau crystal consists on a crystal-like structure of $\mathsf{ShST}(\lambda/\mu,n)$, together with primed and unprimed raising and lowering operators $E_i$, $E_i'$, $F_i$ and $F_i'$, lenght functions $\varphi_i$ and $\varepsilon_i$, for each $i \in I := [n-1]$, and a weight function. For the sake of brevity, we omit these definitions and refer to the original work in \cite{GL19,GLP17}. We will use the notation $\mathsf{ShST}(\lambda/\mu,n)$ to denote both the set and this crystal-like structure. It may be regarded as a directed, acyclic graph with weighted vertices, and $i$-coloured labelled double edges, solid ones for unprimed operators, and dashed ones for primed operators (see Figure \ref{fig:crystal}). This graph is partitioned into $i$-\emph{strings}, which are the $\{i',i\}$-connected components of $\mathsf{ShST}(\lambda/\mu,n)$, for each $i \in I$. There are two possible arrangements for these strings \cite[Section 3.1]{GL19} \cite[Section 8]{GLP17}: \emph{separated strings}, consisting of two $i$-labelled chains of equal length, connected by $i'$-labelled edges, and \emph{collapsed strings} a double chain of both $i$- and $i'$-labelled edges. Moreover, the primed and unprimed operators considered separately yield a type $A$ Kashiwara crystal.

Additionally, $\mathsf{ShST}(\lambda/\mu,n)$ decomposes into connected components, each one having a unique highest weight element (an element for which all primed and unprimed raising operators are undefined) corresponding to a LRS tableau, and a unique lowest weight element (defined analogously with lowering operators), the reversal of it. Thus, each of these connected components is isomorphic, via rectification, to $\mathsf{ShST}(\nu,n)$, for some strict partition $\nu$ \cite[Corollary 6.5]{GLP17}, with $\mathsf{ShST}(\nu,n)$ appearing with multiplicity the shifted Littlewood-Richardson coefficient $f^{\lambda}_{\mu \nu}$,
$$\mathsf{ShST}(\lambda/\mu,n) \simeq \bigsqcup\limits_{\nu} \mathsf{ShST}(\nu,n)^{f^{\lambda}_{\mu \nu}}.$$

\begin{figure}[h!]
\begin{center}
\includegraphics[scale=0.5]{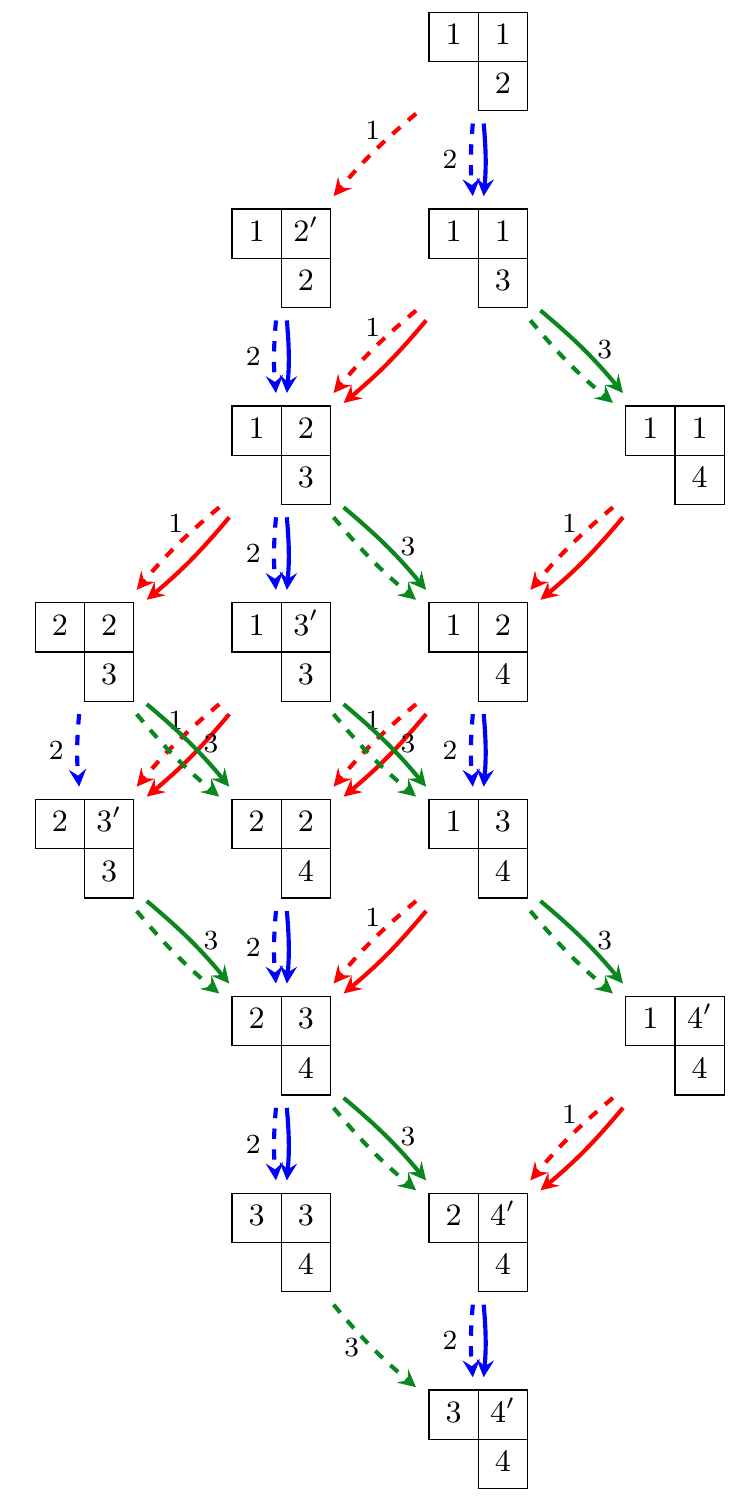}
\hspace{1cm}
\vrule{}
\hspace{1cm}
\includegraphics[scale=0.4]{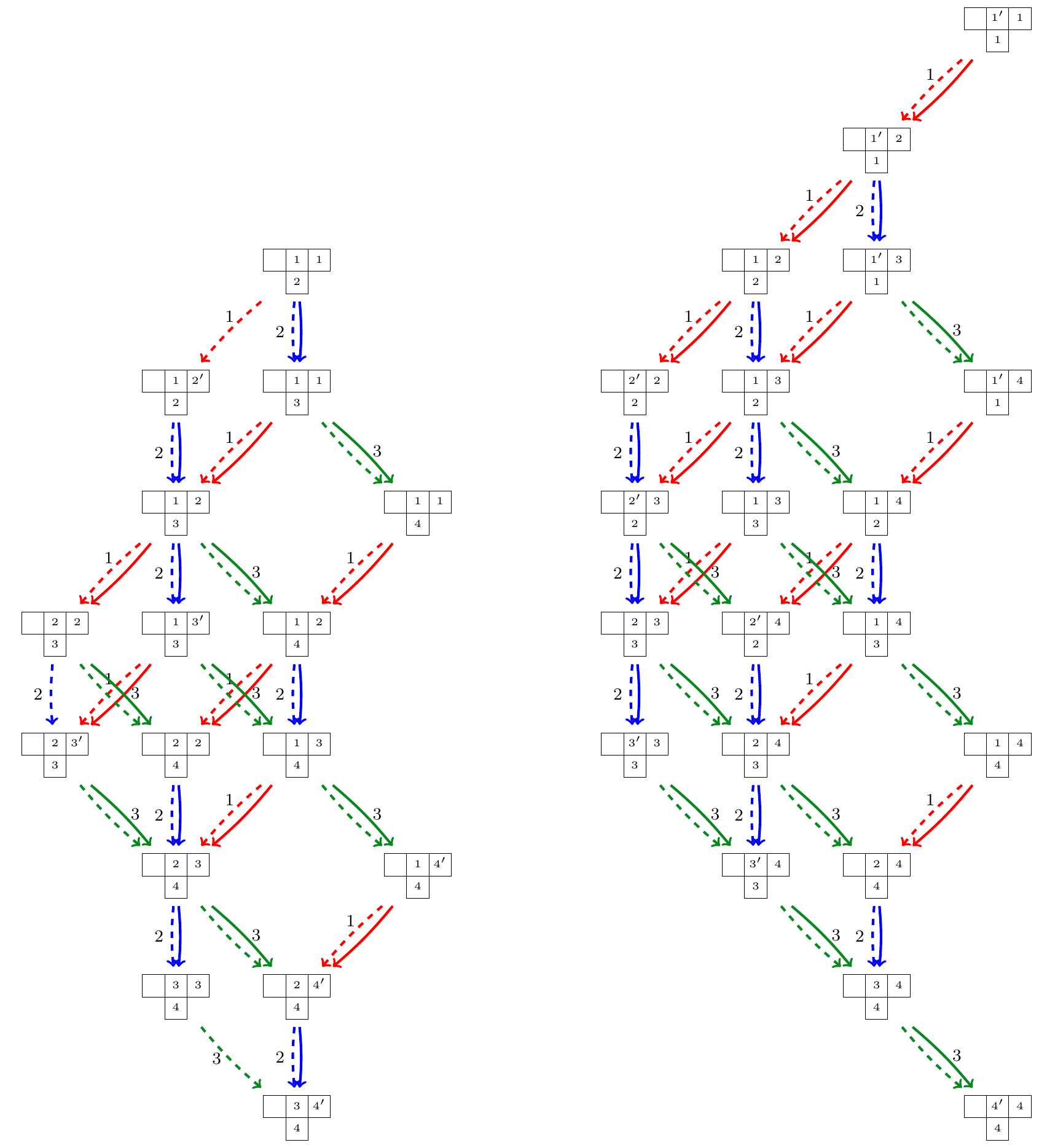}
\end{center}
\caption{On the left, a shifted tableau crystal graph $\mathsf{ShST}((2,1),4)$. On the right, a shifted tableau crystal graph $\mathsf{ShST}((3,1)/(1),4)$, which has two connected components, isomorphic via rectification to the crystal graphs of $\mathsf{ShST}((3),4)$ and $\mathsf{ShST}((2,1),4)$. The operators $F_1,F_1'$ are in red, the $F_2,F_2'$ in blue, and $F_3,F_3'$ in green.}
\label{fig:crystal}
\end{figure}

\subsection{The shifted Schützenberger involution and the crystal reflection operators}
The shifted Schützenberger or Lusztig involution is defined on a shifted tableau crystal \cite[Section 2.3.1]{GL19} in the same fashion as for type $A$ Young tableau crystal. It is realized by the shifted evacuation (for straight shapes) or the shifted reversal (for skew shapes). The shifted crystal reflection operators $\sigma_i$, for $i \in I$, were introduced in \cite[Definition 4.3]{Ro20b}, using the crystal operators. They coincide with the restriction of the Schützenberger involution the intervals of the form $\{i,i+1\}'$. Unlike what happens for the type $A$ case, they do not satisfy the braid relations of $\mathfrak{S}_n$ , thus not yielding a natural action of this group on $\mathsf{ShST}(\lambda/\mu,n)$.

\begin{prop}[{\cite[Proposition 4.1]{Ro20b}}]\label{prop:Schu}
Let $\mathsf{ShST}(\nu,n)$ denote a shifted tableau crystal with $Y_\nu$ as highest weight and $\mathsf{evac}(Y_\nu)$ as lowest weight elements. Then, there exists a unique map of sets $\eta: \mathsf{ShST}(\nu,n) \longrightarrow \mathsf{ShST}(\nu,n)$ that satisfies the following, for all $T\in \mathsf{ShST}(\nu,n)$ and for all $i \in I$:
	\begin{enumerate}
	\item $E'_i \eta (T) = \eta F'_{n-i} (T)$.
	\item $E_i \eta (T) = \eta F_{n-i} (T)$.	
	\item $F'_i \eta (T) = \eta E'_{n-i} (T)$.
	\item $F_i \eta (T) = \eta E_{n-i} (T)$.
	\item $\mathsf{wt}(\eta(T)) = \theta_{1,n} (\mathsf{wt}(T))$.
	\end{enumerate}
This map is defined on $\mathsf{ShST}(\lambda/\mu,n)$ by extending it, by coplacity of the crystal operators, to its connected components. It coincides with the evacuation $\mathsf{evac}$ in $\mathsf{ShST}(\nu,n)$, and with the reversal $e$ on the connected components of $\mathsf{ShST}(\lambda/\mu,n)$.
\end{prop}

This map is called the \emph{Schützenberger or Lusztig involution} and we use the notation $\eta$ for both straight shaped and skew tableaux. It is indeed an involution on the set of vertices of $\mathsf{ShST}(\nu,n)$, that reverses all arrows and indices, thus sending the highest weight element to the lowest and vice versa. It is coplactic and a weight-reversing, shape-preserving involution. We may define a restriction of the Schützenberger involution to the interval $[i,j]' := \{i' < i < \cdots < j' < j\}$, for $1 \leq i < j \leq n$. Given $T \in \mathsf{ShST}(\lambda/\mu,n)$, let $T^{i,j} := T^i \sqcup T^{i+1} \sqcup \cdots \sqcup T^j$. In particular, we have $T^{1,n} = T$. 

\begin{defin}\label{def:schu_ij}
Let $T \in \mathsf{ShST}(\lambda/\mu,n)$ and let $1 \leq i < j \leq n$. The \emph{partial Schützenberger involution} restricted to $[i,j]'$ is the map $\eta_{i,j}: \mathsf{ShST}(\lambda/\mu,n) \longrightarrow \mathsf{ShST}(\lambda/\mu,n) $ defined as
$$\eta_{i,j}(T) = T^{1,i-1} \sqcup \eta(T^{i,j}) \sqcup T^{j+1,n}.$$
In particular, we have $\eta_{1,n} (T) = \eta(T)$.
\end{defin}
The \emph{partial Schützenberger involutions} $\eta_{i,j}$ are also the unique set maps on $\mathsf{ShST}(\nu,n)$ satisfying certain conditions in terms of the crystal reflection operators indexed in $[i,j-1]$, similarly to Proposition \ref{prop:Schu} (see \cite[Lemma 5.4]{Ro20b}), and are extended by coplacity to $\mathsf{ShST}(\lambda/\mu,n)$, as before.

The \emph{crystal reflection operators} were originally defined by Lascoux and Schützenberger \cite{LaSchu81} in the Young tableau crystal of type $A$, as involutions sending each $i$-string to itself by reflection over its middle axis, for all $i \in I$. They coincide with the restrictions of the Schützenberger involutions to the tableaux consisting of the letters $i,i+1$, maintaining the others unchanged. The \emph{shifted crystal reflection operators} $\sigma_i$, for $i \in I$, are defined using shifted tableau crystal operators \cite[Definition 4.3]{Ro20b}. They are also involutions that send each $i$-string to itself, but through a double reflection over its vertical and horizontal middles axes, and they coincide with the restriction of the shifted Schützenberger involution to $\{i,i+1\}'$.

\begin{teo}[\cite{Ro20b, Ro20a}]\label{thm:sigma_schu}
Let $T\in \mathsf{ShST}(\lambda/\mu,n)$ and $i \in I$. Then, we have $\sigma_i (T) = \eta_{i,i+1} (T)$ and $\mathsf{wt}(\sigma_i(T)) = \theta_i (\mathsf{wt}(T))$. Moreover, the shifted crystal reflection operators satisfy the following relations:
\begin{enumerate}
\item $\sigma_i^2 = 1$.
\item $\sigma_i \sigma_j = \sigma_j \sigma_i$, for $|i-j|>1$.
\end{enumerate}
\end{teo}

\begin{ex}\label{ex:sigma2}
Let $T= \begin{ytableau}
1 & 1 & 1 & 2' & 2\\
\none & 2 & 2 & 3\\
\none & \none & 3
\end{ytableau}$. To compute $\sigma_2 (T)$ we consider the skew tableau $T^{2,3}$ and apply the Schützenberger involution (or reversal):

$$T^{2,3} = \begin{ytableau}
{} & {} & {} & 2' & 2\\
\none & 2 & 2 & 3\\
\none & \none & 3
\end{ytableau} \longrightarrow
\begin{ytableau}
*(lblue)1 & *(lblue)2 & *(lblue)3 & 2' & 2\\
\none & 2 & 2 & 3\\
\none & \none & 3
\end{ytableau} \overset{\mathsf{SW}}\longrightarrow
\begin{ytableau}
2 & 2 & 2 & 2 & *(lblue) 3\\
\none & 3 & 3 & *(lblue)1\\
\none & \none & *(lblue) 2
\end{ytableau} \xrightarrow{\mathsf{evac} \times id}
\begin{ytableau}
2 & 2 & 3' & 3 & *(lblue)3\\
\none & 3 & 3 & *(lblue)1\\
\none & \none & *(lblue)2
\end{ytableau} \overset{\mathsf{SW}}\longrightarrow
\begin{ytableau}
*(lblue)1 & *(lblue)2 & *(lblue)3 & 2 & 3'\\
\none & 2 & 3' & 3\\
\none & \none & 3
\end{ytableau} = \eta(T^{2,3}).  
$$ 
Thus, we have 
$$\sigma_2 (T) = \begin{ytableau}
1 & 1 & 1 & 2 & 3'\\
\none & 2 & 3' & 3\\
\none & \none & 3
\end{ytableau}.$$
\end{ex}

Unlike the type $A$ crystals, the reflection operators $\sigma_i$ do not define an action of the symmetric group $\mathfrak{S}_n$ on $\mathsf{ShST}(\lambda/\mu,n)$ because the braid relations $(\sigma_i \sigma_{i+1})^3 = 1$ do not need not hold.

\begin{ex}\label{ex:braid}
Let $\mathsf{ShST}(\lambda,3)$ where $\lambda=(5,3,1)$, and consider the shifted semistandard tableau
$$T = \begin{ytableau}
1 & 1 & 1 & 1 & {3'}\\
\none & 2 & 2 & {3'}\\
\none & \none & 3\end{ytableau}.$$
Then, we have
$$
\sigma_1 \sigma_2 \sigma_1 (T)=
\begin{ytableau}
1 & 1 & 1 & 2 & 3\\
\none & 2 & {3'} & 3\\
\none & \none & 3\end{ytableau}
 \neq
 \begin{ytableau}
1 & 1 & 1 & {2'} & {3'}\\
\none & 2 & {3'} & 3\\
\none & \none & 3\end{ytableau}
= \sigma_2 \sigma_1 \sigma_2 (T).$$
\end{ex}

However, we have the following result, as in \cite[Section 3.2]{AzMaCo09} for ordinary LR tableaux, ensuring that the longest permutation $\theta_{1,n} \in \mathfrak{S}_n$ acts on a connected component of $\mathsf{ShST}(\lambda/\mu,n)$ by interchanging the highest and lowest weight elements. Thus, we have an action of $\theta_{1,n}$, via the corresponding composition of shifted crystal reflection operators, recovering the reversal of LRS tableaux.

\begin{teo}\label{thm:sigma_longp}
Let $T^{\mathsf{high}}$ be a LRS tableau in $\mathsf{ShST}(\lambda/\mu,n)$. Let $\theta_{1,n} = \theta_{i_1} \cdots \theta_{i_k}$ denote the longest permutation in $\mathfrak{S}_n$. Then, $\theta_{1,n}$ acts on a connected component of $\mathsf{ShST}(\lambda/\mu,n)$ by sending its highest weight element $T^{\mathsf{high}}$ to the lowest weight element $T^{\mathsf{low}}$, i.e., 
$$\theta_{1,n} \cdot T^{\mathsf{high}} = \sigma_{i_1} \cdots \sigma_{i_k} (T^{\mathsf{high}}) = \eta(T^{\mathsf{high}}) = T^{\mathsf{low}}.$$
\end{teo}

\subsection{An action of the cactus group}\label{sec:cactusaction}
Halacheva \cite{Hala16} showed that there is a natural action of the cactus group $J_{\mathfrak{g}}$ on any $\mathfrak{g}$-crystal, for $\mathfrak{g}$ a complex, reductive, finite-dimensional Lie algebra. In particular, the cactus group $J_n$ (corresponding to $\mathfrak{g}=\mathfrak{gl}_n$) acts internally on the type $A$ crystal of semistandard Young tableux $\mathsf{SSYT}(\lambda/\mu,n)$ (here considering any partitions), via the partial Schützenberger involutions, which correspond to partial evacuations on $\mathsf{SSYT}(\nu,n)$. Following a similar approach, it was shown in \cite[Theorem 5.7]{Ro20b} that there is a natural action of $J_n$ on a shifted tableau crystal $\mathsf{ShST}(\lambda/\mu,n)$. This action is realized by the restrictions of the Schützenberger involution to all primed intervals of $[n]'$ (thus, containing in particular the shifted crystal reflection operators). We recall the definition of the cactus group as in \cite{HenKam06}.

\begin{defin}[\cite{HenKam06}]\label{def:cactus}
The \emph{$n$-fruit cactus group} $J_n$ is the free group with generators $s_{i,j}$, for $1 \leq i < j \leq n$, subject to the relations:
	\begin{enumerate}
	\item $s_{i,j}^2 = 1$.
	\item $s_{i,j} s_{k,l} = s_{k,l} s_{i,j}$ for $[i,j] \cap [k,l] = \emptyset$.
	\item $s_{i,j}s_{k,l} = s_{i+j-l,i+j-k} s_{i,j}$ for $[k,l] \subseteq [i,j]$.
	\end{enumerate}
\end{defin}

Note that there is an epimorphism $J_n \longrightarrow \mathfrak{S}_n$, sending $s_{i,j}$ to $\theta_{i,j}$. The kernel of this surjection is known as the \emph{pure cactus group} and denoted by $PJ_n$ (see \cite[Section 3.4]{HenKam06}). Moreover, the first and third relations ensure that the elements of the form $s_{1,k}$ generate $J_n$, for $1<k\leq n$, since any $s_{i,j}$ may be written as
\begin{equation}\label{eq:cactus_s1k}
s_{i,j} = s_{1,j} s_{1,j-i+1} s_{1,j}.
\end{equation}

It was shown in \cite{Ro20b} that the cactus group $J_n$ acts on a shifted tableau crystal $\mathsf{ShST}(\lambda/\mu,n)$ via the partial shifted Scützenberger involutions $\eta_{i,j}$, for $1 \leq i < j \leq n$, of Definition \ref{def:schu_ij}. An example is shown in Figure \ref{fig:crystalcactus}. 

\begin{figure}[h]
\begin{center}
\includegraphics[scale=0.5]{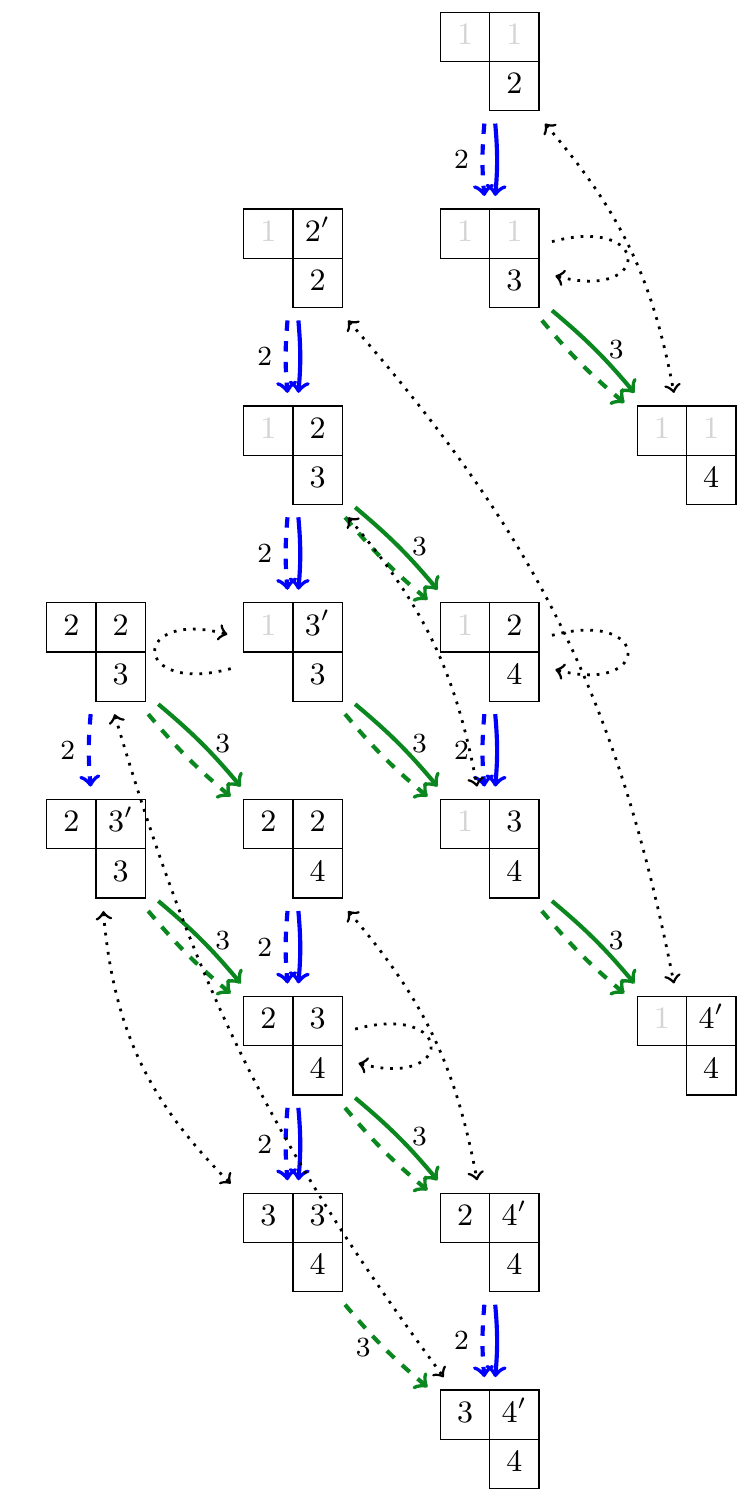}
\end{center}
\caption{The action of $s_{2,4}$ via $\eta_{2,4}$ on $\mathsf{ShST}(\nu,4)$, with $\nu = (2,1)$.}
\label{fig:crystalcactus}
\end{figure}

\begin{teo}[{\cite[Theorem 5.7]{Ro20b}}]\label{thm:cactusaction}
There is a natural action of the $n$-fruit cactus group $J_n$ on a shifted tableau crystal $\mathsf{ShST}(\lambda/\mu,n)$ given by the group homomorphism:

\begin{align*}
\phi: J_n & \longrightarrow \mathfrak{S}_{\mathsf{ShST}(\lambda/\mu,n)}\\
s_{i,j} & \longmapsto \eta_{i,j}
\end{align*}
for $1 \leq i < j \leq n$.
\end{teo}

Recall that, given $T \in \mathsf{ShST}(\nu,n)$, $\mathsf{evac}_j (T) = \mathsf{evac} (T^{1,j}) \sqcup T^{j+1,n} = \eta_{1,j}(T)$. As a consequence, the next results follow from \eqref{eq:cactus_s1k} and from $\phi$ being an homomorphism.

\begin{cor}\label{cor:eta_cactus}
Let $T \in \mathsf{ShST}(\lambda/\mu,n)$ and $1 \leq i < j \leq n$. Then,
$$\eta_{i,j} (T) = \eta_{1,j} \eta_{1,j-i+1} \eta_{1,j} (T).$$
In particular, for $T \in \mathsf{ShST}(\nu,n)$, we have
$$\eta_{i,j}(T) = \mathsf{evac}_{j} \mathsf{evac}_{j-i+1} \mathsf{evac}_{j} (T).$$
\end{cor}

\begin{teo}\label{teo:cact_evaci}
There is a natural action of the $n$-fruit cactus group on a shifted tableau crystal $\mathsf{ShST}(\nu,n)$, given by the group homomorphism, for $1<i\leq n$:

\begin{align*}
\widehat{\phi}: J_n & \longrightarrow \mathfrak{S}_{\mathsf{ShST}(\nu,n)}\\
s_{1,i} & \longmapsto \mathsf{evac}_i.
\end{align*}
\end{teo}

\begin{proof}
Since $\phi$ is an homomorphism from $J_n$ to $\mathfrak{S}_{\mathsf{ShST}(\lambda/\mu,n)}$, in particular it is an homomorphism from $J_n$ to $\mathfrak{S}_{\mathsf{ShST}(\nu,n)}$. The result then follows from \eqref{eq:cactus_s1k}, as we have $\widehat{\phi} (s_{1,i}) = \mathsf{evac}_i = \eta_{1,i} = \phi (s_{1,i})$.
\end{proof}

\section{A shifted Berenstein--Kirillov group}\label{secBK}
In this section we introduce a shifted version of the Bender--Knuth involutions for shifted semistandard tableaux. Stembridge has defined Bender--Knuth moves for shifted tableaux 
\cite[Section 6]{Stem90}, but they differ from the ones we introduce, as they are not preserve classes of canonical form (see Remark \ref{rmk:stembridge}). For ordinary Young tableaux, the Bender--Knuth involutions on letters $\{i,i+1\}$ are known to coincide with the tableau switching applied to horizontal border strips filled with the same letters 
\cite[Proposition 2.6]{BSS96}, \cite[Section 4.1]{PV10},
 together with a swapping of the letters. Thus, it is natural to use the shifted version of that algorithm, introduced by Choi, Nam and Oh 
 \cite{CNO17}, to define the shifted Bender--Knuth moves, or, equivalently, the type $C$ infusion map due to Thomas and Yong 
\cite{TY09} on standardized tableaux, followed by the shifted semistandardization process of Pechenik and Yong \cite{PY17}. As in \cite{BK95}, we are then able to recover the shifted evacuation, promotion, and shifted crystal reflection operators.

We then use the shifted Bender--Knuth involutions to introduce a shifted version of the Berenstein--Kirillov group. Following the works of Halacheva 
\cite{Hala16,Hala20} and Chmutov, Glick and Pylyavskyy 
\cite{CGP16}, we show that the shifted Berenstein--Kirillov group is isomorphic to a quotient of the cactus group and give an alternative presentation for the cactus group in terms of the shifted Bender--Knuth involutions. 

\subsection{Shifted Bender--Knuth involutions}	
We now introduce the \emph{shifted Bender--Knuth involutions} $\mathsf{t}_i$, for $i \in \mathbb{Z}_{>0}$, which will yield another presentation for the cactus group $J_n$. We first fix some notation. Given $i \in I = [n-1]$, recall that $\theta_i \in \mathfrak{S}_n$ denotes the simple transposition $(i,i+1)$. We write the cyclic permutation $\zeta_i = \theta_i \theta_{i-1} \cdots \theta_1$ as  $\zeta_i := (1,i+1,i, \ldots,2) \in \mathfrak{S}_n$. We recall that these permutations act on letters of the marked alphabet $[n]'$ as in \eqref{eq:theta_pri}.

\begin{defin}\label{defSWIJ} Let $T^{i_1}, \ldots , T^{i_n}$ be a sequence of $i_k$-border strips, with $\mathbf{i_k} \in [n]'$ and such that $\{i_1,\dots,i_n\}=[n]$. Suppose that  $T^{i_{k+1}}$ extends $T^{i_{k}}$, for $1 < k< n$. Consider $T := T^{i_1} \sqcup \cdots \sqcup T^{i_n}$, a shifted skew shape filled in the alphabet $[n]'$ (that is not necessarily a shifted semistandard filling).

\begin{enumerate}
\item Let $i,j \in [n]$ be such that $T^j$ extends $T^i$. We define $\mathsf{SP}_{i,j} (T)$ to be the filling of the shape of $T$ obtained by leaving each $T^k$ unchanged, for $k \neq i,j$, and replacing $T^i \sqcup T^j$ with  $\mathsf{SP}_1(T^i,T^j) \sqcup \mathsf{SP}_2(T^i,T^j)$.

\item We also define $\mathsf{SW}_{i_k | i_{k+1}, \ldots, i_{k+l}} (T) := \mathsf{SP}_{i_k, i_{k+l}} \mathsf{SP}_{i_k, i_{k+l-1}} \cdots \mathsf{SP}_{i_k, i_{k+1}} (T)$.
\end{enumerate}
\end{defin}

\begin{ex}
Let $T = \begin{ytableau}
1 & 1 & 2' & 2 & 3\\
\none & 2 & 2 & 3'\\
\none & \none & 3
\end{ytableau}$. Then, to compute $\mathsf{SP}_{2,3} (T)$ we have:
$$
\begin{ytableau}
1 & 1 & *(lblue)2' & *(lblue)2 & *(llblue)3\\
\none & *(lblue)2 & *(lblue)2 & *(llblue)3'\\
\none & \none & *(llblue)3
\end{ytableau} \overset{\textbf{(S6)}}\longrightarrow
\begin{ytableau}
1 & 1 & *(lblue)2' & *(llblue)3' & *(llblue)3\\
\none & *(lblue)2 & *(lblue)2 & *(lblue)2\\
\none & \none & *(llblue)3
\end{ytableau} \overset{\textbf{(S4)}}\longrightarrow
\begin{ytableau}
1 & 1 & *(lblue)2' & *(llblue)3' & *(llblue)3\\
\none & *(llblue)3 & *(lblue)2' & *(lblue)2\\
\none & \none & *(lblue)2
\end{ytableau} \overset{\textbf{(S1),(S1)}}\longrightarrow
\begin{ytableau}
1 & 1 & *(llblue)3' & *(llblue)3 & *(lblue)2'\\
\none & *(llblue)3 & *(lblue)2' & *(lblue)2\\
\none & \none & *(lblue)2
\end{ytableau} = \mathsf{SP}_{2,3} (T).$$ 
To compute $\mathsf{SW}_{1|2,3} (T)$, first apply the shifted tableau switching to the pair $(T^1, T^2)$, obtaining $(\tilde{T}^2,\tilde{T}^1)$, and then apply it again to the pair $(\tilde{T}^{1}, T^3)$: 

\begin{align*}
&\begin{ytableau}
*(lblue)1 & *(lblue)1 & *(llblue)2' & *(llblue)2 & 3\\
\none & *(llblue)2 & *(llblue)2 & 3'\\
\none & \none & 3
\end{ytableau} \overset{\textbf{(S5)}}\longrightarrow
\begin{ytableau}
*(lblue)1  & *(llblue)2' & *(lblue)1 & *(llblue)2 & 3\\
\none & *(llblue)2 & *(llblue)2 & 3'\\
\none & \none & 3
\end{ytableau} \overset{\textbf{(S6)}}\longrightarrow
\begin{ytableau}
*(lblue)1  & *(llblue)2' & *(llblue)2 & *(llblue)2 & 3\\
\none & *(llblue)2 & *(lblue)1 & 3'\\
\none & \none & 3
\end{ytableau}\overset{\textbf{(S3)}}\longrightarrow
\begin{ytableau}
*(llblue)2 & *(llblue)2 & *(llblue)2 & *(llblue)2 & 3\\
\none & *(lblue)1 & *(lblue)1 & 3'\\
\none & \none & 3
\end{ytableau} = \mathsf{SP}_{1,2} (T)\\
&\longrightarrow \begin{ytableau}
2 & 2 & 2 & 2 & *(llblue)3\\
\none & *(lblue)1 & *(lblue)1 & *(llblue)3'\\
\none & \none & *(llblue)3
\end{ytableau} \overset{\textbf{(S5)}}\longrightarrow
\begin{ytableau}
2 & 2 & 2 & 2 & *(llblue)3\\
\none & *(lblue)1 & *(llblue)3' & *(lblue)1 \\
\none & \none & *(llblue)3
\end{ytableau} \overset{\textbf{(S3)}}\longrightarrow
\begin{ytableau}
2 & 2 & 2 & 2 & *(llblue)3\\
\none & *(llblue)3 & *(llblue)3 & *(lblue)1 \\
\none & \none & *(lblue)1
\end{ytableau} = \mathsf{SP}_{1,3} \mathsf{SP}_{1,2} (T) = \mathsf{SW}_{1|2,3} (T).
\end{align*}
\end{ex}

We remark that $\mathsf{SP}_{i,j}$ and $\mathsf{SW}_{K|J}$ in general do not yield shifted semistandard tableaux, as the rows and columns may not be weakly increasing, as shown in the previous example, but they may be composed with adequate permutations of $\mathfrak{S}_n$, acting as in \eqref{eq:theta_pri} on the entries in $[n']$, ensuring that the resulting filling is a valid shifted semistandard tableau.

\begin{lema}\label{lem:comutspi}
Let $1 \leq i < j \leq n$ and $T \in \mathsf{ShST}(\lambda/\mu,n)$, such that $T^j$ extends $T^i$. Then,
\begin{enumerate}
\item $\mathsf{wt}(\mathsf{SP}_{i,j} (T)) = \mathsf{wt}(T)$\footnote{The weight of a filling of a shifted shape, not necessarily a valid shifted semistandard tableau, is defined as before.}.
\item $\mathsf{SP}_{j,i}\mathsf{SP}_{i,j}=1$.
\item $\tau \mathsf{SP}_{i,j} (T) = \mathsf{SP}_{\tau(i), \tau(j)} \tau (T)$, for any permutation $\tau \in \mathfrak{S}_n$.
\end{enumerate}
\end{lema}

\begin{proof}
To prove the first statement, we note that the shifted tableau switching solely moves boxes, not changing the total weight. For the second statement, we assume, without loss of generality, that $T = A \sqcup B$, with $A = T^i$ and $B=T^j$. Then, $\mathsf{SP}_{i,j} (A \sqcup B) = \mathsf{SP}_1(A,B) \sqcup \mathsf{SP}_2(A,B)$, where $\mathsf{SP}_1(A,B)$ is filled in $\{j',j\}$ and $\mathsf{SP}_2(A,B)$ is filled in $\{i',i\}$. Then, since the shifted tableau switching is an involution \cite[Theorem 4.3]{CNO17}, we have
\begin{align*}
\mathsf{SP}_{j,i} \big( &\mathsf{SP}_1(A,B) \sqcup \mathsf{SP}_2(A,B) \big) =\\
&= \mathsf{SP}_1\big( \mathsf{SP}_1(A,B), \mathsf{SP}_2(A,B)  \big) \sqcup \mathsf{SP}_2\big( \mathsf{SP}_1(A,B), \mathsf{SP}_2(A,B)  \big)\\
&= \mathsf{SP}_1\big( \mathsf{SP}(A,B) \big) \sqcup \mathsf{SP}_2\big( \mathsf{SP}(A,B) \big)\\
&= A \sqcup B.
\end{align*}

For the last assertion, we note that applying the shifted tableau switching to the pair $(T^i,T^j)$, followed by the action of a permutation $\tau \in \mathfrak{S}_n$ is the same as first apply the permutation $\tau$ to the letters in $T$, and then compute the shifted tableau switching to the pair that previously corresponded to $(T^i, T^j)$, which is now $(T^{\tau(i)}, T^{\tau(j)})$. 
\end{proof}

We may now define the operators $\mathsf{t}_i$, for $i \in \mathbb{Z}_{>0}$, for shifted semistandard tableaux.

\begin{defin}\label{def:sbk_mov}
Given $T \in \mathsf{ShST}(\lambda/\mu,n)$, for $n>1$, and $i \in I$, we define the \emph{shifted Bender--Knuth move} $\mathsf{t}_i$ as  
$$\mathsf{t}_i (T) := \theta_i \mathsf{SP}_{i,i+1} (T) = \mathsf{SP}_{i+1,i} \theta_i (T).$$
\end{defin}

\begin{ex}
Let $T = \begin{ytableau}
1 & 1 & 1 & 2' & 2\\
\none & 2 & 2 & 3\\
\none & \none & 3
\end{ytableau}$.
Then, we have
\begin{align*}
T = \begin{ytableau}
1 & 1 & 1 & 2' & 2\\
\none & 2 & 2 & 3\\
\none & \none & 3
\end{ytableau} &\longrightarrow
\begin{ytableau}
*(llblue)1 & *(llblue)1 & *(llblue)1 & *(lblue)2' & *(lblue)2\\
\none & *(lblue)2 & *(lblue)2 & 3\\
\none & \none & 3
\end{ytableau} \overset{\mathbf{(S5)}}\longrightarrow
\begin{ytableau}
*(llblue)1 & *(llblue)1 & *(lblue)2' & *(llblue)1 & *(lblue)2\\
\none & *(lblue)2 & *(lblue)2 & 3\\
\none & \none & 3
\end{ytableau} \overset{\mathbf{(S1)}}\longrightarrow
\begin{ytableau}
*(llblue)1 & *(llblue)1 & *(lblue)2' & *(lblue)2 & *(llblue)1\\
\none & *(lblue)2 & *(lblue)2 & 3\\
\none & \none & 3
\end{ytableau} \overset{\mathbf{(S5)}}\longrightarrow
\begin{ytableau}
*(llblue)1 & *(lblue)2' & *(llblue)1 & *(lblue)2 & *(llblue)1\\
\none & *(lblue)2 & *(lblue)2 & 3\\
\none & \none & 3
\end{ytableau}\\
&\overset{\mathbf{(S6)}}\longrightarrow
\begin{ytableau}
*(llblue)1 & *(lblue)2' & *(lblue)2 & *(lblue)2 & *(llblue)1\\
\none & *(lblue)2 & *(llblue)1 & 3\\
\none & \none & 3
\end{ytableau} \overset{\mathbf{(S3)}}\longrightarrow
\begin{ytableau}
*(lblue)2 & *(lblue)2 & *(lblue)2 & *(lblue)2 & *(llblue)1\\
\none & *(llblue)1 & *(llblue)1 & 3\\
\none & \none & 3
\end{ytableau} \overset{\theta_1}\longrightarrow
\begin{ytableau}
1 & 1 & 1 & 1 & 2\\
\none & 2 & 2 & 3\\
\none & \none & 3
\end{ytableau} = \mathsf{t}_1 (T).
\end{align*}

\begin{align*}
T = \begin{ytableau}
1 & 1 & 1 & 2' & 2\\
\none & 2 & 2 & 3\\
\none & \none & 3
\end{ytableau} &\overset{\theta_2}\longrightarrow
\begin{ytableau}
1 & 1 & 1 & *(llblue)3' & *(llblue)3\\
\none & *(llblue)3 & *(llblue)3 & *(lblue)2\\
\none & \none & *(lblue)2
\end{ytableau} \overset{\mathbf{(S7)}}\longrightarrow
\begin{ytableau}
1 & 1 & 1 & *(llblue)3' & *(llblue)3\\
\none & *(lblue)2 & *(llblue)3' & *(lblue)2\\
\none & \none & *(llblue)3
\end{ytableau} \overset{\mathbf{(S1)}}\longrightarrow
\begin{ytableau}
1 & 1 & 1 & *(llblue)3' & *(llblue)3\\
\none & *(lblue)2 & *(lblue)2 & *(llblue)3'\\
\none & \none & *(llblue)3
\end{ytableau} = \mathsf{t}_2 (T).
\end{align*}
\end{ex}

\begin{obs}
A shifted Bender--Knuth in may be formulated in terms of type $C$ infusion and semistandardization. The tableau $\mathsf{t}_1 (T)$, as in the previous example, may be computed as follows:

\begin{align*}
T = \begin{ytableau}
1 & 1 & 1 & 2' & 2\\
\none & 2 & 2 & 3\\
\none & \none & 3
\end{ytableau} &\longrightarrow
\begin{ytableau}
*(llblue)1 & *(llblue)1 & *(llblue)1 & *(lblue)2' & *(lblue)2\\
\none & *(lblue)2 & *(lblue)2 & {\color{lgray} 3}\\
\none & \none & {\color{lgray} 3}
\end{ytableau} \overset{\mathsf{std}\times \mathsf{std}}\longrightarrow
\begin{ytableau}
*(llblue)1 & *(llblue)2 & *(llblue)3 & *(lblue)1 & *(lblue)4\\
\none & *(lblue)2 & *(lblue)3 & {\color{lgray} 3}\\
\none & \none & {\color{lgray} 3}
\end{ytableau} \longrightarrow
\begin{ytableau}
*(llblue)1 & *(llblue)2 & *(lblue)1  & *(llblue)3 & *(lblue)4\\
\none & *(lblue)2 & *(lblue)3 & {\color{lgray} 3}\\
\none & \none & {\color{lgray} 3}
\end{ytableau} \longrightarrow
\begin{ytableau}
*(llblue)1 & *(llblue)2 & *(lblue)1 & *(lblue)4 & *(llblue)3 \\
\none & *(lblue)2 & *(lblue)3 & {\color{lgray} 3}\\
\none & \none & {\color{lgray} 3}
\end{ytableau}\\
&\longrightarrow
\begin{ytableau}
*(llblue)1 & *(lblue)1 & *(llblue)2 & *(lblue)4 & *(llblue)3 \\
\none & *(lblue)2 & *(lblue)3 & {\color{lgray} 3}\\
\none & \none & {\color{lgray} 3}
\end{ytableau} \longrightarrow
\begin{ytableau}
*(llblue)1 & *(lblue)1 & *(lblue)3 & *(lblue)4 & *(llblue)3 \\
\none & *(lblue)2 & *(llblue)2 & {\color{lgray} 3}\\
\none & \none & {\color{lgray} 3}
\end{ytableau} \longrightarrow
\begin{ytableau}
*(lblue)1 & *(llblue)1 & *(lblue)3 & *(lblue)4 & *(llblue)3 \\
\none & *(lblue)2 & *(llblue)2 & {\color{lgray} 3}\\
\none & \none & {\color{lgray} 3}
\end{ytableau} 
\longrightarrow
\begin{ytableau}
*(lblue)1 & *(lblue)2 & *(lblue)3 & *(lblue)4 & *(llblue)3 \\
\none & *(llblue)1 & *(llblue)2 & {\color{lgray} 3}\\
\none & \none & {\color{lgray} 3}
\end{ytableau}.
\end{align*}
Then, the semistandardization process with respect to $wt_2 = (4)$ and $wt_1 = (3)$ yields:

$$\begin{ytableau}
*(lblue)1 & *(lblue)2 & *(lblue)3 & *(lblue)4 & *(llblue)3 \\
\none & *(llblue)1 & *(llblue)2 & {\color{lgray} 3}\\
\none & \none & {\color{lgray} 3}
\end{ytableau}
\overset{\mathsf{sstd}_{(4)} \times \mathsf{sstd}_{(3)}}\longrightarrow
\begin{ytableau}
1 & 1 & 1 & 1 & 2\\
\none & 2 & 2 & 3\\
\none & \none & 3
\end{ytableau} = \mathsf{t}_1 (T).$$
\end{obs}

\begin{prop}\label{prop:ti_props}
The shifted Bender--Knuth operators $\mathsf{t}_i$ satisfy the following, for any $i \in I$:
\begin{enumerate}
	\item $\mathsf{t}_i^2 = 1$.
	\item $\mathsf{t}_i \mathsf{t}_j = \mathsf{t}_j \mathsf{t}_i$, for $|i-j| >1$.
	\item $\mathsf{wt}(\mathsf{t}_i (T))= \theta_i (\mathsf{wt}(T))$, for any $T \in \mathsf{ShST}(\lambda/\mu,n)$.
\end{enumerate}
Thus, $\mathsf{t}_i$ defines a bijection between the set of shifted semistandard tableaux of shape $\lambda/\mu$ and weight $\nu$, and the set of shifted semistandard tableaux of the same shape and weight $\theta_i(\nu)$.
\end{prop}

\begin{proof}
By Lemma \ref{lem:comutspi}, we have 
$$\mathsf{t}_i^2 = \theta_i \mathsf{SP}_{i,i+1}\theta_i \mathsf{SP}_{i,i+1} = \mathsf{SP}_{i+1,i}  \theta_i^2 \mathsf{SP}_{i,i+1} = \mathsf{SP}_{i+1,i} \mathsf{SP}_{i,i+1} = 1.$$ 
The second assertion results from $\mathsf{t}_i$ acting only on the letters $\{i,i+1\}'$, leaving the others unchanged. For the third statement, Lemma \ref{lem:comutspi}, ensures that 
$$\mathsf{wt}(\mathsf{t}_i (T)) = \mathsf{wt}(\mathsf{SP}_{i+1,i} \theta_i (T)) = \mathsf{wt} (\theta_i (T)) = \theta_i (\mathsf{wt}(T)).$$
\end{proof}

\begin{obs}
Since the operators $\mathsf{t}_i$ act on the weight of a shifted semistandard tableau $T$ as the simple transposition $\theta_i$, for each $i$, they can be used to derive a proof that the Schur $P$-functions are symmetric, similarly to the one for classic Schur functions using Bender--Knuth moves.
\end{obs}

\begin{figure}[h]
\includegraphics[scale=0.6]{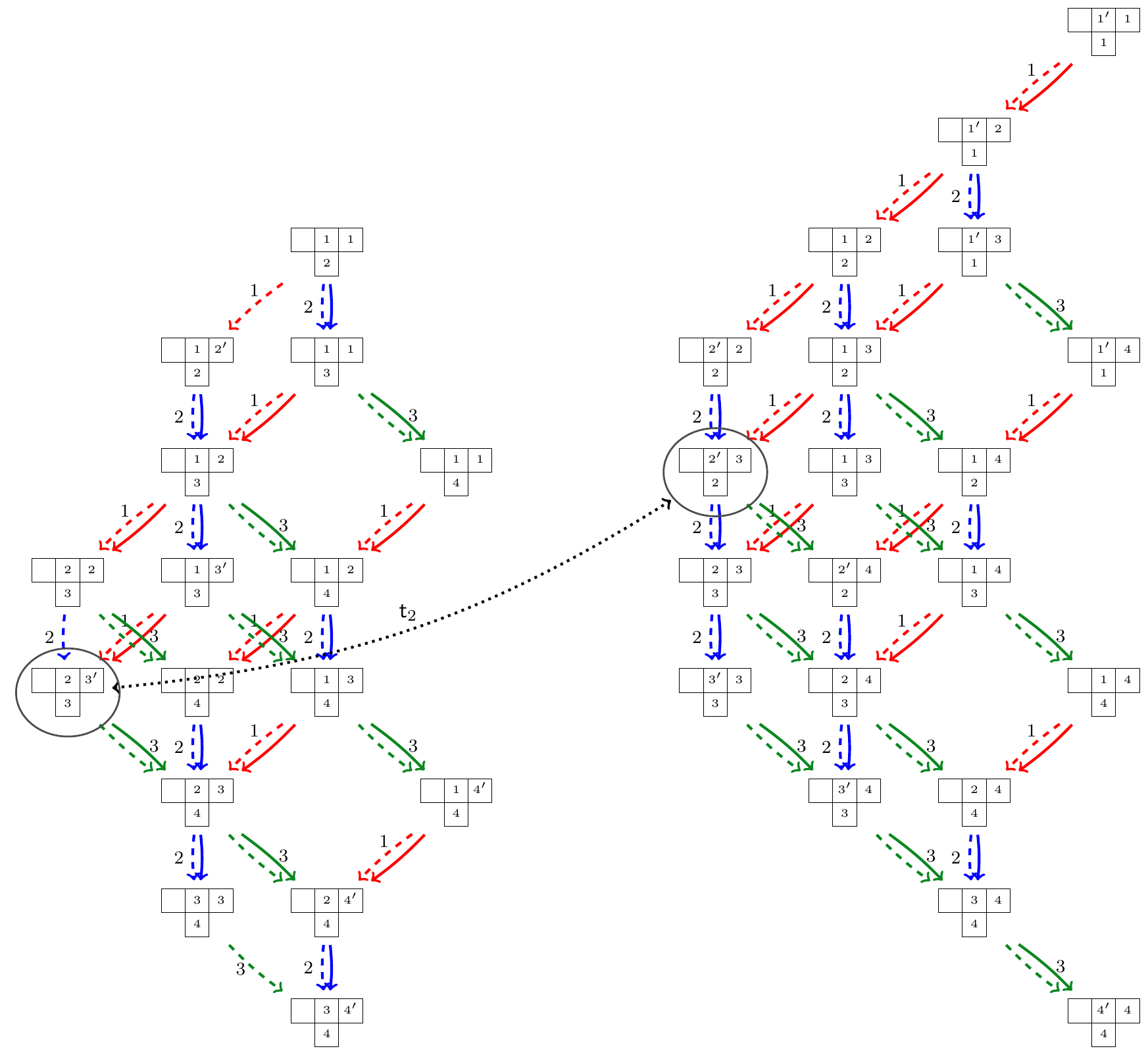}
\caption{An example of the action of $\mathsf{t}_2$ on a shifted tableau crystal $\mathsf{ShST}(\lambda/\mu,4)$, with $\lambda=(3,1)$ and $\nu=(1)$, which has two connected components.}
\label{fig:crystal_t2}
\end{figure}

As in the ordinary case, the operators $\mathsf{t}_i$ do not commute with the \textit{jeu de taquin}, as shown in Example \ref{ex:not_copl}. In general, $\mathsf{t}_i$ does not coincide with $\sigma_i$ (although $\mathsf{t}_1$ and $\sigma_1$ coincide on straight-shaped tableaux). Moreover, if $T$ is in a $i$-string $\mathcal{B}_i$, it is not necessary for $\mathsf{t}_i (T)$ to be in the same $i$-string (see Figure \ref{fig:crystal_t2}).

\begin{ex}\label{ex:not_copl}
Considering $T$ of the previous example, we have
$$T = \begin{ytableau}
1 & 1 & 1 & 2' & 2\\
\none & 2 & 2 & 3\\
\none & \none & 3
\end{ytableau} \equiv_k \begin{ytableau}
{} & {} & 1' & 2' & 2\\
\none & 1 & 1 & 2 & 3\\
\none & \none & 2 & 3
\end{ytableau} = T'$$
and $$\mathsf{t}_2 (T) = \begin{ytableau}
1 & 1 & 1 & 3' & 3\\
\none & 2 & 2 & 3'\\
\none & \none & 3
\end{ytableau} \not\equiv_k  \begin{ytableau}
{} & {} & 1' & 2 & 2\\
\none & 1 & 1 & 3' & 3\\
\none & \none & 3 & 3
\end{ytableau} = \mathsf{t}_2 (T').$$
Moreover, note that (see Example \ref{ex:sigma2}) 
$$\sigma_2 (T) = \begin{ytableau}
1 & 1 & 1 & 2 & 3'\\
\none & 2 & 3' & 3\\
\none & \none & 3
\end{ytableau} \neq \mathsf{t}_2 (T).$$ 
\end{ex}

Like the case for type $A$, we can define a shifted version of the \emph{promotion} operator due to Schützenberger, using the shifted Bender--Knuth involutions, and then recover the shifted evacuation and shifted crystal reflection operators for straight-shaped tableaux.

\begin{defin}\label{def:prom}
Given $T \in \mathsf{ShST}(\lambda/\mu,n)$ and $i \in I$, we define the \emph{shifted promotion operator} $\mathsf{p}_i$ as
$$\mathsf{p}_i (T) := \mathsf{t}_i \mathsf{t}_{i-1} \cdots \mathsf{t}_1 (T).$$
As a result of $\mathsf{t}_i$ being involutions, we have $\mathsf{p}_i^{-1} = \mathsf{t}_1 \cdots \mathsf{t}_{i-1} \mathsf{t}_i$.
\end{defin}

We will show that the promotion $\mathsf{p}_i (T)$ coincides with the shifted tableau switching on the pairs $(T^1, T^2 \sqcup \cdots \sqcup T^{i+1})$, followed by an adequate cyclic substitution of the letters. We first prove some auxiliary results.

\begin{lema}\label{lem:zeta_comut}
Let $T \in \mathsf{ShST}(\lambda/\mu,n)$ and let $1 \leq i < j \leq n-1$. Then, $T^{i+1} \sqcup \cdots \sqcup T^{j}$ extends $T^{i}$, and for any $\tau \in \mathfrak{S}_n$ we have
$$\tau \mathsf{SW}_{i|i+1, \ldots, j} (T) = \mathsf{SW}_{\tau(i)| \tau(i+1), \ldots, \tau(j)} \tau (T).$$
\end{lema}

\begin{proof}
By Definition \ref{defSWIJ} and Lemma \ref{lem:comutspi}, we have
\begin{align*}
\tau \mathsf{SW}_{i|i+1, \ldots, j} (T) &= \tau \mathsf{SP}_{i,j} \mathsf{SP}_{i,j-i} \cdots \mathsf{SP}_{i,i+1} (T)\\
&= \mathsf{SP}_{\tau(i), \tau(j)} \mathsf{SP}_{\tau(i),\tau(j-1)} \cdots \mathsf{SP}_{\tau(i),\tau(i+1)} \tau (T)\\
&= \mathsf{SW}_{\tau(i)| \tau(i+1), \ldots, \tau(j)} \tau (T).
\end{align*}
\end{proof}

\begin{lema}\label{lem:zetai} 
Let $T \in \mathsf{ShST}(\lambda/\mu,n)$ and let $1 < i \leq n-1$. We have 
$$\zeta_i \mathsf{SW}_{i|i+1} \mathsf{SW}_{i-1|i,i+1} \cdots \mathsf{SW}_{2|3, \ldots,i+1} (T) = \mathsf{SW}_{i-1|i} \mathsf{SW}_{i-2|i-1,i} \cdots \mathsf{SW}_{1|2, \ldots,i} \zeta_i (T).$$
\end{lema}

\begin{proof}
Applying successively Lemma \ref{lem:zeta_comut}, we have 
\begin{align*}
\zeta_i \mathsf{SW}_{i|i+1} \mathsf{SW}_{i-1|i,i+1} \cdots \mathsf{SW}_{2|3, \ldots,i+1} &= \mathsf{SW}_{\zeta_i (i)|\zeta_i(i+1)} \mathsf{SW}_{\zeta_i(i-1)|\zeta_i(i),\zeta_i(i+1)} \cdots \mathsf{SW}_{\zeta_i(2)|\zeta_i(3), \ldots,\zeta_i(i+1)} \zeta_i\\
&= \mathsf{SW}_{i-1|i} \mathsf{SW}_{i-2|i-1,i} \cdots \mathsf{SW}_{1|2, \ldots,i} \zeta_i.
\end{align*}
\end{proof}

\begin{prop}\label{prop:prom_swi}
Given $T \in \mathsf{ShST}(\lambda/\mu,n)$, and $i\in I$, we have
$$\mathsf{p}_i (T) = \zeta_i \mathsf{SW}_{1| 2,\ldots,i+1} (T).$$
\end{prop}

\begin{proof}
The proof is done by induction on $i$. For $i=1$, we have 
$$\mathsf{p}_1 (T) = \mathsf{t}_1 (T) = \theta_1 \mathsf{SP}_{1,2} (T) = \zeta_1 \mathsf{SW}_{1|2} (T).$$ Assuming the result is true for some $i \geq 1$, by Definition \ref{defSWIJ} and Lemma \ref{lem:comutspi}, we have
\begin{align*}
\mathsf{p}_{i+1}(T) &= \mathsf{t}_{i+1} \mathsf{p}_i (T)\\
&= \theta_{i+1} \mathsf{SP}_{i+1,i+2} \zeta_{i} \mathsf{SW}_{1|2, \ldots, i+1} (T)\\
&= \theta_{i+1} \zeta_{i} \mathsf{SP}_{\zeta_i^{-1}(i+1), \zeta_i^{-1}(i+2)} \mathsf{SW}_{1|2, \ldots, i+1} (T)\\
&= \theta_{i+1} \zeta_{i} \mathsf{SP}_{1,i+2} \mathsf{SW}_{1|2, \ldots, i+1} (T)\\
&= \zeta_{i+1}  \mathsf{SW}_{1|2, \ldots, i+1, i+2} (T).
\end{align*}
\end{proof}

For $i \geq 1$, we define
\begin{equation}\label{eq:def_qi}
\mathsf{q}_i := \mathsf{t}_1 (\mathsf{t}_2 \mathsf{t}_1) \cdots (\mathsf{t}_i \cdots \mathsf{t}_1).
\end{equation}
Recall that $\widetilde{\mathsf{evac}}_k$ is the operator obtained by allowing skew-shaped tableaux on the algorithm of Figure \ref{fig:SGE_k}, which differs from the reversal on skew shapes. We will show that $\widetilde{\mathsf{evac}}_k$ and $\mathsf{evac}_k$ may be written as a composition of promotion operators. As a consequence, $\mathsf{q}_i$ coincides with $\widetilde{\mathsf{evac}}_{i+1}$ on skew-shaped shifted tableaux and with $\mathsf{evac}_{i+1}$ on straight-shapes ones. This coincidence implies that $\mathsf{q}_i$ are involutions, for any $i \geq 1$. 

\begin{prop}\label{prop:evac_ti}
Given $T \in \mathsf{ShST}(\lambda/\mu,n)$ and $i \in I$, we have
$$\widetilde{\mathsf{evac}}_{i+1} (T) = \mathsf{q}_i(T) = \mathsf{p}_1 \mathsf{p}_2 \cdots \mathsf{p}_i (T) = \mathsf{t}_1 (\mathsf{t}_2 \mathsf{t}_1) \cdots (\mathsf{t}_i \mathsf{t}_{i-1} \cdots \mathsf{t}_1)(T).$$
In particular, when $T \in \mathsf{ShST}(\nu,n)$ we have
$$\eta_{1,i+1}(T)=\mathsf{evac}_{i+1} (T) = \mathsf{q}_i(T)  = \mathsf{p}_1 \mathsf{p}_2 \cdots \mathsf{p}_i (T) = \mathsf{t}_1 (\mathsf{t}_2 \mathsf{t}_1) \cdots (\mathsf{t}_i \mathsf{t}_{i-1} \cdots \mathsf{t}_1)(T).$$
\end{prop}

\begin{proof}
The proof is analogous either for straight or skew shape cases, as $\mathsf{evac}_{i+1}$ and  $\widetilde{\mathsf{evac}}_{i+1}$ coincide on straight-shaped tableaux. By \eqref{eq:d_neg}, we have $\mathsf{d}_{i+1} \mathsf{neg}_{i+1} \cdots \mathsf{neg}_{1} = \theta_{1,i+1} = \zeta_1 \cdots \zeta_{i}$. Moreover, it is clear that, for $l < k < i$,
\begin{equation}
\begin{split}
\mathsf{SW}_{\shortminus k|k+1, \ldots, i+1} \mathsf{neg}_k &= \mathsf{neg}_k \mathsf{SW}_{k|k+1, \ldots, i+1}\\
\mathsf{SW}_{k|k+1, \ldots, i+1} \mathsf{neg}_l &= \mathsf{neg}_l \mathsf{SW}_{k|k+1, \ldots, i+1}.
\end{split}
\end{equation}
Then, the algorithm for $\widetilde{\mathsf{evac}}_{i+1}$ (see Figure \ref{fig:SGE_k}) performed on $T$ can be written as:

\begin{align*}
\widetilde{\mathsf{evac}}_{i+1}(T) &= \mathsf{d}_{i+1} \mathsf{neg}_{i+1} \mathsf{SW}_{\shortminus i|i+1} \mathsf{neg}_{i} \cdots \mathsf{SW}_{\shortminus 2|3,\ldots,i+1} \mathsf{neg}_2 \mathsf{SW}_{\shortminus 1|2, \ldots, i+1} \mathsf{neg}_1 (T)\\
&= \mathsf{d}_{i+1} \mathsf{neg}_{i+1} \mathsf{neg}_{i} \mathsf{SW}_{i|i+1} \cdots \mathsf{neg}_2 \mathsf{SW}_{2|3,\ldots,i+1}  \mathsf{neg}_1 \mathsf{SW}_{1|2, \ldots, i+1} (T)\\
&= \mathsf{d}_{i+1} \mathsf{neg}_{i+1} \cdots  \mathsf{neg}_2  \mathsf{neg}_1 \mathsf{SW}_{i|i+1} \cdots \mathsf{SW}_{2|3,\ldots,i+1} \mathsf{SW}_{1|2, \ldots, i+1} (T)\\
&= \zeta_{1} \cdots \zeta_{i} \mathsf{SW}_{i|i+1} \cdots \mathsf{SW}_{2|3,\ldots,i+1} \mathsf{SW}_{1|2, \ldots, i+1} (T).
\end{align*}
To conclude the proof, we claim that
\begin{equation}\label{eq:evac_zeta}
\zeta_{1} \cdots \zeta_{i} \mathsf{SW}_{i|i+1} \cdots \mathsf{SW}_{2|3,\ldots,i+1} \mathsf{SW}_{1|2, \ldots, i+1} (T) = \zeta_1 \mathsf{SW}_{1|2} \zeta_2 \mathsf{SW}_{1|2,3} \cdots  \zeta_i \mathsf{SW}_{1|2,\ldots,i} (T) = \mathsf{p}_1 \mathsf{p}_2 \cdots \mathsf{p}_i (T).
\end{equation}

We prove \eqref{eq:evac_zeta} by induction on $i$. The base case is trivial. For the induction step, assume the claim holds for some $i \geq 1$. Then, by Lemma \ref{lem:zetai} and Proposition \ref{prop:prom_swi}, we have
\begin{align*}
\zeta_{1} \cdots \zeta_{i} &\zeta_{i+1} \mathsf{SW}_{i+1|i+2} \cdots \mathsf{SW}_{2|3,\ldots,i+1,i+2} \mathsf{SW}_{1|2, \ldots, i+1,i+2} (T) = \\ 
&=  \zeta_{1} \cdots \zeta_{i} \mathsf{SW}_{i|i+1} \cdots  \mathsf{SW}_{1|2, \ldots, i+1} \zeta_{i+1} \mathsf{SW}_{1|2, \ldots, i+1, i+2}\\
&= \zeta_1 \mathsf{SW}_{1|2} \zeta_2 \mathsf{SW}_{1|2,3} \cdots  \zeta_i \mathsf{SW}_{1|2,\ldots,i} \zeta_{i+1} \mathsf{SW}_{1|2, \ldots, i+1, i+2} (T)\\
&=  \mathsf{p}_1 \mathsf{p}_2 \cdots \mathsf{p}_i \mathsf{p}_{i+1} (T).
\end{align*}
\end{proof}

\begin{cor}\label{cor:qi_inv}
Let $i \in I$. Then $\mathsf{q}_i^2 = 1$ and $\mathsf{wt} (\mathsf{q}_i (T)) = \theta_{1,i+1} (T)$.
\end{cor}

\begin{proof}
Since $\widetilde{\mathsf{evac}}_{i+1}$ is an involution, for any $i \geq 1$, then so it is $\mathsf{q}_i$.
From Proposition \ref{prop:ti_props}, we have $\mathsf{wt} (\mathsf{q}_i (T)) = \mathsf{wt} (\mathsf{t}_1 (\mathsf{t}_2 \mathsf{t}_1) \cdots (\mathsf{t}_i \mathsf{t}_{i-1} \cdots \mathsf{t}_1)(T)) = \theta_{1} (\theta_2 \theta_1) \cdots (\theta_i \cdots \theta_{1})  (T) = \theta_{1,i+1} (T)$.
\end{proof}

\begin{cor}\label{cor:sigma_prom}
Given $T \in \mathsf{ShST}(\nu,n)$ and $i \in I$, we have
$$\sigma_i (T) = \mathsf{evac}_{i+1} \mathsf{evac}_2 \mathsf{evac}_{i+1} (T) = \mathsf{q}_i \mathsf{t}_1 \mathsf{q}_i (T) = \mathsf{p_1} (\mathsf{p}_2 \cdots \mathsf{p}_i)^2 (T).$$
\end{cor}

\begin{proof}
By Theorem \ref{thm:sigma_schu} and Corollary \ref{cor:eta_cactus}, we have $\sigma_i (T) = \mathsf{evac}_{i+1} \mathsf{evac}_2 \mathsf{evac}_{i+1} (T)$. From Proposition \ref{prop:evac_ti}, we have 
\begin{align*}
\mathsf{evac}_{i+1} \mathsf{evac}_2 \mathsf{evac}_{i+1} (T) &= \mathsf{q}_i \mathsf{q}_1 \mathsf{q}_i (T) = \mathsf{q}_i \mathsf{t}_1 \mathsf{q}_i (T)\\
&= (\mathsf{p}_1 \mathsf{p}_2 \cdots \mathsf{p}_i) \mathsf{t}_1 (\mathsf{p}_1 \mathsf{p}_2 \cdots \mathsf{p}_i)(T)\\
&= (\mathsf{p}_1 \mathsf{p}_2 \cdots \mathsf{p}_i) \mathsf{t}_1 (\mathsf{t}_1 \mathsf{p}_2 \cdots \mathsf{p}_i)(T)\\
&= \mathsf{p}_1 (\mathsf{p}_2 \cdots \mathsf{p}_i)(\mathsf{p}_2 \cdots \mathsf{p}_i)(T).
\end{align*}
\end{proof}

It is natural to consider the restriction of the operator $\widetilde{\mathsf{evac}}_k$ to an interval $[i,j]'$, for $1 \leq i < j \leq n$, in the same fashion as Definition \ref{def:schu_ij}. For $T \in \mathsf{ShST}(\lambda/\mu,n)$ and $1 \leq i < j \leq n$, we define
\begin{equation}\label{eq:evac_tilde_ij}
\widetilde{\mathsf{evac}}_{i,j}(T):= T^{1,i-1} \sqcup \widetilde{\mathsf{evac}}(T^{i,j}) \sqcup  T^{j+1,n}.
\end{equation}
Clearly, $\widetilde{\mathsf{evac}}_{1,k} = \widetilde{\mathsf{evac}}_k$ and $\widetilde{\mathsf{evac}}_{i,j}$ coincides with $\eta_{i,j}$, on straight-shaped shifted tableaux. However, these operators do not satisfy the relation $\widetilde{\mathsf{evac}}_{i,j}=\widetilde{\mathsf{evac}}_{j} \widetilde{\mathsf{evac}}_{j-i+1} \widetilde{\mathsf{evac}}_{j}$, for $\mu \neq \emptyset$, unlike the operators $\eta_{i,j}$ (Corollary \ref{cor:eta_cactus}), as shown in the next example.

\begin{ex}\label{ex:evac_tilde_neq}
Let $T = \begin{ytableau}
1 & 1 & 1 & 1 & 3'\\
\none & 2 & 2 & 3'\\
\none & \none & 3
\end{ytableau}$.

We have 
$$\widetilde{\mathsf{evac}}_{2,3} (T) = \begin{ytableau}
1 & 1 & 1 & 1 & 2'\\
\none & 2 & 2 & 3 \\
\none & \none & 3
\end{ytableau} \neq 
\begin{ytableau}
1 & 1 & 1 & 1 & 2\\
\none & 2 & 2 & 3'\\
\none & \none & 3
\end{ytableau} =  \widetilde{\mathsf{evac}}_{3} \widetilde{\mathsf{evac}}_{2} \widetilde{\mathsf{evac}}_{3} (T).$$
\end{ex}

\begin{obs}\label{rmk:stembridge}
Stembridge introduced a shifted version of Bender--Knuth moves in \cite[Section 6]{Stem90}. These are two-to-two maps acting on adjacent letters by reverting their weight. Shifted tableaux are not required to be in canonical form here, and in general, these maps are not compatible with canonical form. For instance, consider the following tableau, in canonical form:
 
$$T=\begin{ytableau}
\none & \none & \none & \none & \none & \none & 1 & 2'\\
\none & \none & \none & \none & 2' & 2 & 2\\
1' & 1 & 1 & 1 & 2\\
1 & 2'\\
\none & 2
\end{ytableau}$$
and consider the representatives of $T$:
$$T_1 = \begin{ytableau}
\none & \none & \none & \none & \none & \none & 1 & 2'\\
\none & \none & \none & \none & 2' & 2 & 2\\
1' & 1 & 1 & 1 & 2\\
1 & 2'\\
\none & 2
\end{ytableau} \quad 
T_2 = \begin{ytableau}
\none & \none & \none & \none & \none & \none & 1 & 2'\\
\none & \none & \none & \none & 2' & 2 & 2\\
1' & 1 & 1 & 1 & 2\\
1 & 2'\\
\none & 2'
\end{ytableau}
\quad
T_3 = \begin{ytableau}
\none & \none & \none & \none & \none & \none & 1 & 2'\\
\none & \none & \none & \none & 2' & 2 & 2\\
1' & 1 & 1 & 1 & 2\\
1' & 2'\\
\none & 2
\end{ytableau} \quad
T_4 = \begin{ytableau}
\none & \none & \none & \none & \none & \none & 1 & 2'\\
\none & \none & \none & \none & 2' & 2 & 2\\
1' & 1 & 1 & 1 & 2\\
1' & 2'\\
\none & 2'
\end{ytableau}.$$
Using the maps in \cite{Stem90}, we have:
$$\{T_1, T_2\} \longrightarrow \Big\{ \begin{ytableau}
\none & \none & \none & \none & \none & \none & 1' & 2\\
\none & \none & \none & \none & 1' & 1 & 1\\
1' & 1 & 2 & 2 & 2\\
1 & 2'\\
\none & 2
\end{ytableau},
\begin{ytableau}
\none & \none & \none & \none & \none & \none & 1' & 2\\
\none & \none & \none & \none & 1' & 1 & 1\\
1' & 1 & 2 & 2 & 2\\
1 & 2'\\
\none & 2'
\end{ytableau} \Big\} =: \{\hat{T}_1, \hat{T}_2\}$$

$$\{T_3, T_4\} \longrightarrow \Big\{
\begin{ytableau}
\none & \none & \none & \none & \none & \none & 1' & 2\\
\none & \none & \none & \none & 1' & 1 & 1\\
1' & 1 & 2' & 2 & 2\\
1' & 2'\\
\none & 2
\end{ytableau},
\begin{ytableau}
\none & \none & \none & \none & \none & \none & 1' & 2\\
\none & \none & \none & \none & 1' & 1 & 1\\
1' & 1 & 2' & 2 & 2\\
1' & 2'\\
\none & 2'
\end{ytableau}
\Big\} =: \{\hat{T}_3, \hat{T}_4\}$$	

The tableaux in $\{\hat{T}_1, \hat{T}_2\}$ do not have the same canonical form as the ones in $\{\hat{T}_3,\hat{T}_4\}$.
\end{obs}	
	
\subsection{The Berenstein--Kirillov group}\label{subsec:bk}

The Bender--Knuth involutions $t_i$, for $i \in I$, are involutions on semistandard Young tableaux filled in $[n]$, that act only on the letters $\{i,i+1\}$, reverting their weight \cite{BeKn72}. They are known to coincide with the tableau switching on type $A$ on two consecutive letters, together with a swapping of those letters \cite{BSS96}. The \emph{Berenstein--Kirillov group} $\mathcal{BK}$ (or \emph{Gelfand-Tsetlin group}), is the free group generated by these involutions $t_i$, for $i > 0$, modulo the relations they satisfy on semistandard Young tableaux of any shape \cite{BK16, BK95, CGP16}. Some of the known relations to hold in $\mathcal{BK}$ \cite[Corollary 1.1]{BK95} are
\begin{equation}\label{eq:bkrelations}
t_i^2 = 1, \qquad t_i t_j = t_j t_i,\; \text{for}\; |i-j|>1, \qquad (t_1 q_i)^4 = 1,\; \text{for}\; i > 2,
\end{equation}
where $q_{i} := t_1 (t_2 t_1) \cdots (t_i t_{i-1} \cdots t_1)$, for $i \geq 1$, are involutions, and
\begin{equation}\label{eq:bkrelations_special}
(t_1 t_2)^6 = 1.
\end{equation}

The restriction of the evacuation to the alphabet $\{1, \ldots, i\}$, on straight-shaped semistandard Young tableaux, may be regarded as an element of $\mathcal{BK}$, and it is computed by $q_{i-1}$ \cite{BK95,CGP16,Hala20,Hala16}. We also let $q_{j,k} := q_{k-1} q_{k-j} q_{k-1}$, for $j<k$. In particular, $q_i = q_{1,i+1}$ and $q_{j,k}$ computes the restriction of the evacuation to the alphabet $\{j, \ldots, k\}$, as an element of $\mathcal{BK}$. Chmutov, Glick and Pylyavskyy found another relation\cite[Theorem 1.6]{CGP16}.
\begin{equation}\label{eq:bkrelationextra}
(t_i q_{j,k})^2 = 1, \;\text{for}\; i+1<j<k.
\end{equation}

The relation \eqref{eq:bkrelationextra} does not follow from the previous known relations \eqref{eq:bkrelations} and \eqref{eq:bkrelations_special} in $\mathcal{BK}$, but is instead a consequence from the cactus relations satisfied by the operators $q_{i,j}$ in $\mathcal{BK}$, studied by Halacheva \cite{Hala20,Hala16} and Chmutov, Glick and Pylyavskyy \cite{CGP16}. We remark that \eqref{eq:bkrelationextra} generalizes the relation $(t_1 q_i)^4=1$, since
$$(t_1 q_i)^4 = (t_1 q_{i} t_1 q_{i})^2 = (t_1 q_{i} q_1 q_{i})^2 = (t_1 q_{i,i+1})^2.$$

Let $\mathcal{BK}_n$ be the subgroup of $\mathcal{BK}$ generated by $t_1, \ldots, t_{n-1}$. The involutions $q_i$, for $i \in I$, provide another set of generators for $\mathcal{BK}_n$, and their action on straight-shaped Young tableaux coincide with the one of the restriction of the Schützenberger involution (or evacuation) to $[i+1]$ \cite[Remark 1.3]{BK95}. It was shown in \cite{CGP16}, using semistandard growth diagrams, that $\mathcal{BK}_n$ is isomorphic to a quotient of the cactus group. This result could also be derived by noting the coincidence of the actions of $J_n$ \cite{Hala16} and $\mathcal{BK}_n$ on a straight-shaped semistandard Young tableau crystal $\mathsf{SSYT}(\nu,n)$, as noted in \cite[Remark 3.9]{Hala20}.

\begin{teo}
The group $\mathcal{BK}_n$ is isomorphic to a quotient of $J_n$, as a result of the following being group epimorphisms from $J_n$ to $\mathcal{SBK}_n$:
\begin{enumerate}
\item $s_{i,j} \mapsto q_{i,j}$ \cite[Theorem 1.4]{CGP16}.

\item $s_{1,j} \mapsto q_{j-1}$ \cite[Remark 1.3]{BK95}, \cite[Section 10.2]{Hala16}, \cite[Remark 3.9]{Hala20}.
\end{enumerate}
\end{teo}

Chmutov, Glick and Pylyavskyy established in \cite[Theorem 1.8]{CGP16} an equivalence between the relations \eqref{eq:bkrelations} and \eqref{eq:bkrelationextra} that are satisfied in $\mathcal{BK}_n$ and the ones of the cactus group $J_n$ (Definition \ref{def:cactus}), thus obtaining an alternative presentation for the latter via the Bender--Knuth moves. More precisely, they consider the free group generated by $t_i$, for $i\in \mathbb{Z}_{>0}$, and consider another free group generated by $q_{i,j}$, $1 \leq i \leq j$.

\begin{teo}[{\cite[Theorem 1.8]{CGP16}}]\label{thm:rel_cact_bk}
The relations 
\begin{equation}\label{eq:rel1}
t_i^2 = 1, \qquad t_i t_j = t_j t_i, \;\text{for}\; |i-j| > 1, \qquad (t_i q_{k-1} q_{k-j} q_{k-1})^2 = 1, \;\text{for}\; i+1<j<k
\end{equation}
where $q_i := t_1 (t_2 t_1) \cdots (t_i t_{i-1} \cdots t_1)$, are equivalent to the relations
\begin{equation}\label{eq:rel2}
q_{i,j}^2 = 1, \qquad q_{i,j}q_{k,l} = q_{i+j-l,i+j-k} q_{i,j}, \;\text{for}\; i \leq k < l \leq j, \qquad q_{i,j} q_{k,l} = q_{k,l} q_{i,j}, \;\text{for}\; j <k.
\end{equation}
\end{teo}

As a consequence, we have the following group isomorphism
$$ \langle t_i, i \in I | \;\text{relations in \eqref{eq:rel1}} \rangle \simeq \langle q_{i,j}, 1 \leq i < j \leq n |\; \text{relations in \eqref{eq:rel2}} \rangle = J_n.$$

\begin{obs}
In type $A$ crystals, the crystal reflection operators $\varsigma_i$ (see \cite{BumpSchi17,LaSchu81}) acting on straight-shaped Young tableaux are elements of the group $\mathcal{BK}_n$, as they can be written as $\varsigma_i := q_i t_1 q_i $, for $i \in I$. Moreover, they satisfy the relation \cite[Proposition 1.4]{BK95}
\begin{equation}\label{eq:sigma_br_6}
(\varsigma_i \varsigma_{i+1})^3 = q_i t_1 p_{i+1} t_1 (t_1 t_2)^6 t_1 p_{i+1}^{-1}t_1 q_i
\end{equation}
for $i \in [n-2]$, where $p_i := t_1 (t_2 t_1) \cdots (t_i t_{i-1} \cdots t_1)$. Thus, the relation $(t_1 t_2)^6 = 1$ is equivalent to the braid relation relations $(\varsigma_i \varsigma_{i+1})^3 = 1$, for all $1 \leq i \leq n-2$. It is known that the operators $\varsigma_i$ define an action of the symmetric group on a type $A$ crystal (for instance, see \cite[Theorem 11.14]{BumpSchi17}). We shall see in Proposition \ref{prop:braid_t6} that the shifted crystal reflection operators $\sigma_i$ satisfy a similar identity, but since the braid relations do not need to be satisfied by $\sigma_i$ (see Example \ref{ex:braid}), then the relation $(\mathsf{t}_1 \mathsf{t}_2)^6 =1$ does not need to hold as well (see Example \ref{ex:t1t26}).
\end{obs}

\subsection{A shifted Berenstein--Kirillov group and a cactus group action}
Motivated by the definition of the Berenstein--Kirillov group, we consider $\mathcal{SBK}$ to be the free group generated by the shifted Bender--Knuth involutions $\mathsf{t}_i$, for $i > 0$,  modulo the relations they satisfy when acting on shifted semistandard tableaux of any shape. We call it the \emph{shifted Berenstein--Kirillov group}, and consider its subgroup $\mathcal{SBK}_n$ generated by $\mathsf{t}_1, \ldots, \mathsf{t}_{n-1}$. From Proposition \ref{prop:ti_props}, we know that the relations $\mathsf{t}_i^2 = 1$ and $\mathsf{t}_i \mathsf{t}_j = \mathsf{t}_j \mathsf{t}_i$, for $|i-j|>1$, hold in $\mathcal{SBK}$. Recall from \eqref{eq:def_qi}, that
$$\mathsf{q}_i := \mathsf{t}_1 (\mathsf{t}_2 \mathsf{t}_1) \cdots (\mathsf{t}_i \mathsf{t}_{i-1} \cdots \mathsf{t}_1)$$
for $i\geq 1$. From Proposition \ref{prop:evac_ti}, the shifted evacuation restricted to the primed interval $[1,i+1]'$, on straight-shaped shifted tableaux, is an element of $\mathcal{SBK}$, being computed by $\mathsf{q}_{i}$. In particular, the operators $\mathsf{q}_i$ are involutions. We will show in Proposition \ref{prop:rels_SBK} that the relation $(\mathsf{t}_i\mathsf{q}_{j,k})^2 = 1$, for $2 \leq i+1 < j < k \leq n$, which is the shifted version of \eqref{eq:bkrelationextra} (see \cite[Theorem 1.6]{CGP16}), also holds in $\mathcal{SBK}$.

Recall from Definition \ref{def:prom} that $\mathsf{p}_i = \mathsf{t_1} (\mathsf{t}_2 \mathsf{t}_1) \cdots (\mathsf{t}_{i} \mathsf{t}_{i-1} \cdots \mathsf{t}_1)$ and the promotion operators $\mathsf{p}_i$ are elements of $\mathcal{SBK}$. By Corollary \ref{cor:sigma_prom}, the shifted crystal reflection operators $\sigma_i$ are also elements of $\mathcal{SBK}$, for $i \geq 1$, as they can be written as $\sigma_i = \mathsf{q}_i \mathsf{t}_1 \mathsf{q}_i$. Following a similar computation in \cite[Proposition 1.4]{BK95}, we show that they satisfy the following identity.

\begin{prop}\label{prop:braid_t6}
Let $i \in [n-2]$ and $m \in \mathbb{N}$. Then, writing $\sigma_i = \mathsf{q}_i \mathsf{t}_1 \mathsf{q}_i$, we have
\begin{equation}
(\sigma_i \sigma_{i+1})^m = \mathsf{q}_i \mathsf{t}_1 \mathsf{p}_{i+1} \mathsf{t}_1 (\mathsf{t}_1 \mathsf{t}_2)^{2m} \mathsf{t}_1 \mathsf{p}_{i+1}^{-1} \mathsf{t}_1 \mathsf{q}_i.
\end{equation}
Thus, in particular we have
\begin{equation}\label{eq:braidt12}
(\sigma_i \sigma_{i+1})^3 = \mathsf{q}_i \mathsf{t}_1 \mathsf{p}_{i+1} \mathsf{t}_1 (\mathsf{t}_1 \mathsf{t}_2)^6 \mathsf{t}_1 \mathsf{p}_{i+1}^{-1} \mathsf{t}_1 \mathsf{q}_i.
\end{equation}
\end{prop}

\begin{proof}
By Corollary \ref{cor:sigma_prom} and the fact that $\mathsf{q}_i$ is an involution, we have
\begin{align*}
(\sigma_i \sigma_{i+1})^m &= (\mathsf{q}_i \mathsf{t}_1 \mathsf{q}_i \mathsf{q}_{i+1} \mathsf{t}_1 \mathsf{q}_{i+1})^m\\
&= (\mathsf{q}_i \mathsf{t}_1 \mathsf{q}_i \mathsf{q}_{i} \mathsf{p}_{i+1} \mathsf{t}_1 \mathsf{q}_{i+1})^m\\
&= (\mathsf{q}_i \mathsf{t}_1  \mathsf{p}_{i+1} \mathsf{t}_1 \mathsf{q}_{i+1})^m\\
&= (\mathsf{q}_i \mathsf{t}_1  \mathsf{p}_{i+1} \mathsf{t}_1 \mathsf{p}_{i+1}^{-1} \mathsf{q}_{i})^m\\
&= \mathsf{q}_i (\mathsf{t}_1  \mathsf{p}_{i+1} \mathsf{t}_1 \mathsf{p}_{i+1}^{-1})^m \mathsf{q}_i\\
&= \mathsf{q}_i (\mathsf{t}_1  \mathsf{p}_{i+1} \mathsf{t}_1 \mathsf{p}_{i+1}^{-1})^m (\mathsf{t}_1 \mathsf{p}_{i+1} \mathsf{t}_1 \mathsf{t}_1 \mathsf{p}_{i+1}^{-1} \mathsf{t}_1) \mathsf{q}_i\\
&= \mathsf{q}_i \mathsf{t}_1 \mathsf{p}_{i+1} \mathsf{t}_1 (\mathsf{p}_{i+1}^{-1} \mathsf{t}_1 \mathsf{p}_{i+1} \mathsf{t}_1)^m \mathsf{t}_1 \mathsf{p}_{i+1}^{-1} \mathsf{t}_1 \mathsf{q}_i.
\end{align*}
To conclude the proof, we claim that $\mathsf{p}_{i+1}^{-1} \mathsf{t}_1 \mathsf{p}_{i+1} \mathsf{t}_1 = (\mathsf{t}_1 \mathsf{t}_2)^2$, for any $i \geq 1$. The proof is done by induction. For $i=1$, we have 
$$\mathsf{p}_2^{-1} \mathsf{t}_1 \mathsf{p}_2 \mathsf{t}_1 = (\mathsf{t}_1 \mathsf{t}_2) \mathsf{t}_1 (\mathsf{t}_2 \mathsf{t}_1) \mathsf{t}_1 = (\mathsf{t}_1 \mathsf{t}_2)^2.$$
For the induction step, assume the claim holds for some $i \geq 1$. Then, due to Proposition \ref{prop:ti_props}, as $|(i+2)+1| > 1$ , we have
\begin{align*}
\mathsf{p}_{i+2}^{-1} \mathsf{t}_1 \mathsf{p}_{i+2} \mathsf{t}_1 & = (\mathsf{t}_1 \cdots \mathsf{t}_{i+1} \mathsf{t}_{i+2}) \mathsf{t}_1 (\mathsf{t}_{i+2} \mathsf{t}_{i+1} \cdots \mathsf{t}_1) \mathsf{t}_1\\
&= \mathsf{p}_{i+1}^{-1} \mathsf{t}_{i+2} \mathsf{t}_1 \mathsf{t}_{i+2} \mathsf{p}_{i+1} \mathsf{t}_1\\
&= \mathsf{p}_{i+1}^{-1} \mathsf{t}_1  (\mathsf{t}_{i+2})^2 \mathsf{p}_{i+1} \mathsf{t}_1\\
&= \mathsf{p}_{i+1}^{-1} \mathsf{t}_1  \mathsf{p}_{i+1} = (\mathsf{t}_1 \mathsf{t}_2)^2.
\end{align*}
\end{proof}

Recall that the braid relations $(\sigma_i \sigma_{i+1})^3=1$, for $1 \leq i \leq n-2$, for the shifted crystal reflection operators do not need to hold (see Example \ref{ex:braid}). Thus, \eqref{eq:braidt12} ensures that the relation $(\mathsf{t}_1 \mathsf{t}_2)^6=1$ does not need to hold either, as illustrated in Example \ref{ex:t1t26}. This will habe no effects on our results, as none of the cactus group relations is equivalent to this one \cite[Remark 1.9]{CGP16}. 

\begin{ex}\label{ex:t1t26}
Let $T = \begin{ytableau}
1 & 1 & 2' & 2 & 3\\
\none & 2 & 3' & 3\\
\none & \none & 3
\end{ytableau}$. Then, we have

\begin{align*}
T= \begin{ytableau}
1 & 1 & 2' & 2 & 3\\
\none & 2 & 3' & 3\\
\none & \none & 3
\end{ytableau} &\xrightarrow{\mathsf{t}_2}
\begin{ytableau}
1 & 1 & 2 & 2& 2\\
\none & 2 & 3' & 3\\
\none & \none & 3
\end{ytableau} \xrightarrow{\mathsf{t}_1}
\begin{ytableau}
1 & 1 & 1 & 1& 2'\\
\none & 2 & 3' & 3\\
\none & \none & 3
\end{ytableau} \xrightarrow{\mathsf{t}_2}
\begin{ytableau}
1 & 1 & 1 & 1& 3'\\
\none & 2 & 2 & 2\\
\none & \none & 3
\end{ytableau}\\
&\xrightarrow{\mathsf{t}_1}
\begin{ytableau}
1 & 1 & 1 & 2'& 3'\\
\none & 2 & 2 & 2\\
\none & \none & 3
\end{ytableau} \xrightarrow{\mathsf{t}_2}
\begin{ytableau}
1 & 1 & 1 & 2'& 3'\\
\none & 2 & 3' & 3\\
\none & \none & 3
\end{ytableau} \xrightarrow{\mathsf{t}_1}
\begin{ytableau}
1 & 1 & 2 & 2 & 3'\\
\none & 2 & 3' & 3\\
\none & \none & 3
\end{ytableau}\\
&\xrightarrow{\mathsf{t}_2}
\begin{ytableau}
1 & 1 & 2' & 2 & 3\\
\none & 2 & 2 & 3\\
\none & \none & 3
\end{ytableau} \xrightarrow{\mathsf{t}_1}
\begin{ytableau}
1 & 1 & 1 & 1& 3\\
\none & 2 & 2 & 3\\
\none & \none & 3
\end{ytableau} \xrightarrow{\mathsf{t}_2}
\begin{ytableau}
1 & 1 & 1 & 1& 2\\
\none & 2 & 2 & 3'\\
\none & \none & 3
\end{ytableau}\\
&\xrightarrow{\mathsf{t}_1}
\begin{ytableau}
1 & 1 & 1 & 2'& 2\\
\none & 2 & 2 & 3'\\
\none & \none & 3
\end{ytableau} \xrightarrow{\mathsf{t}_2}
\begin{ytableau}
1 & 1 & 1 & 3'& 3\\
\none & 2 & 2 & 3\\
\none & \none & 3
\end{ytableau} \xrightarrow{\mathsf{t}_1}
\begin{ytableau}
1 & 1 & 2' & 3'& 3\\
\none & 2 & 2 & 3\\
\none & \none & 3
\end{ytableau}=(\mathsf{t}_1 \mathsf{t}_2)^6 (T) \neq T.
\end{align*}
\end{ex}

\begin{prop}\label{prop:ti_si}
As elements of $\mathcal{SBK}$, we have
$$\mathsf{t}_1 = \mathsf{q}_1, \qquad \mathsf{t}_i = \mathsf{q}_{i-1} \mathsf{q}_{i} \mathsf{q}_{i-1} \mathsf{q}_{i-2}, \,\text{for}\; i\geq 2,$$
considering $\mathsf{q}_0 := 1$. Consequently, the elements $\mathsf{q}_1, \ldots, \mathsf{q}_{n-1}$ are generators of $\mathcal{SBK}_n$.
\end{prop}

\begin{proof}
The first identity is a direct consequence of the definition of $\mathsf{q}_1$. For the second one, we note that by definition of the promotion operators, we have $\mathsf{p}_i = \mathsf{t}_i \mathsf{p}_{i-1}$, for $i\geq 2$, and thus $\mathsf{t}_i = \mathsf{p}_i \mathsf{p}_{i-1}^{-1}$. It also follows from the definition that, for $i \geq 2$, $\mathsf{q}_{i} = \mathsf{q}_{i-1} \mathsf{p}_i$, which is equivalent to $\mathsf{p}_i = \mathsf{q}_{i-1} \mathsf{q}_i$, as $\mathsf{q}_j$ are involutions, for any $j \geq 1$. Then, we have
$$\mathsf{t}_i = \mathsf{p}_i \mathsf{p}_{i-1}^{-1} = \mathsf{q}_{i-1} \mathsf{q}_{i} (\mathsf{q}_{i-2} \mathsf{q}_{i-1})^{-1} =  \mathsf{q}_{i-1} \mathsf{q}_{i} \mathsf{q}_{i-1} \mathsf{q}_{i-2}.$$
\end{proof}

We denote, for $1 \leq i < j \leq n$,
\begin{equation}\label{eq:qij}
\mathsf{q}_{i,j} := \mathsf{q}_{j-1} \mathsf{q}_{j-i} \mathsf{q}_{j-1}.
\end{equation}
In particular, we have $\mathsf{q}_{i} = \mathsf{q}_{1,i+1}$. Corollary \ref{cor:eta_cactus} ensures that $\mathsf{q}_{i,j}$ is realized by $\eta_{i,j}=\mathsf{evac}_{j} \mathsf{evac}_{j-i+1} \mathsf{evac}_{j}$ when acting on straight-shaped shifted tableaux. As an element of the $\mathcal{SBK}$ group, the shifted  Schützenberger involution restricted to the alphabet $[i,j]'$, on straight-shaped shifted tableaux, is computed by $\mathsf{q}_{i,j}$, for $1 \leq i < j \leq n$. In general, $\mathsf{q}_{i,j}$ is not realized by $\eta_{i,j}$ when acting on skew shapes (see Example \ref{ex:evac_tilde_neq}). 

As a consequence of the internal action of the cactus group in $\mathsf{ShST}(\nu,n)$ (Theorem \ref{teo:cact_evaci}), we have the following result.

\begin{teo}\label{teo:cact_sbk}
The following map is an epimorphism, for $1 \leq i < j \leq n$. 
\begin{align*}
\psi: J_n & \longrightarrow \mathcal{SBK}_n \\
s_{i,j} & \longmapsto \mathsf{q}_{i,j}.
\end{align*}
Hence $\mathcal{SBK}_n$ is isomorphic to $J_n / \ker \psi$.
\end{teo}

\begin{proof}
From Proposition \ref{prop:ti_si}, $\mathcal{SBK}_n$ is generated by $\mathsf{q}_i$, for $i \in I$. Then, considering that $\mathsf{q}_i = \mathsf{q}_{1,i}$ we have $\mathsf{q}_i = \psi (s_{1,i})$, and thus $\psi$ is a surjection. Since $\mathsf{q}_i = \mathsf{evac}_{i+1}$ for straight-shaped tableaux, Theorem \ref{teo:cact_evaci} then ensures that $\psi$ is an homomorphism. Thus, $\mathcal{SBK}_n$ is isomorphic to the quotient of $J_n$ by $\ker \psi$.
\end{proof}

As a consequence, we are able to recover the relation \eqref{eq:bkrelationextra} for the shifted operators. The known relations that are satisfied in $\mathcal{SBK}$ are listed below.

\begin{prop}\label{prop:rels_SBK}
The following relations hold in $\mathcal{SBK}_n$:
	\begin{enumerate}
	\item $\mathsf{t}_i^2 = 1$, for $i \in I$.
	\item $\mathsf{t}_i \mathsf{t}_j = \mathsf{t}_j \mathsf{t}_i$, for $|i-j| >1$.
	\item $(\mathsf{t}_i\mathsf{q}_{j,k})^2 = 1$, for $2 \leq i+1 < j < k \leq n$.
	\end{enumerate}
\end{prop}

\begin{proof}
The first two relations correspond to Proposition \ref{prop:ti_props}. By Theorem \ref{teo:cact_sbk}, the action of the operator $\mathsf{q}_{j,k}$ on straight-shaped shifted tableaux defines an action of the cactus group. Thus, since $[1,2] \cap [j,k] = \emptyset$, we have $(\mathsf{t}_i\mathsf{q}_{j,k})^2 = (\mathsf{q}_{1,2} \mathsf{q}_{j,k})^2 = 1$.
\end{proof}

Theorem \ref{thm:rel_cact_bk} is stated and proved in terms of group generators satisfying the relations in Proposition \ref{prop:rels_SBK}, and do not depend on specific operators. This ensures that the relations in Proposition \ref{prop:rels_SBK} are equivalent to
$$\mathsf{q}_{i,j}^2 = 1, \qquad \mathsf{q}_{i,j} \mathsf{q}_{k,l} \mathsf{q}_{i,j} = \mathsf{q}_{i+j-l,i+j-k}, \;\text{for}\; i \leq k < l \leq j, \qquad \mathsf{q}_{i,j} \mathsf{q}_{k,l} = \mathsf{q}_{k,l} \mathsf{q}_{i,j}, \;\text{for}\; j <k.$$
Then, we have the following alternative presentation for the cactus group, via the shifted Bender--Knuth moves:
\begin{equation}\label{eq:cactus_alt_prst}
\begin{split}
J_n = \langle \mathsf{t}_{i},\; i \in I \; | \; \mathsf{t}_i^2 = 1, \mathsf{t}_i \mathsf{t}_j= \mathsf{t}_j \mathsf{t}_i, \;\text{if}\; |i-j| >1,\\ (\mathsf{t}_i \mathsf{q}_{k-1} \mathsf{q}_{k-j} \mathsf{q}_{k-1})^2 = 1, \;\text{for} \; i+1 < j < k \rangle.
\end{split}
\end{equation}

\subsection{Final remarks and further questions}

Berenstein and Kirillov have also showed in \cite{BK16} that the Berenstein--Kirillov group is isomorphic to a quotient of the cactus group. This was done independently of the work of Chmutov, Glick and Pylyavskyy \cite{CGP16}. Moreover, the Berenstein--Kirillov group $\widetilde{\mathcal{BK}}_n$ considered in \cite{BK16} differs from the one considered here, being defined as the free group generated by $t_1, \ldots, t_{n-1}$, subject only to the relations
\begin{enumerate}
\item $(t_i)^2 = 1$,
\item $(t_1 t_2) ^6 =1$,
\item $t_i t_j = t_j t_i$, for $|i-j|>1$,
\item $(t_1 q_i)^4=1$, for $i>2$.
\end{enumerate}

Thus, besides concluding that $\widetilde{\mathcal{BK}}_n$ is isomorphic to a quotient $J_n / \ker \tilde{\phi}$ of the cactus group, where $\tilde{\phi}$ is an epimorphism from $J_n$ to $\widetilde{\mathcal{BK}}_n$, the quotient is completely described as $\ker \tilde{\phi}$ is $\{(s_{1,2} s_{1,3})^{6m}, m \in \mathbb{Z}\}$, the normal subgroup of $J_n$ generated by $(s_{1,2} s_{1,3})^6$. For the Berenstein--Kirillov group presented in 
\cite{CGP16}, considering $\phi: s_{i,j} \mapsto q_{i,j}$, one concludes that $\ker \phi$, must contain $\{(s_{1,2} s_{1,3})^{6m}, m \in \mathbb{Z}\}$, as it follows from the relation \eqref{eq:bkrelations_special} holding on $\mathcal{BK}_n$ that is not equivalent to any relation of the cactus group. But as a comprehensive set of relations for $\mathcal{BK}_n$ is not known, it could be the case that there would be other relations not following from the cactus group.

For the shifted case, we have seen that the relation $(\mathsf{t}_1 \mathsf{t}_2)^6=1$ does not need to hold. However, fixing a shifted tableau crystal $\mathsf{ShST}(\nu,n)$, which is finite, there must exist some $m > 6$ such that $(\mathsf{t}_1 \mathsf{t}_2)^m (T) = T$, for all $T \in \mathsf{ShST}(\nu,n)$. Previous computations suggested that if there exists $r \in \mathbb{Z}_{>0}$ such that $(\sigma_1 \sigma_2)^r = 1$, then $r \geq 90$ 
\cite[Appendix A]{Ro20b}. Thus, if there exists $m$ such that $(\mathsf{t}_1 \mathsf{t}_2)^m = 1$, for any shape $\nu$, Proposition \ref{prop:braid_t6} implies that $m \geq 180$. We do not know whether there exists such $m$ valid for any shifted tableau crystal. Thus, considering the epimorphism $\psi$ from $J_n$ to $\mathcal{SBK}_n$ of Theorem \ref{teo:cact_sbk}, an explicit element of the kernel $\ker \psi$ is not known, although we can state that the kernel does not contain $\{(s_{1,2} s_{1,3})^{6m}, m \in \mathbb{Z}\}$. Proposition \ref{prop:braid_t6} shows that the study of $\ker \psi$ is closely related to the group generated by the shifted crystal reflection operators $\sigma_i$ on $\mathsf{ShST}(\nu,n)$. For future work, it would also be interesting to find whether there are other relations that are satisfied in $\mathcal{SBK}_n$ that do not follow from the cactus group relations. We refer to 
\cite[Problem 1.7]{BGL19} for similar problems.

\section{Shifted growth diagrams}\label{sec:gd}
In this section we will recall the notion of growth diagrams for shifted tableaux, following the work of Thomas and Yong \cite{TY16} for shifted standard tableaux. We present alternative formulations for some of the algorithms presented before in the same fashion as \cite{CGP16}, namely, the shifted \textit{jeu de taquin}, tableau switching, evacuation and its restrictions.
Using the semistandardization process of Pechenik and Yong \cite{PY17}, these algorithms may be applied to shifted semistandard tableaux. We then provide a proof for Theorem \ref{thm:cactusaction} \cite[Theorem 5.7]{Ro20b} using growth diagrams for shifted tableaux. This proof relies on the algorithmic description of partial Schützenberger involutions as the restrictions of the shifted reversal to primed intervals, while the one in \cite[Theorem 5.7]{Ro20b} uses the description in terms of the shifted tableau crystal operators (see \cite[Lemma 5.4]{Ro20b}).

We remark that, unlike the case for semistandard growth diagrams for Young tableaux introduced by Chmutov, Glick and Pylyavskyy \cite[Section 3]{CGP16}, shifted semistandard tableau, filled in a primed alphabet, are not encoded by a sequence of shape chains, as both each entry $i$ and $i'$ contribute the same to the weight.

\begin{defin}
Let $T$ be a shifted standard tableau of shape $\lambda/\mu$. Its \emph{shape chain} is the saturated chain of strict partitions
$$\mu = \lambda^{(0)} \subseteq \lambda^{(1)} \subseteq \cdots \subseteq \lambda^{(k)} = \lambda$$
where $k = |\lambda| - |\mu|$ and $\lambda^{(i)}$ is the shape of $T^1 \sqcup \cdots \sqcup T^i$, for $i \geq 1$. Since $T$ is standard, each shape $\lambda^{(i)}$ has exactly one more box than $\lambda^{(i-1)}$.
\end{defin}

The shape chain uniquely represents $T$. Since $T$ is standard, $\lambda^{(i)}$ differs from $\lambda^{(i-1)}$ by a single box. If $T$ is straight-shaped, then the chain starts with $\mu = \emptyset$. 
Moreover, the sub-chain 
$$\lambda^{(i-1)} \subseteq \lambda^{(i)} \subseteq \cdots \subseteq \lambda^{(j)},$$
for $i \geq j$, encodes the tableau $T^{i,j}$. More precisely, it encodes the shifted standard tableau with the same shape as $T^{i,j}$, filled by the letters $\{1, \ldots, i-j+1\}$, but one may consider a relabelling of those letters, in order to have $T^{i,j}$.

\begin{ex}\label{ex:sh_chain}
Consider the following shifted standard tableau of shape $(5,3,1)/(3,1)$, 
$$T = \begin{ytableau}
{} & {} & {} & 1 & 3\\
\none & {} & 2  & 5\\
\none & \none & 4
\end{ytableau}$$
which is represented by
$$(3,1) \subseteq (4,1) \subseteq (4,2) \subseteq (5,2) \subseteq (5,2,1) \subseteq (5,3,1).$$
\end{ex}

Given a skew-shaped standard tableau of shape $\lambda/\mu$, a sequence of slides to rectify it may be encoded by a straight-shaped standard tableau of shape $\mu$, where the slides are performed starting on the inner corner corresponding to the largest entry. 

\begin{ex}\label{ex:sh_rect_seq}
Considering the tableau of the previous example, we have the following rectification sequences (corresponding to the straight-shaped tableaux in the inner shape of $T$, with gray letters):

\begin{align*}
T &= \begin{ytableau}
{\color{lblue}1} & {\color{lblue}2} & {\color{lblue}3} & 1 & 3\\
\none & {\color{lblue}4} & 2  & 5\\
\none & \none & 4
\end{ytableau} \longrightarrow
\begin{ytableau}
{\color{lblue}1} & {\color{lblue}2} & {\color{lblue}3}& 1 & 3\\
\none & 2 & 4  & 5\\
\end{ytableau} \longrightarrow
\begin{ytableau}
{\color{lblue}1} & {\color{lblue}2} & 1 & 3\\
\none & 2 & 4  & 5\\
\end{ytableau} \longrightarrow
\begin{ytableau}
{\color{lblue}1} & 1 & 3 & 5\\
\none & 2 & 4\\
\end{ytableau} \longrightarrow
\begin{ytableau}
1 & 2 & 3 & 5\\
\none & 4\\
\end{ytableau} = \mathsf{rect}(T) \\
T &= \begin{ytableau}
{\color{lblue}1} & {\color{lblue}2} & {\color{lblue}4} & 1 & 3\\
\none & {\color{lblue}3} & 2  & 5\\
\none & \none & 4
\end{ytableau} \longrightarrow
\begin{ytableau}
{\color{lblue}1} & {\color{lblue}2} & 1 & 3\\
\none & {\color{lblue}3} & 2  & 5\\
\none & \none & 4
\end{ytableau} \longrightarrow
\begin{ytableau}
{\color{lblue}1}  & {\color{lblue}2} & 1 & 3\\
\none & 2 & 4 & 5
\end{ytableau} \longrightarrow
\begin{ytableau}
{\color{lblue}1} & 1 & 3 & 5\\
\none & 2 & 4
\end{ytableau}
\longrightarrow
\begin{ytableau}
1 & 2 & 3 & 5\\
\none& 4
\end{ytableau} = \mathsf{rect}(T).
\end{align*}
The order in which the shifted \textit{jeu de taquin} slides must be performed in these two cases is encoded by the following shape chains, respectively,
$$\emptyset \subseteq (1) \subseteq (2) \subseteq (3) \subseteq (3,1),$$
$$\emptyset \subseteq (1) \subseteq (2) \subseteq (2,1) \subseteq (3,1).$$
\end{ex}

Each of the tableaux that appear in the intermediate steps of the rectification process may be encoded as well, thus we have the following definition.

\begin{defin}[{\cite[Section 2.1]{TY16}}]
A \emph{shifted rectification growth diagram} for $T$ a standard tableau of shape $\lambda/\mu$ is a table with $|\mu|$ rows and $|\lambda|-|\mu|$ columns, where the leftmost column is filled with the chain encoding a fixed rectification sequence, the top row is filled with the chain encoding $T$, and the subsequent rows are filled with the chain encoding the intermediate tableaux corresponding to the said rectification sequence. In particular, the bottom row will encode $\mathsf{rect}(T)$ and the rightmost column encodes the order in which the outer corners were vacated during the rectification process.
\end{defin}

The following table is a shifted rectification growth diagram for the tableau $T$ of Example \ref{ex:sh_chain}, fixing the first rectification sequence of Example \ref{ex:sh_rect_seq}. It is also convenient to display these diagrams under a rotation, as depicted in Figure \ref{fig:gr_sh_rect}.

\begin{center}
\begin{tabular}{|c|c|c|c|c|c|}
\hline 
$(3,1)$ & $(4,1)$ & $(4,2)$ & $(5,2)$ & $(5,2,1)$ & $(5,3,1)$ \\ 
\hline 
$(3)$ & $(4)$ & $(4,1)$  & $(5,1)$  & $(5,2)$  & $(5,3)$  \\ 
\hline 
$(2)$  & $(3)$  & $(3,1)$  & $(4,1)$  & $(4,2)$  & $(4,3)$  \\ 
\hline 
$(1)$  & $(2)$  & $(2,1)$  & $(3,1)$  & $(3,2)$  & $(4,2)$  \\ 
\hline 
$\emptyset$ & $(1)$  & $(2)$  & $(3)$  & $(3,1)$  & $(4,1)$  \\ 
\hline 
\end{tabular} 
\end{center}

\begin{figure}[h!]
\begin{center}
\includegraphics[scale=0.6]{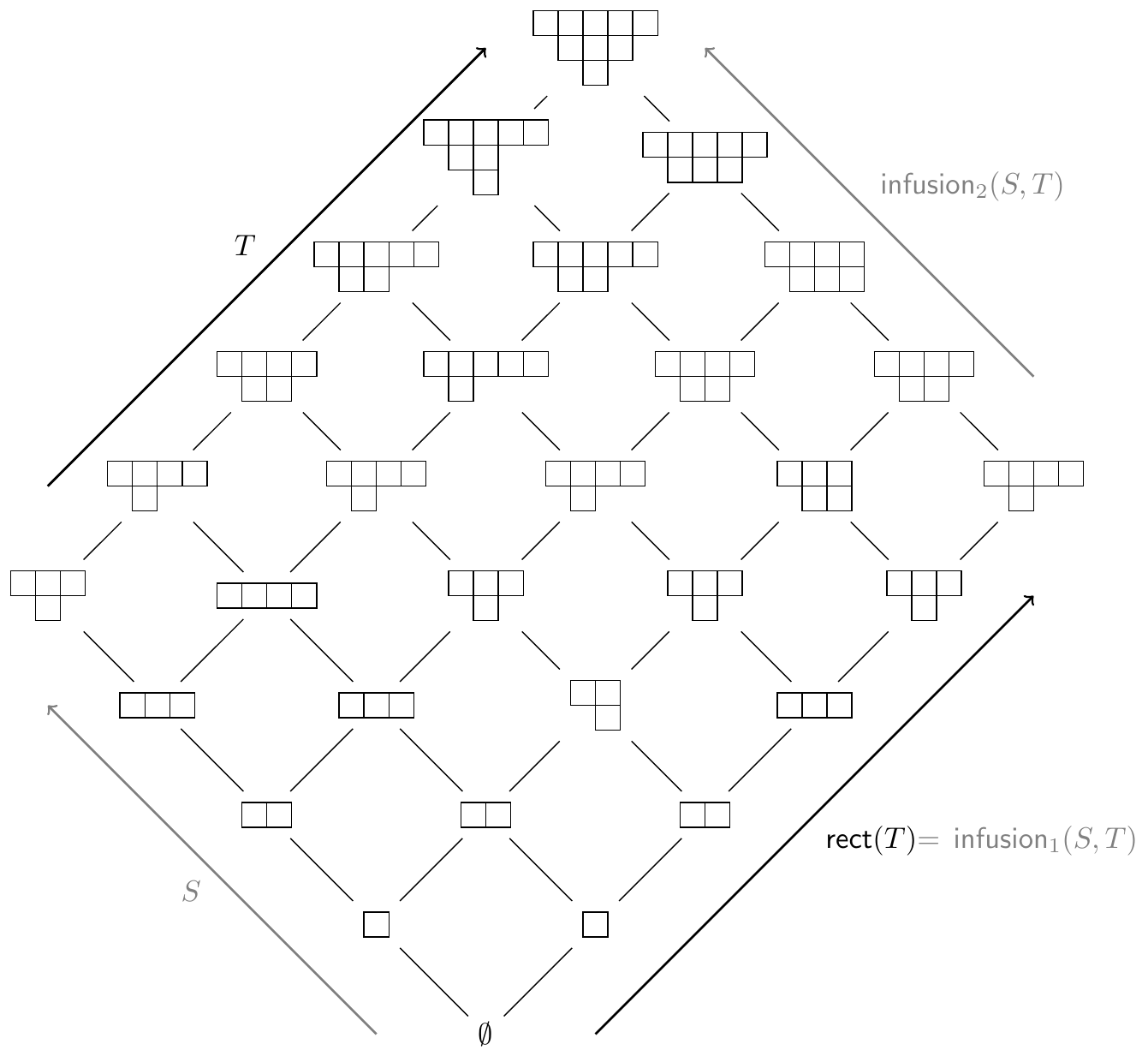}
\end{center}
\caption{A growth diagram depicting rectification of $T$, according to a rectification sequence encoded by $S$. This may also be used to compute the type $C$ infusion on a pair of shifted standard tableaux $(S,T)$.}
\label{fig:gr_sh_rect}
\end{figure}

We have seen in Lemma \ref{lem:sw_inf_std} that the shifted tableau switching and the type $C$ infusion maps agree on shifted standard tableaux, and both can be regarded as a sequence of shifted \textit{jeu de taquin} slides. Thus, given $(S,T)$ a pair of shifted standard tableaux, where $S$ is a straight-shaped shifted tableau extended by $T$, we may place $S$ and $T$ on the southwestermost and northwesternmost sides of a shifted rectification growth diagram, respectively, and then the southeasternmost and northeasternmost sides will encode $\mathsf{infusion}_1(S,T)$ and $\mathsf{infusion}_2(S,T)$, respectively. Thus, the diagram in Figure \ref{fig:gr_sh_rect} is also referred to as a \emph{shifted infusion growth diagram}.

\begin{ex}
Consider the following pair of shifted standard tableaux (these correspond to $T$ and the first rectification sequence, as in Example \ref{ex:sh_rect_seq}):
$$(S,T) = \begin{ytableau}
*(lblue)1 & *(lblue)2 & *(lblue)3 & 1 & 3\\
\none & *(lblue)4 & 2  & 5\\
\none & \none & 4
\end{ytableau}.
$$
This pair is encoded in the southwestern and northwestern edges of the diagram of Figure \ref{fig:gr_sh_rect}. Thus, we have
$$\mathsf{infusion}(S,T) = \begin{ytableau}
1 & 2 & 3 & 5 & *(lblue)3\\
\none & 4 & *(lblue)1 & *(lblue)2\\
\none & \none & *(lblue)3
\end{ytableau}.$$
The obtained pair is encoded in the southeastern and  northeastern edges of the said diagram.
\end{ex}

Similar to the growth diagrams for standard Young tableaux, which are characterized by \emph{local rules}, due to Fomin \cite[Proposition A1.2.7]{Sta99}, the shifted growth diagrams may also be described by similar rules.

\begin{teo}[{\cite[Theorem 2.1]{TY16}}]
An array of straight shapes is a shifted growth diagram if and only if for any subgrid of the form
\begin{center}
\begin{tikzpicture}[yscale=0.7,xscale=0.7]
\node at (1,0) (1) {$\nu$};
\node at (0,1) (2) {$\mu$};
\node at (1,2) (3) {$\lambda$};
\node at (2,1) (4) {$\mu'$};
\draw[-] (1) to (2) to (3) to (4) to (1);
\end{tikzpicture}
\end{center}
where $\nu \subseteq \mu \subseteq \lambda$ and $\nu \subseteq \mu' \subseteq \lambda$, the Fomin growth conditions hold:
	\begin{enumerate}
	\item $\lambda/\mu$, $\lambda/\mu'$, $\mu/\nu$ and $\mu'/\nu$ consist of a single box.
	\item If $\mu$ is the unique shape that is contained in $\lambda$ and contains $\nu$, then $\mu'=\mu$.
	\item Otherwise, there exists exactly one strict partition in the same conditions other than $\mu$, which is $\mu'$.
	\end{enumerate}
\end{teo}

These growth conditions exhibit a symmetry under a vertical reflection. Thus, vertically reflecting the diagram of Figure \ref{fig:gr_sh_rect}, we obtain 
\begin{align*}
S &= \mathsf{infusion}_1 ( \mathsf{infusion}_1(S,T), \mathsf{infusion}_2(S,T) )\\
T &= \mathsf{infusion}_2 ( \mathsf{infusion}_1(S,T), \mathsf{infusion}_2(S,T) )
\end{align*}
which explains that the infusion is an involution.

\begin{cor}[{\cite[Lemma 2.2]{TY16}}]
Let $(S,T)$ be a pair of shifted standard tableaux, with $T$ extending $S$. Then, $\mathsf{infusion} (\mathsf{infusion} (S,T)) = (S,T)$.
\end{cor}

\subsection{Evacuation and reversal}
We may obtain growth diagrams for the shifted evacuation and reversal (Section \ref{ss:evac}), by combining the previous diagrams and local rules. As in the previous sections, most results will be stated for shifted standard tableaux, and may be extended to the semistandard case using the semistandardization process \cite{PY17}. Throughout the next sections, unless otherwise stated, we consider any standard shifted tableau to have $n$ boxes, filled with the letters in $[n]$.

\begin{prop}\label{prop:gr_evac}
Let $T$ be straight-shaped shifted standard tableau. Consider an equilateral triangular array such that the shape chain encoding $T$ is placed on the northwestern edge and each vertex of the bottom edge is filled with $\emptyset$ and apply the local growth rules from left to right. Then, the shape chain on the northeastern edge corresponds to $\mathsf{evac}(T)$.
\end{prop}

\begin{proof}
Proposition \ref{prop:evac_ti} states that the evacuation of $T$ may be obtained by applying sequentially the promotion operators $\mathsf{p}_{n-1}, \mathsf{p}_{n-2}, \ldots, \mathsf{p}_{1}$ to $T$, where we recall that we are assuming that $T$ has $n$ boxes, filled in $[n]$.
By Proposition \ref{prop:prom_swi}, 
$$\mathsf{p}_i (T) = \zeta_i \mathsf{SW}_{1|2,\ldots,i+1} (T) = \zeta_i ( \mathsf{infusion}_1(T^1, T^{2,i+1}) \sqcup \mathsf{infusion}_2(T^1, T^{2,i+1})),$$
and then each of promotion operator $\mathsf{p}_i$, acting on standard tableaux, may be computed using a shifted infusion growth diagram with the southwestern edge having length 1, and northwestern edge having length $i$. Then, the diagram in Figure \ref{fig:gr_evac_prom} corresponds to sequentially concatenate, from left to right, the growth diagrams of promotion operators $\mathsf{p}_{n-1}, \mathsf{p}_{n-2}, \ldots, \mathsf{p}_{1}$, thus coinciding with $\mathsf{evac}(T)$.
\end{proof}

The symmetry of the local growth rules ensures that the diagram is symmetric under a vertical reflection. Thus, taking a shifted evacuation growth diagram on input $\mathsf{evac}(T)$, we obtain $\mathsf{evac}(\mathsf{evac}(T))=T$, thus exhibiting the fact that the shifted evacuation is an involution. 
Since $\mathsf{evac}_i (T) := \mathsf{evac}(T^{1,i}) \sqcup T^{i+1,n}$, we have the following result.

\begin{cor}\label{cor:gr_evac}
Let $T$ be a straight-shaped standard shifted tableau and let $i \in [n]$. Consider the shifted evacuation growth diagram having $T$ as input on the northestern edge and $\mathsf{evac}(T)$ on the northeastern one. 
Then:
\begin{enumerate}
\item Removing the $n-i$ rightmost northeastern edges of the diagram yields the shifted evacuation growth diagram on input $T^{1,i}$.
\item Removing the $n-j$ leftmost northwestern edges of the diagram yields the shifted evacuation growth diagram computing on input $\mathsf{rect}(T^{n-j+1,n})$.
\item Removing simultaneously the $n-i$ rightmost northeastern edges and the $n-j$ leftmost northwestern edges of the diagram, for $i \geq j$, yields the shifted evacuation growth diagram on input $\mathsf{rect}(T^{i-j+1,i})$.
\end{enumerate}
\end{cor}

\begin{ex}\label{ex:gr_std_evac}
Consider the following shifted standard tableau 
$$T = \begin{ytableau}
1 & 2 & 3 & 5 \\
\none & 4 & 6\\
\none & \none & 7
\end{ytableau}.$$
Then, the left side of the triangular array in Figure \ref{fig:gr_evac} corresponds to the shape chain of $T$, while the right side corresponds to $\mathsf{evac}(T)$. Then, we have
$$\mathsf{evac}(T) = \begin{ytableau}
1 & 2 & 3 & 7\\
\none & 4 & 5\\
\none & \none & 6
\end{ytableau}.$$
Using the same diagram we also obtain restrictions of $\mathsf{evac}$. For instance, removing the rightmost 3 northeastern edges (see the gray area in Figure \ref{fig:gr_evac}), we have
$$\mathsf{evac}_4 (T) = 
\mathsf{evac}(T^{1,4}) \sqcup T^{5,7}
= \begin{ytableau}
1 & 2 & 4 & {}\\
\none & 3 & {}\\
\none & \none & {}
\end{ytableau} \sqcup \begin{ytableau}
{} & {} & {} & 7\\
\none & {} & 5\\
\none & \none & 6
\end{ytableau} = \begin{ytableau}
1 & 2 & 4 & 7\\
\none & 3 & 5\\
\none & \none & 6
\end{ytableau}.$$
\end{ex}

\begin{figure}
\includegraphics[scale=0.6]{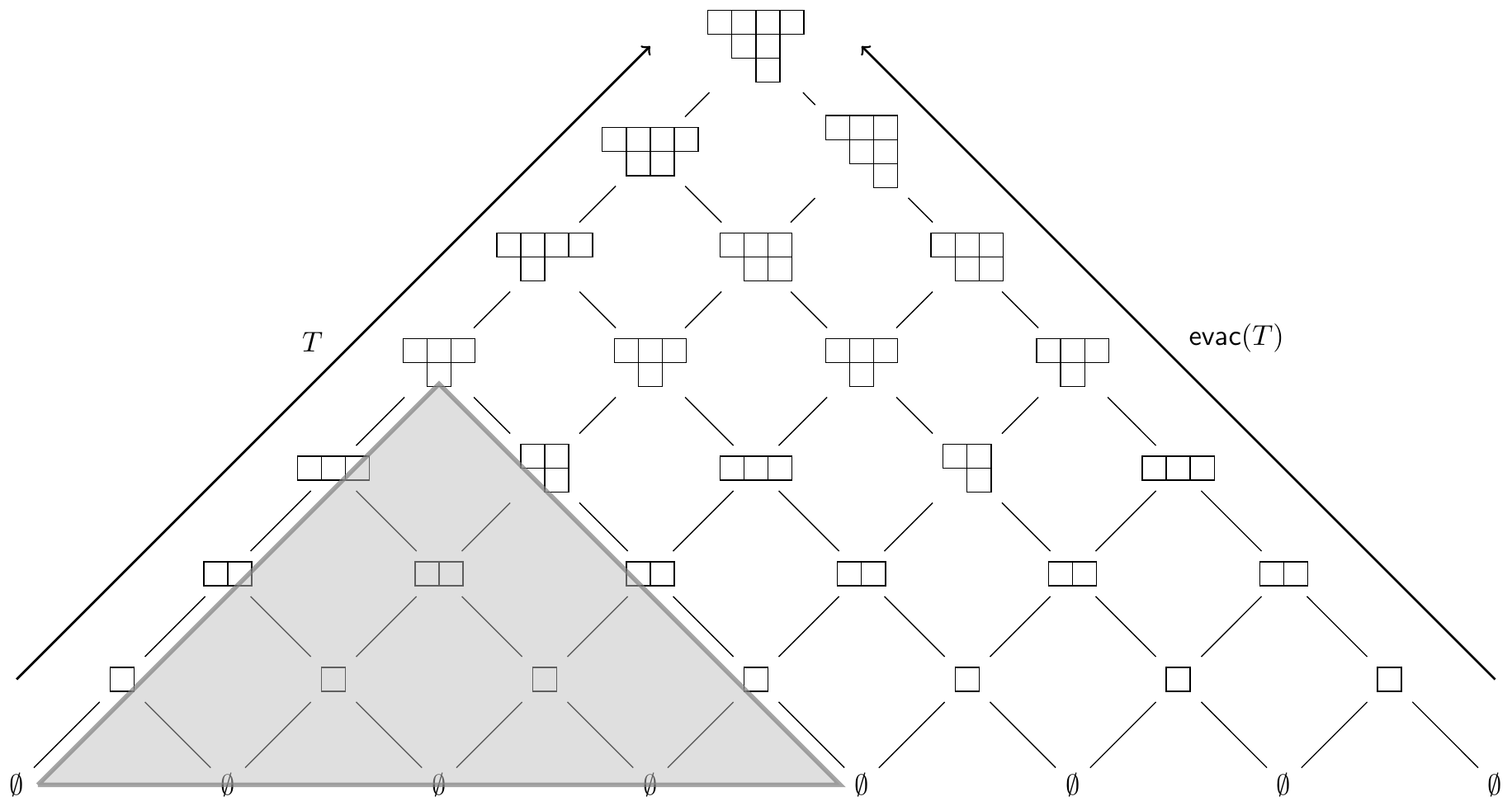}
\caption{A shifted evacuation growth diagram. The smaller gray diagram computes the restriction $\mathsf{evac}_4$, on $T^{1,4}$.}
\label{fig:gr_evac}
\end{figure}

\begin{figure}
\includegraphics[scale=0.6]{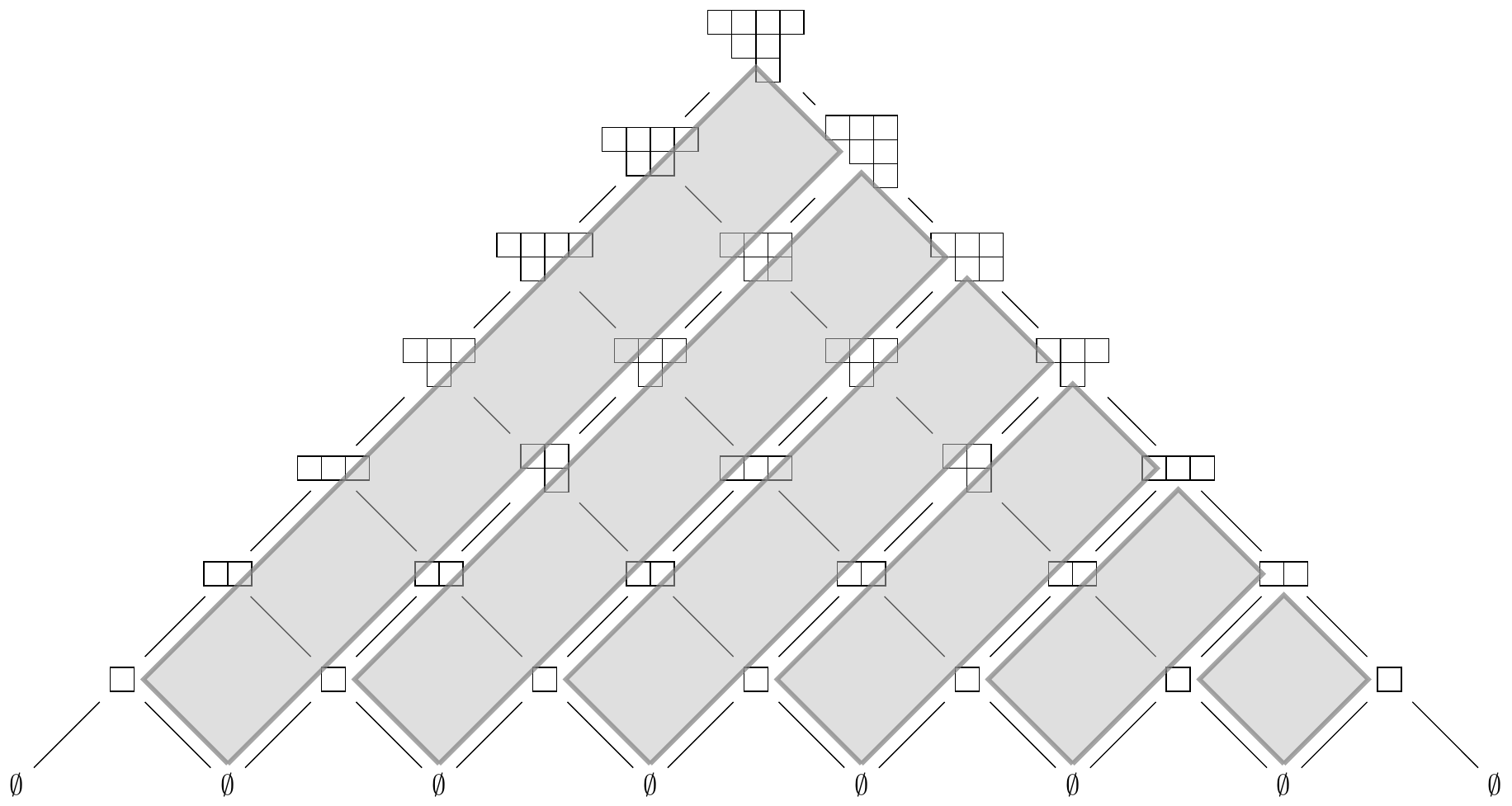}
\caption{Illustration of the shifted evacuation as a composition of promotion operators, corresponding to the gray rectangles.}
\label{fig:gr_evac_prom}
\end{figure}

The shifted \textit{jeu de taquin} and shifted tableau switching are compatible with standardization. Thus, the previous characterizations with growth diagrams may be applied to a shifted semistandard tableau $T$, by first standardizing it, then apply the standard growth diagrams, and then compute the adequate semistandardization of the obtained tableau.

\begin{ex}
Consider the following shifted semistandard tableau, which has weight $\nu = (2,2,3)$, 
$$T = \begin{ytableau}
1 & 1 & 2' & 3'\\
\none & 2 & 3'\\
\none & \none & 3
\end{ytableau}.$$
To compute $\mathsf{evac}(T)$ using growth diagrams, we consider its standardization and compute the growth diagram (see Example \ref{ex:gr_std_evac}) and then apply the $\nu'$-semistandardization, with $\nu'= \theta_{1,3} (\nu) =(3,2,2)$:
$$\mathsf{std}(T)= \begin{ytableau}
1 & 2 & 3 & 5 \\
\none & 4 & 6\\
\none & \none & 7
\end{ytableau} \overset{\mathsf{evac}}\longrightarrow 
\begin{ytableau}
1 & 2 & 3 & 7\\
\none & 4 & 5\\
\none & \none & 6
\end{ytableau}
\overset{\mathsf{sstd}_{\nu'}}\longrightarrow 
\begin{ytableau}
1 & 1 & 1 & 3\\
\none & 2 & 2\\
\none & \none & 3
\end{ytableau} = \mathsf{evac}(T).
$$
\end{ex}

Given $T$ a skew-shaped shifted standard tableau, Proposition \ref{prop:rev_sw} says that the reversal $T^e$ may be computed by filling the diagram of $\mu$ with a standard tableau $U$, applying the shifted infusion (or shifted tableau switching) to the pair $(S,T)$ obtaining $\mathsf{infusion}(S,T) = \big( \mathsf{rect}(T), \mathsf{infusion}_2(S,T) \big)$, applying the evacuation to $\mathsf{rect}(T)$, and then the shifted tableau switching again to the pair $\big( \mathsf{evac}(\mathsf{rect}(T)), \mathsf{infusion}_2(S,T) \big)$. Then, 
\begin{equation}\label{eq:sw_reversal}
T^e = \mathsf{infusion}_2 \big( \mathsf{evac}(\mathsf{rect}(T)), \mathsf{infusion}_2(U,T) \big)
\end{equation}
Thus, we have the following.

\begin{prop}\label{prop:rev_gr}
Let $T$ be a shifted standard tableau of shape $\lambda/\mu$. Consider a diagram as in Figure \ref{fig:gr_reversal}, with $T$ on the segment $[bc]$ and any standard tableau $S$ of shape $\mu$ on the segment $[ab]$\footnote{We consider a segment $[ab]$ to be directed, from $a$ to $b$. In the growth diagrams, segments are read from bottom to top.}, and such that $[dc] = [df]$. Then, the segment $[gf]$ encodes $T^e$.
\end{prop}

\begin{proof}
The diagram $[abcd]$ computes the shifted tableau switching on the pair $(S,T)$, thus $[ad]$ encodes $\mathsf{infusion}_1(S,T) = \mathsf{rect}(T)$ and $[dc]$ encodes $\mathsf{infusion}_2(S,T)$. By Proposition \ref{prop:gr_evac}, the diagram $[ade]$ computes the evacuation with input $[ad]$, thus the segment $[ed]$ corresponds to $\mathsf{evac}(\mathsf{rect}(T))$. Finally, since $[df] = [dc]$, the diagram $[edfg]$ computes the shifted tableau switching on the pair $\big(\mathsf{evac}(\mathsf{rect}(T)), \mathsf{infusion}_2(S,T)\big)$. It then follows from \eqref{eq:sw_reversal} that $[gf]$ corresponds to $T^e$.
\end{proof}

\begin{figure}[h!]
\includegraphics[scale=0.6]{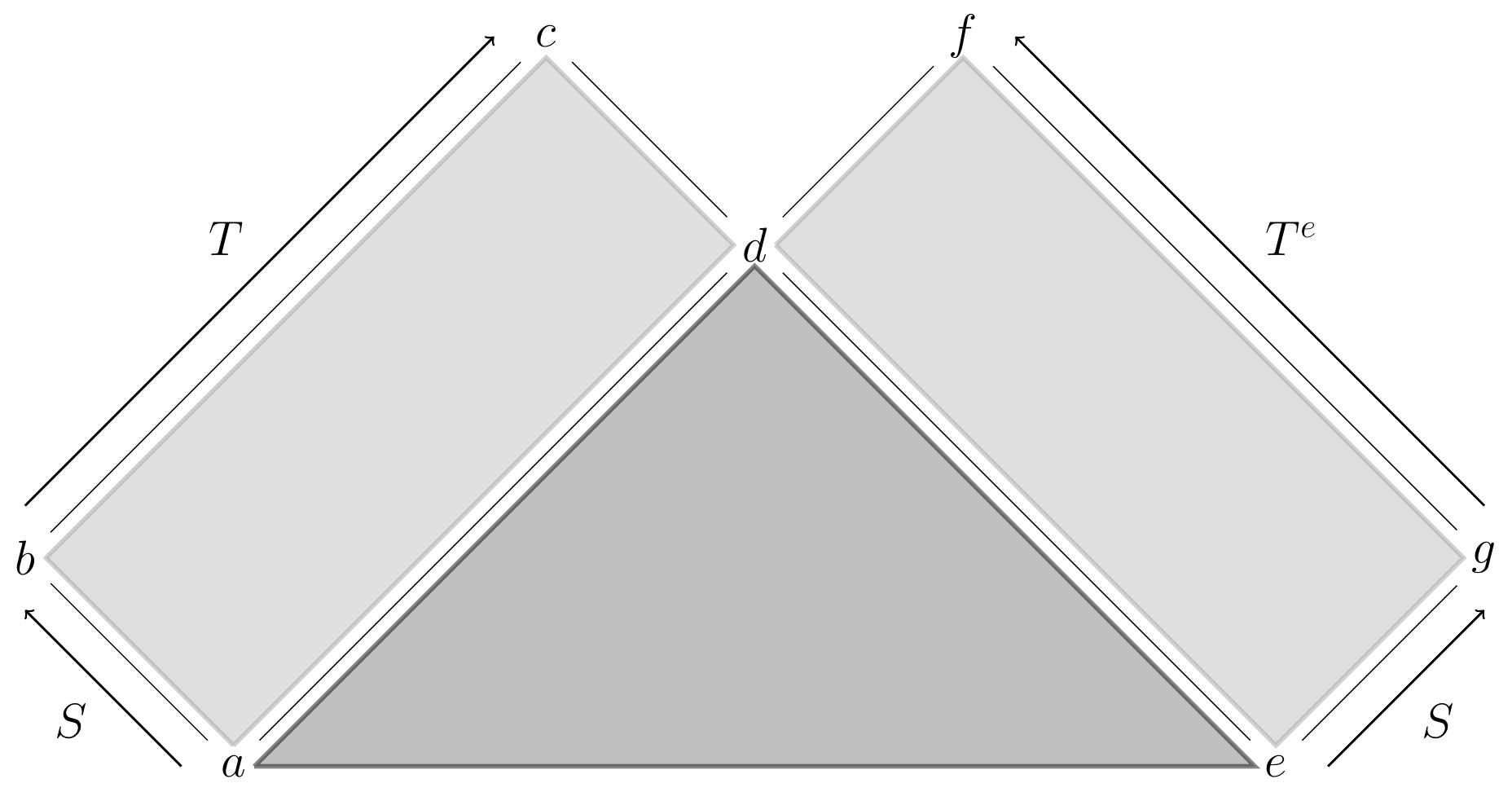}
\caption{Growth diagram to compute the shifted reversal on skew shapes. By construction, we put $[dc] = [df]$.}
\label{fig:gr_reversal}
\end{figure}

\subsection{Partial Schützenberger involutions}
Following the same approach as in \cite[Section 4.1]{CGP16}, we may use the shifted growth diagrams for rectification and evacuation to construct an array that computes $\eta_{i,j}$ for straight-shaped shifted tableaux. From \eqref{eq:qij} and Proposition \ref{prop:evac_ti}, and since $\eta_{i,j}$ is computed by $\mathsf{q}_{i,j}$ when acting on straight shapes, we have $\eta_{i,j} (T) = \mathsf{q}_{i,j} (T)$, for $T \in \mathsf{ShST} (\nu,n)$, thus the next growth diagram computes $\mathsf{q}_{i,j}$ as well. From Definition \ref{def:schu_ij}, given $T \in \mathsf{ShST}(\lambda/\mu,n)$, then $\eta_{i,j}(T) = T^{1,i-1} \sqcup \eta (T^{i,j}) \sqcup T^{j+1,n} = \eta_{i,j} (T^{1,j}) \sqcup T^{j+1,n}$.

\begin{figure}[h!]
\includegraphics[scale=0.6]{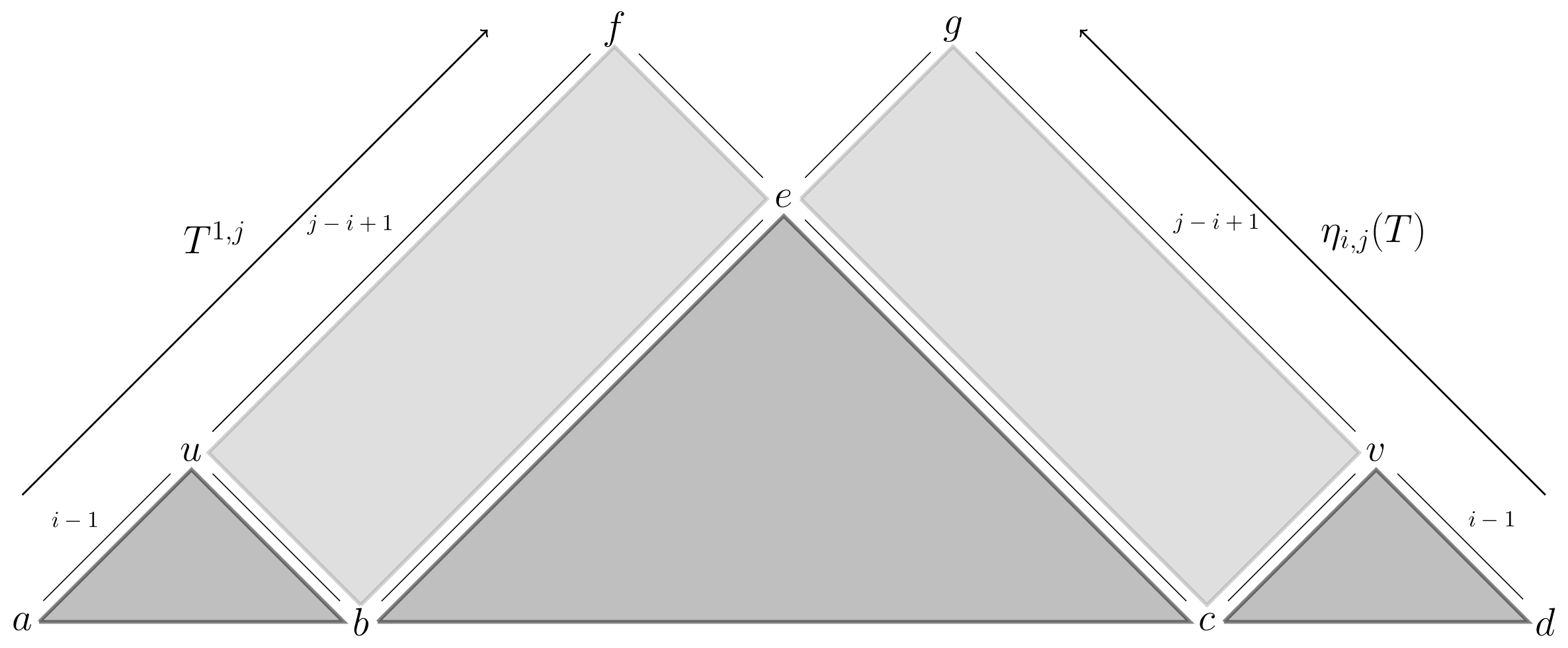}
\caption{The growth diagram to compute $\eta_{i,j}$ or $\mathsf{q}_{i,j}$ on straight-shaped tableaux \cite[Figure 6]{CGP16}. By construction, $[ef]=[eg]$.}
\label{fig:gr_eta}
\end{figure}

\begin{prop}\label{prop:gr_eta_ij}
Let $T$ be a straight-shaped shifted standard tableau filled in $[n]$ and let $1 \leq i < j \leq n$. 
Consider the diagram in Figure \ref{fig:gr_eta}, which consists, from left to right, in the growth diagrams of $\mathsf{evac}_{i-1}$, infusion, $\mathsf{evac}_{j-i+1}$, infusion, and $\mathsf{evac}_{i-1}$, and such that the segments $[ef]$ and $[eg]$ coincide. Then, if the segment $[af]$ encodes $T^{1,j}$, then the segment $[dg]$ encodes  $\eta_{i,j} (T^{1,j})$ .
\end{prop}

\begin{proof}

We will show that $[dv] = T^{1,i-1} \sqcup \eta (T^{i,j})$.
We have $[af] = T^{1,j}$, $[au] = T^{1,i-1}$ and $[uf] = T^{i,j}$, thus, by Proposition \ref{prop:gr_evac}, $[bu] = \mathsf{evac}(T^{1,i-1}) =: S$. Applying the shifted infusion growth diagram on inputs $[bu]$ and $[uf]$, we have
\begin{equation}\label{eq:gr_eg}
\begin{split}
[be] &= \mathsf{infusion}_1 (S,T^{i,j}) = \mathsf{rect}(T^{i,j})\\
[ef] &= \mathsf{infusion}_2 (S,T^{i,j}) = [eg],
\end{split}
\end{equation}
and by Corollary \ref{cor:gr_evac}, applying the shifted evacuation growth diagram, we have
\begin{equation}\label{eq:gr_ce}
[ce] = \mathsf{evac} (\mathsf{rect} (T^{i,j})).
\end{equation}
Then, applying the shifted infusion growth diagram on inputs $[eg]$ \eqref{eq:gr_eg} and $[ce]$ \eqref{eq:gr_ce}, we obtain
\begin{equation}\label{eq:gr_cv}
\begin{split}
[cv] &= \mathsf{infusion}_1 \big( \mathsf{evac}(\mathsf{rect}(T^{i,j})) , \mathsf{infusion}_2 (S,T^{i,j}) \big)\\
[vg] &= \mathsf{infusion}_2 \big( \mathsf{evac}(\mathsf{rect}(T^{i,j})) , \mathsf{infusion}_2 (S,T^{i,j}) \big).
\end{split}
\end{equation}
By \eqref{eq:sw_reversal}, we have 
\begin{equation}
[vg] = \eta (T^{i,j}).
\end{equation}
We recall that $\mathsf{infusion}_1 (S,T) = \mathsf{rect}(T)$, for any standard straight-shaped tableau $S$ extended by $T$. Considering that rectification does not depend on the chosen rectification sequence, from \eqref{eq:gr_cv}, we have
\begin{align*}
[cv] &= \mathsf{infusion}_1 \big( \mathsf{evac}(\mathsf{rect}(T^{i,j})) , \mathsf{infusion}_2 (S,T^{i,j}) \big) \\
&= \mathsf{infusion}_1 \big( \mathsf{evac}(\mathsf{infusion}_1(S,T^{i,j})) , \mathsf{infusion}_2 (S,T^{i,j}) \big) \\
&= \mathsf{rect} \big( \mathsf{infusion}_2(S,T^{i,j}) \big)\\
&= \mathsf{infusion}_1 \big( \mathsf{infusion}_1(S,T^{i,j}) , \mathsf{infusion}_2 (S,T^{i,j})\big)\\
&= \mathsf{infusion}_1 \big( \mathsf{infusion}(S,T^{i,j}) \big) = S.
\end{align*}
Finally, the shifted evacuation growth diagram ensures that 
\begin{equation}\label{eq:gr_dv}
[dv] = \mathsf{evac} (S) = \mathsf{evac}^2 (T^{1,i-1})= T^{1,i-1}.
\end{equation}
Thus, by \eqref{eq:gr_cv} and \eqref{eq:gr_dv}, we have
\begin{equation}
[dg] = T^{1,i-1} \sqcup \eta (T^{i,j}) = \eta_{i,j} (T^{1,j}). 
\end{equation}
\end{proof}

Using Proposition \ref{prop:rev_gr}, and considering that $\eta_{i,j}$ commutes with the shifted \textit{jeu de taquin}, we may generalize the previous growth diagram for skew-shaped tableaux. We remark this generalization is not valid for $\mathsf{q}_{i,j}$, as it does not commute with the shifted \textit{jeu de taquin}.

\begin{cor}\label{cor:gr_eta_ij}
Let $T$ be a skew-shaped shifted standard tableau of shape $\lambda/\mu$ and let $1 \leq i < j \leq n$. Consider the diagram on Figure \ref{fig:gr_eta_skew}, where the segment $[pr]$ encodes $T^{1,j}$, $S$ is any standard tableau  of shape $\mu$, being encoded by $[ap]$, and the segments $[er]$ and $[es]$ coincide. Then, $\eta_{i,j}(T^{1,j})$ is encoded by segment $[ws]$.
\end{cor}

\begin{proof}
Since $[pr] = T^{1,j}$ and $[ap]=S$, then 
\begin{equation}
\begin{split}
[fr] &= \mathsf{infusion}_2(S,T^{1,j})\\
[af] &= \mathsf{infusion}_1(S,T^{1,j}) = \mathsf{rect}(T^{1,j}) = (\mathsf{rect}(T))^{1,j}. 
\end{split}
\end{equation}
By Proposition \ref{prop:gr_eta_ij}, the segment $[dg]$ encodes $\eta_{i,j} ((\mathsf{rect}(T))^{1,j})$. By construction, $[er]=[es]$, and thus $[gs] = [fr] = \mathsf{infusion}_2(S,T^{1,j})$. Then, considering the shifted infusion growth diagram on inputs $[ap]$ and $[pr]$,
\begin{equation}\label{eq:ws_dw}
\begin{split}
[ws] &= \mathsf{infusion}_2 (\eta_{i,j} ((\mathsf{rect}(T))^{1,j}), \mathsf{infusion}_2(S,T^{1,j}))\\
[dw] &= \mathsf{infusion}_1 (\eta_{i,j} ((\mathsf{rect}(T))^{1,j}), \mathsf{infusion}_2(S,T^{1,j}))
\end{split}
\end{equation} 
Since $\eta_{i,j}$ commutes with the shifted \textit{jeu de taquin}, in particular we have
\begin{equation}\label{eq:eta_rect_com}
\eta_{i,j} ((\mathsf{rect}(T))^{1,j}) = \eta_{i,j} (\mathsf{rect}(T^{1,j})) = \mathsf{rect}( \eta_{i,j} (T^{1,j}))
\end{equation}
Moreover, the operator $\eta_{i,j}$ preserves shifted dual equivalence, and thus $T^{1,j}$ and $\eta_{i,j}(T^{1,j})$ are in the same shifted dual equivalence class. Then, by Proposition \ref{prop:sw_dual},
\begin{equation}\label{eq:infu_dual_S}
\mathsf{infusion}_2 (S,T^{1,j}) = \mathsf{infusion}_2 (S, \eta_{i,j} (T^{1,j})).
\end{equation}

Then, by \eqref{eq:ws_dw}, \eqref{eq:eta_rect_com} and \eqref{eq:infu_dual_S}, and since $\mathsf{infusion}$ is an involution, we have
\begin{align*}
[ws] &= \mathsf{infusion}_2 (\eta_{i,j} ((\mathsf{rect}(T))^{1,j}), \mathsf{infusion}_2(S,T^{1,j}))\\
&= \mathsf{infusion}_2 (\eta_{i,j} ((\mathsf{rect}(T))^{1,j}), \mathsf{infusion}_2 (S, \eta_{i,j} (T^{1,j})))\\
&= \mathsf{infusion}_2 (\mathsf{rect}( \eta_{i,j} (T^{1,j})), \mathsf{infusion}_2 (S, \eta_{i,j} (T^{1,j})))\\
&= \mathsf{infusion}_2 (\mathsf{infusion}_1(S, \eta_{i,j} (T^{1,j})), \mathsf{infusion}_2 (S, \eta_{i,j} (T^{1,j})))\\
&= \eta_{i,j} (T^{i,j}).
\end{align*}
\end{proof}

\begin{figure}
\includegraphics[scale=0.5]{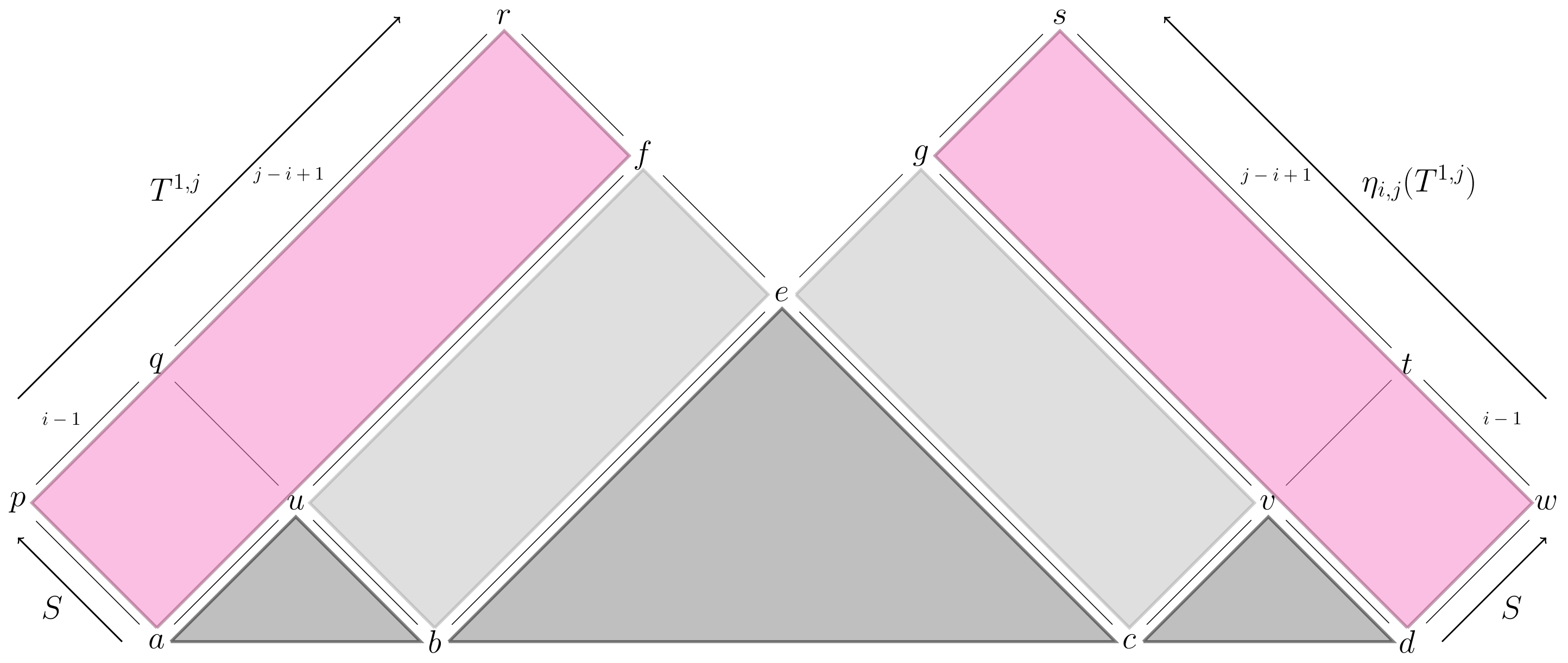}
\caption{A growth diagram to compute $\eta_{i,j}$ on shifted standard tableaux of shape $\lambda/\mu$, with $S$ being any standard tableau of shape $\mu$. By construction, $[er]=[es]$. A diagram to compute $\mathsf{q}_{i,j}$ on straight-shaped shifted standard tableaux is obtained by removing the pink sections.}
\label{fig:gr_eta_skew}
\end{figure}

As before, the growth diagrams for $\eta_{i,j}$ may me used on a shifted semistandard tableau $T$, with weight $\nu$. Since we have
\begin{equation}\label{eq:eta_std}
\mathsf{std} (\eta_{i,j} (T)) = \eta_{k,l} (\mathsf{std}(T)),
\end{equation}
where $k := \min \mathcal{P}_i (\nu)$ and $l := \max \mathcal{P}_j (\nu)$, we may standardize $T$, apply $\eta_{k,l}$, and then apply the semistandardization (see Definition \ref{def:sstd}) with respect to $\nu'$ to the obtained tableau, with $\nu' = \theta_{i,j} (\nu)$, that is,
\begin{equation}\label{eq:eta_std_sstd}
\eta_{i,j} (T) = \mathsf{sstd}_{\nu'} \big( \eta_{k,l} (\mathsf{std} (T)) \big).
\end{equation}

\begin{ex}
Consider the following shifted semistandard tableau of weight $ \nu = (2,2,3)$,
$$T=\begin{ytableau}
{} & {} & 1 & 2\\
\none & 1 & 2 & 3'\\
\none & \none & 3 & 3
\end{ytableau}.$$

To compute $\eta_{2,3} (T)$, we use the growth diagram in Figure \ref{fig:gr_eta_37} on the standardization of $T$, followed by rectification, using the rectification sequence encoded by $S= \begin{ytableau}
*(lblue)1 & *(lblue)2
\end{ytableau}$.

$$(S,T)=\begin{ytableau}
*(lblue)1 & *(lblue)2 & 1 & 2\\
\none & 1 & 2 & 3'\\
\none & \none & 3 & 3
\end{ytableau} \overset{\mathsf{std}}\longrightarrow
\begin{ytableau}
*(lblue)1 & *(lblue)2 & 2 & 4\\
\none & 1 & 3 & 5\\
\none & \none & 6 & 7
\end{ytableau} \overset{\mathsf{infusion}}\longrightarrow
\begin{ytableau}
1 & 2 & 4 & 7\\
\none & 3 & 5 & *(lblue)1\\
\none & \none & 6 & *(lblue)2
\end{ytableau} = (\mathsf{rect}( \mathsf{std} (T)), S')$$
where $S' := \mathsf{infusion}_2 (S,T)$. In the Figure \ref{fig:gr_eta_skew}, $\mathsf{rect}( \mathsf{std} (T))$ corresponds to the segment $[af]$ and $S'$ to $[fr]$. Then, by \eqref{eq:cal_p_nu}, we have
$$\mathcal{P}_2 (\nu) = \{3,4\} \qquad \mathcal{P}_3 (\nu) = \{5,6,7\}.$$
Thus, to obtain $\eta_{2,3} (T)$, we must apply $\eta_{3,7}$ to  $\mathsf{rect}(\mathsf{std}(T))$. Note that the $\eta_{3,7} \big( \mathsf{rect} (\mathsf{std} (T)) \big)$ is encoded in the segment corresponding to $[dg]$ in Figure \ref{fig:gr_eta_37}.

$$\mathsf{rect}( \mathsf{std} (T)) \begin{ytableau}
1 & 2 & 4 & 7\\
\none & 3 & 5\\
\none & \none & 6
\end{ytableau} \overset{\eta_{3,7}}\longrightarrow 
\begin{ytableau}
1 & 2 & 3 & 6\\
\none & 4 & 5\\
\none & \none & 7
\end{ytableau} = \eta_{3,7} \big( \mathsf{rect} (\mathsf{std} (T)) \big) =: T'.$$
Then, we apply the shifted infusion growth diagram (the rightmost pink region, in Figure \ref{fig:gr_eta_37}), to recover the skew shape before the rectification:
$$(T',S') = \begin{ytableau}
1 & 2 & 3 & 6\\
\none & 4 & 5 & *(lblue)1\\
\none & \none & 7 & *(lblue)2
\end{ytableau} \overset{\mathsf{infusion}}\longrightarrow
\begin{ytableau}
*(lblue)1 & *(lblue)2 & 2 & 3\\
\none & 1 & 4 & 6\\
\none & \none & 5 & 7
\end{ytableau} = \mathsf{infusion}(T',S').$$
This corresponds to the tableau of the segment $[ws]$.
Finally, we apply the semistandardization with respect to $\nu'$, where $\nu'= \theta_{2,3} (2,2,3) = (2,3,2)$:
$$ \mathsf{infusion}_2(T',S') = \begin{ytableau}
{} & {} & 2 & 3\\
\none & 1 & 4 & 6\\
\none & \none & 5 & 7
\end{ytableau} \overset{\mathsf{sstd}_{\nu'}}\longrightarrow
\begin{ytableau}
{} & {} & 1 & 2'\\
\none & 1 & 2' & 3'\\
\none & \none & 2 & 3
\end{ytableau} = \eta_{2,3} (T). 
$$
\end{ex}

\begin{figure}[H]
\includegraphics[scale=0.5]{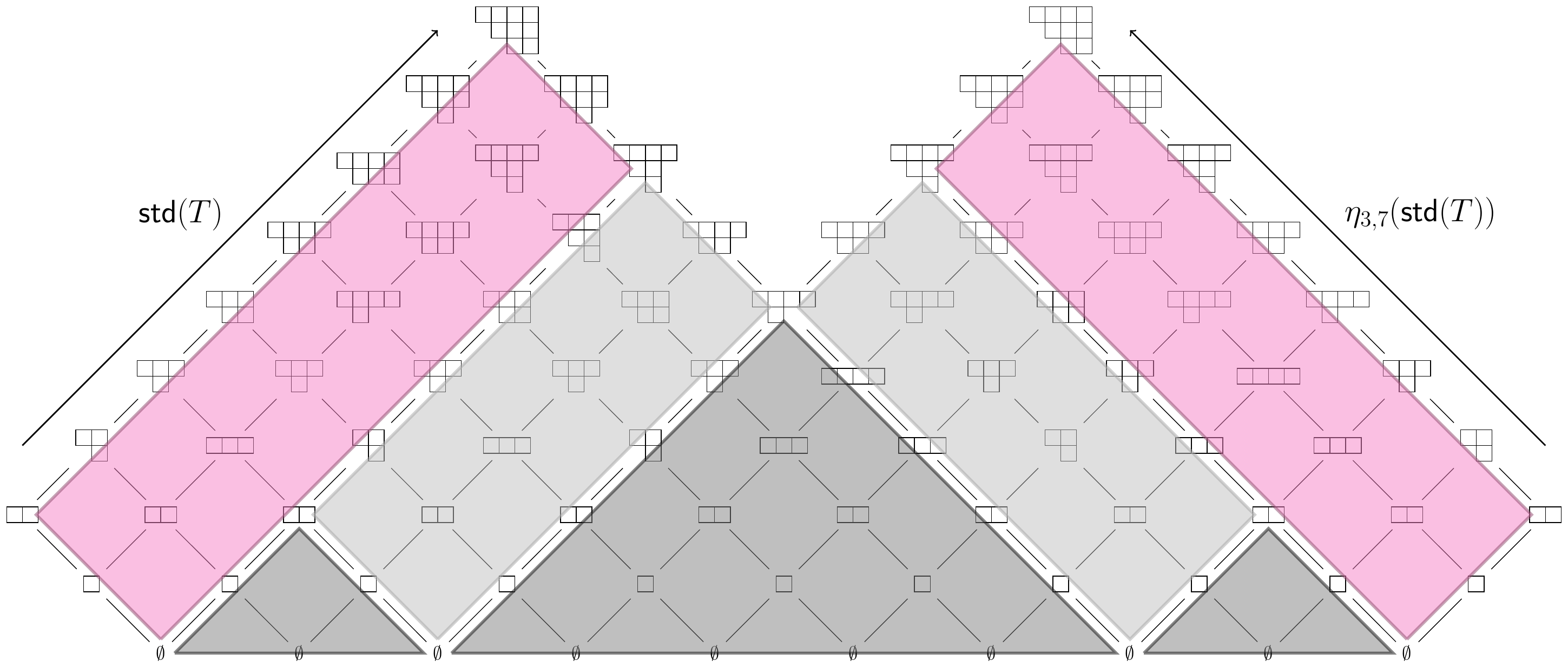}
\caption{A growth diagram to compute $\eta_{3,7}$ on a skew-shaped tableau. A diagram to compute $\mathsf{q}_{3,7}$ on straight-shaped shifted standard tableaux is obtained by removing the pink sections.}
\label{fig:gr_eta_37}
\end{figure}

The shifted growth diagrams may be used to obtain an alternative proof to Theorem \ref{thm:cactusaction}, which then implies Theorem \ref{teo:cact_evaci} and Theorem \ref{teo:cact_sbk}, similarly to the one presented by Chmutov, Glick and Pylyavskyy \cite[Theorem 1.4]{CGP16}. The proof is done for shifted standard tableaux, and may be generalized for the semistandard case using \eqref{eq:eta_std}. More precisely, we will consider the diagram in Figure \ref{fig:gr_cactus_geral}, to prove that the partial Schützenberger involutions satisfy the third cactus relation (recall Definition \ref{def:cactus}),
$$\eta_{i,j} \eta_{k,l} = \eta_{i+j-l,i+j-k} \eta_{i,j}, \quad \text{for}\, [k,l] \subseteq [i,j],$$
when acting on shifted standard tableaux.

\begin{figure}
\includegraphics[scale=0.6]{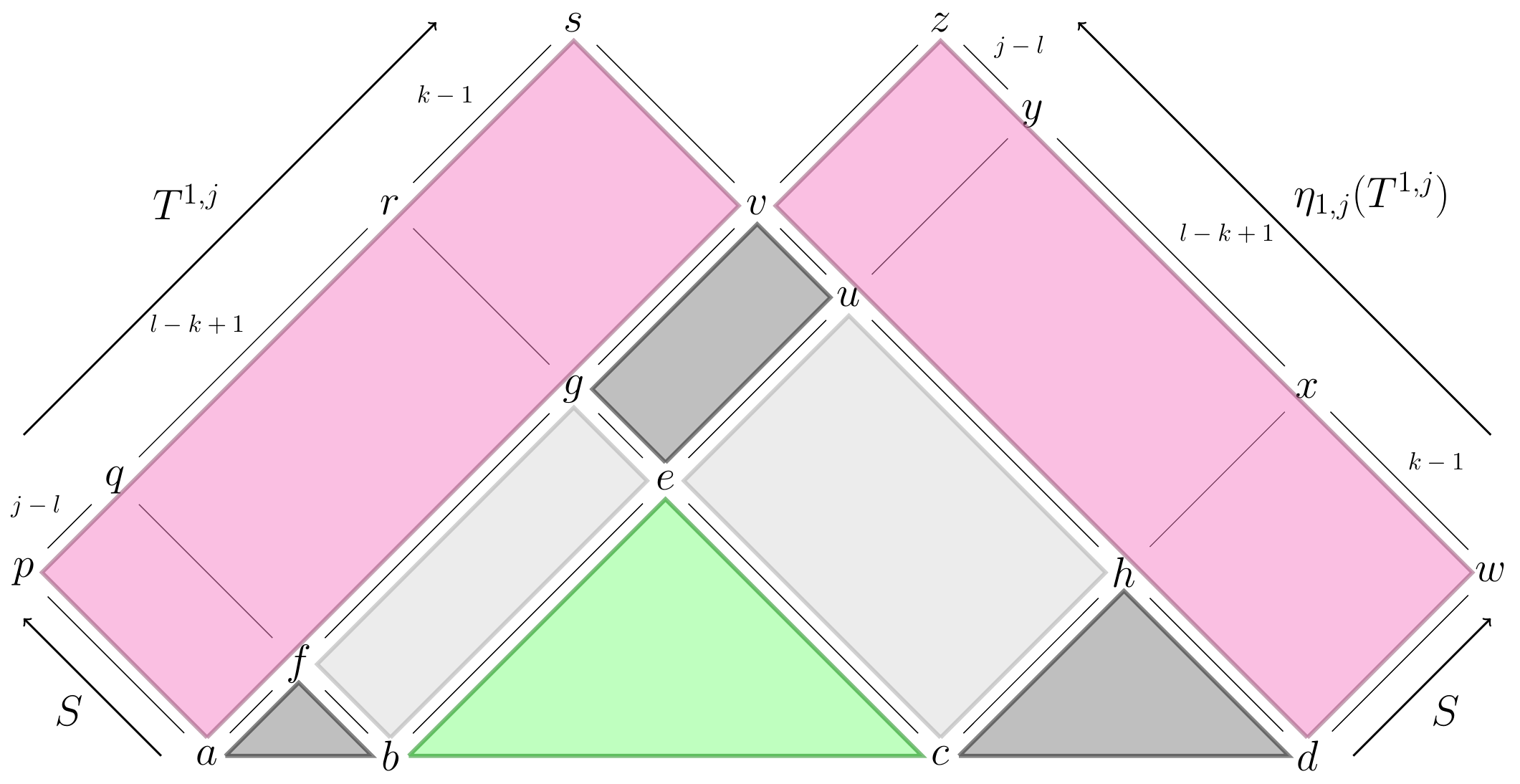}
\caption{A growth diagram with input $T^{1,j}$, a shifted standard tableau of shape $\lambda/\mu$, which is encoded on the segment $[ps]$, and having $\eta_{1,j}(T^{1,j})$ encoded on segment $[wz]$, with $S$ being any standard tableau of shape $\mu$. By construction, $[vs]=[vz]$. The corresponding diagram with primed vertices has $\eta_{1,j} \eta_{k,l} \eta_{1,j} (T^{1,j})$ on the segment $[p's']$ and $\eta_{k,l} \eta_{1,j} (T^{1,j})$ on $[w'z']$. }
\label{fig:gr_cactus_geral}
\end{figure}

\begin{proof}[Proof of Theorem \ref{thm:cactusaction} {\cite[Theorem 5.7]{Ro20b}}]
The relations $\eta_{i,j}^2 = 1$ and $\eta_{i,j} \eta_{k,l} = \eta_{k,l} \eta_{i,j}$, for $[k,l] \cap [i,j] = \emptyset$, are trivial, thus it remains to show that $\eta_{i,j} \eta_{k,l} = \eta_{i+j-l,i+j-k} \eta_{i,j}$, for $[k,l] \subseteq [i,j]$. We claim that it suffices to do that for $i=1$. Indeed, assume that, for any $[k,l] \subset [1,j]$ we have
\begin{equation}\label{eq:cact_rel_claim}
\eta_{1,j} \eta_{k,l} = \eta_{j-l+1,j-k+1} \eta_{1,j}.
\end{equation}
We show that \eqref{eq:cact_rel_claim} implies the third cactus group relation. Indeed, we have $[i,j] \subseteq [1,j]$, and thus \eqref{eq:cact_rel_claim} ensures that
\begin{equation}\label{eq:3cact_1}
\eta_{1,j} \eta_{i,j} = \eta_{1,j-i+1} \eta_{1,j}.
\end{equation}
Moreover, $[k,l] \subseteq [i,j]$ implies that $[k-i+1,l-i+1] \subseteq [1,j-i+1]$ and thus, by \eqref{eq:cact_rel_claim},
\begin{equation}\label{eq:3cact_2}
\eta_{1,j-i+1} \eta_{k-i+1,l-i+1} = \eta_{j-l+1,j-k+1} \eta_{1,j-i+1}.
\end{equation} 
Similarly, since $[j+i-l,j+i-k] \subseteq [1,j]$, we have
\begin{equation}\label{eq:3cact_3}
\eta_{1,j} \eta_{j+i-l,j+i-k} = \eta_{k-i+1,l-i+1} \eta_{1,j}.
\end{equation}
Then, using the fact that $\eta_{i,j}$ is an involution, we have, for any $[i,j]$,
\begin{align*}
\eta_{i,j} \eta_{k,l} &= \eta_{1,j} \eta_{1,j-i+1} \eta_{1,j} \eta_{1,j} \eta_{j-l+1,j-k+1} \eta_{1,j} & \text{by \eqref{eq:cact_rel_claim} and \eqref{eq:3cact_1}} \\
&= \eta_{1,j} \eta_{1,j-i+1} \eta_{j-l+1,j-k+1} \eta_{1,j}\\
&= \eta_{1,j} \eta_{1,j-i+1} (\eta_{1,j-i+1} \eta_{k-i+1,l-i+1} \eta_{1,j-i+1}) \eta_{1,j} & \text{by \eqref{eq:3cact_2}}\\
&= \eta_{1,j} \eta_{k-i+1,l-i+1} \eta_{1,j-i+1} \eta_{1,j}\\
&= \eta_{1,j} (\eta_{1,j} \eta_{j+i-l, j+i-k} \eta_{1,j}) \eta_{1,j-i+1} \eta_{1,j} & \text{by \eqref{eq:3cact_3}}\\
&= \eta_{j+i-l,j+i-k} \eta_{i,j}  & \text{by \eqref{eq:3cact_1}}
\end{align*}

We will now prove \eqref{eq:cact_rel_claim}, using growth diagrams. Let $T \in \mathsf{ShST}(\lambda/\mu,n)$ be a standard tableau and consider the diagram in Figure \ref{fig:gr_cactus_geral}, where the segment $[ap]$ encodes a fixed standard tableau $S$ of shape $\mu$,  $[ps]$ encodes $T^{1,j}$, $[av]$ encodes $\mathsf{rect}(T^{1,j}) = (\mathsf{rect}(T))^{1,j}$, $[dv]$ encodes $\eta_{1,j} (\mathsf{rect} (T^{1,j}))$ and $[wz]$ encodes $\eta_{1,j} (T^{1,j})$. Consider also another growth diagram similar to this one, with the vertices labelled as $\{a', b', c', \ldots\}$, with the segment $[a'p']$ encoding the same $S$ as before,  $[p's']$ encoding $T' := \eta_{1,j} \eta_{k,l} \eta_{1,j} (T^{1,j})$ and $[w'z']$ encoding $\eta_{1,j} (T') = \eta_{k,l} \eta_{i,j} (T^{1,j})$. The proof then mimics the one in \cite[Theorem 1.4]{CGP16}.
Since $[ps]$ and $[p's']$ encode $T^{1,j}$ and $\eta_{1,j} \eta_{k,l} \eta_{1,j} (T^{1,j})$, respectively, we have
\begin{equation}\label{eq:gr_cactus_1}
\begin{split}
[av] &= \mathsf{rect} (T^{1,j})\\
[a'v'] &= \mathsf{rect} (\eta_{1,j} \eta_{k,l} \eta_{1,j} (T^{1,j})).
\end{split}
\end{equation}
Taking the shifted evacuation growth diagrams, for $\eta_{1,j}$, with the inputs in \eqref{eq:gr_cactus_1}, which correspond to $\eta_{1,j}$, and considering that the operators $\eta_{i,j}$ are coplactic, we have
\begin{equation}\label{eq:gr_cactus_2}
\begin{split}
[dv] &= \eta_{1,j} (\mathsf{rect}(T^{1,j}))\\
[d'v'] &= \eta_{1,j} (\mathsf{rect} (\eta_{1,j} \eta_{k,l} \eta_{1,j} (T^{1,j}))) = \eta_{k,l} \eta_{1,j} (\mathsf{rect}(T^{1,j})).
\end{split}
\end{equation}
Thus, in particular, $[d'v'] = \eta_{k,l} [dv]$. Since $[dv] = [dh] \sqcup [hu] \sqcup [uv]$ and $[d'v'] = [d'h'] \sqcup [h'u'] \sqcup [u'v']$, by definition of $\eta_{k,l}$ we have
$$\eta_{k,l} ([dv]) = [dh] \sqcup \eta ([hu]) \sqcup [uv] = [d'v'],$$
and consequently
\begin{equation}\label{eq:gr_cactus_3}
\begin{split}
[dh] &= [d'h'], [uv] = [u'v'],\\
[hu] &= \eta ([h'u'])
\end{split}
\end{equation}  
Since $[dh] = [d'h']$, taking the shifted evacuation growth diagrams on those inputs yield
\begin{equation}\label{eq:gr_cactus_4}
[ch] = [c'h'].
\end{equation}
From \eqref{eq:gr_cactus_3} and \eqref{eq:gr_cactus_4}, considering shifted infusion growth diagrams, we have
$$[ce] = \mathsf{infusion}_1 ([ch],[hu]) = \mathsf{infusion}_1 ([c'h'], \eta ([h'u']))$$
and by Corollary \ref{cor:sw_rev_com}, 
$$\mathsf{infusion}_1 ([c'h'], \eta ([h'u'])) = \eta (\mathsf{infusion}_1 ([c'h'], [h'u'])) = \eta ([c'e'])
$$
and thus
\begin{equation}\label{eq:gr_cactus_5}
[ce] = \eta([c'e']).
\end{equation}
Considering the same shifted infusion growth diagrams, we have
$$[eu] = \mathsf{infusion}_2 ([ch],[hu]) = \mathsf{infusion}_2 ([c'h'], \eta([h'u'])),$$ 
and by Proposition \ref{prop:sw_dual}, as $[h'u']$ is shifted dual equivalent to $\eta ([h'u'])$, we have
$$ 
\mathsf{infusion}_2 ([c'h'], \eta([h'u'])) = \mathsf{infusion}_2 ([c'h'], [h'u']) = [e'u'],
$$
and thus
\begin{equation}\label{eq:gr_cactus_6}
[eu]=[e'u'].
\end{equation}
Considering now the shifted infusion growth diagrams on inputs $[eu]$ and $[uv]$, and on inputs $[e'u']$ and $[u'v']$, respectively, from \eqref{eq:gr_cactus_3} and \eqref{eq:gr_cactus_6}, we have
\begin{equation}\label{eq:gr_cactus_7}
[eg]=[e'g'], [gv]=[g'v'].
\end{equation}
Then, considering the shifted evacuation growth diagrams, on inputs $[be]$ and $[b'e']$, respectively, we have, from \eqref{eq:gr_cactus_5},
$$\eta ([be]) = [ce] = \eta ([c'e']) = [b'e'],$$
and thus
\begin{equation}\label{eq:gr_cactus_8}
\eta([be]) = [b'e'].
\end{equation}
We now consider the shifted infusion growth diagrams on inputs $[be]$ and $[eg]$, and on inputs $[b'e']$ and $[e'g']$, respectively. Then, by Proposition \ref{prop:sw_dual}, since $\eta([be])$ is shifted dual equivalent to $[be]$, we have
$$[bf] = \mathsf{infusion}_1 ([be],[eg]) = \mathsf{infusion}_1 (\eta ([be]),[eg]),$$
and by \eqref{eq:gr_cactus_7} and \eqref{eq:gr_cactus_8},
$$\mathsf{infusion}_1 (\eta ([be]),[eg]) =  \mathsf{infusion}_1 ([b'e'],[e'g']) = [b'f'],$$
and consequently
\begin{equation}\label{eq:gr_cactus_9}
[bf] = [b'f'].
\end{equation}
Finally, taking the shifted evacuation growth diagram with inputs in \eqref{eq:gr_cactus_9}, we get
\begin{equation}\label{eq:gr_cactus_10}
[af] = [a'f'].
\end{equation}
By \eqref{eq:gr_cactus_7} and \eqref{eq:gr_cactus_9}, we have $[gv] = [g'v']$ and $[af]=[a'f']$. Thus, $\mathsf{rect}(T^{1,j})$ agrees with $\mathsf{rect}(\eta_{1,j} \eta_{k,l} \eta_{1,j} (T^{1,j}))$ on the letters outside of $[j-l+1,j-k+1]$ and may differ on the segments $[fg]$ and $[f'g']$. Considering the shifted infusion diagram on inputs $[be]$ and $[eg]$, and on inputs $[b'e']$ and $[e'g']$, respectively, by \eqref{eq:gr_cactus_7} and \eqref{eq:gr_cactus_8}, we have
$$[f'g'] = \mathsf{infusion}_2 ([b'e'], [e'g'])  =  \mathsf{infusion}_2 (\eta ([be]), [eg]),$$
and by Corollary \ref{cor:sw_rev_com}, we have
$$\mathsf{infusion}_2 (\eta ([be]), [eg]) = \eta (\mathsf{infusion}_2 ([be],[eg])) = \eta ([fg]),$$
and thus,
\begin{equation}\label{eq:gr_cactus_10a}
\eta([fg]) = [f'g'].
\end{equation}
Then, from the definition of $\eta_{j-l+1,j-k+1}$ and the fact that it is coplactic, we have
\begin{align*}
\mathsf{rect}(\eta_{1,j} \eta_{k,l} \eta_{1,j} (T^{1,j})) &= [a'v'] = [a'f'] \sqcup [f'g'] \sqcup [g'v']\\
&= [af] \sqcup \eta([fg]) \sqcup [gv] &\text{by \eqref{eq:gr_cactus_7}, \eqref{eq:gr_cactus_10} and \eqref{eq:gr_cactus_10a}}\\
&= \eta_{j-l+1,j-k+1} ([av])\\
&= \eta_{j-l+1,j-k+1} (\mathsf{rect} (T^{1,j}))\\
&= \mathsf{rect} (\eta_{j-l+1,j-k+1} (T^{1,j})).
\end{align*}
It remains to show that the segments $[ps]$ and $[p's']$ differ only on $[qr]$ and $[q'r']$, and that $\eta ([qr]) = [q'r']$. We have $[pq] = T^{1,j-l}$. By the definition $\eta_{1,j} \eta_{k,l} \eta_{1,j}$ and since $j-l \leq j$, we have
\begin{align*}
[p'q'] &= (\eta_{1,j} \eta_{k,l} \eta_{1,j} (T^{1,j}))^{1,j-l}\\
&= \eta_{1,j} \eta_{k,l} \eta_{1,j} (T^{1,j-l})\\
&= \eta_{1,j} \eta_{k,l} \eta_{1,j} ([pq]).
\end{align*}
We recall that, by construction, $[ap] = [a'p'] = S$. Since the partial Schützenberger involutions preserve shifted dual equivalence, $[pq]$ is shifted dual equivalent to $[p'q']$, and thus, by Proposition \ref{prop:sw_dual}, we have
$$
[fq] = \mathsf{infusion}_2 ([ap],[pq]) = \mathsf{infusion}_2 ([a'p'],[p'q']) = [f'q'],
$$
that is,
\begin{equation}\label{eq:gr_cactus_11}
[fq] = [f'q'].
\end{equation}
Then, by \eqref{eq:gr_cactus_10} and \eqref{eq:gr_cactus_11},
$$
[pq] = \mathsf{infusion}_2 ([af],[fq]) = \mathsf{infusion}_2 ([a'f'],[f'q']) = [p'q'],
$$
and thus
\begin{equation}\label{eq:gr_cactus_12}
[pq] = [p'q'].
\end{equation}
Since $[p's'] = \eta_{1,j} \eta_{k,l} \eta_{1,j} (T^{1,j}) = \eta_{1,j} \eta_{k,l} \eta_{1,j} ([ps])$, and the partial Schützenberger involutions preserve shifted dual equivalence, then $[ps]$ is shifted dual equivalent to $[p's']$. Then, Proposition \ref{prop:sw_dual} and the fact that $[ap] = [a'p']$ ensure that
$$
[vs] = \mathsf{infusion}_2 ([ap],[ps])= \mathsf{infusion}_2 ([a'p'],[p's']) = [v's'],
$$
that is,
\begin{equation}\label{eq:gr_cactus_13}
[vs] = [v's'].
\end{equation}
Then, by \eqref{eq:gr_cactus_7} and \eqref{eq:gr_cactus_13},
$$
[rs] = \mathsf{infusion}_2 ([gv],[vs])= \mathsf{infusion}_2 ([g'v'],[v's']) = [r's']
$$
and then,
\begin{equation}\label{eq:gr_cactus_14}
[rs] = [r's'].
\end{equation}
From \eqref{eq:gr_cactus_7} and \eqref{eq:gr_cactus_13} we also conclude that
$$
[gr]  = \mathsf{infusion}_1 ([gv],[vs])= \mathsf{infusion}_1 ([g'v'],[v's']) = [g'r'] 
$$
and thus, by \eqref{eq:gr_cactus_7}, we have
\begin{equation}\label{eq:gr_cactus_15}
[er] = [eq] \sqcup [gr] = [e'q'] \sqcup [g'r'] = [e'r'].
\end{equation}
Then, by \eqref{eq:gr_cactus_8} and \eqref{eq:gr_cactus_15}, we have
$$[q'r'] =  \mathsf{infusion}_2 ([b'e'],[e'r']) = \mathsf{infusion}_2 (\eta([be]), [er])$$
 and by Corollary \ref{cor:sw_rev_com},
$$\mathsf{infusion}_2 (\eta([be]), [er]) = \eta (\mathsf{infusion}_2 ([be], [er])) = \eta ([qr]), $$ 
and then
\begin{equation}\label{eq:gr_cactus_16}
\eta([qr]) = [q'r'].
\end{equation}
To conclude the proof, we remark that by the definition of $\eta_{j-l+1,j-k+1}$, we have
\begin{align*}
\eta_{1,j} \eta_{k,l} \eta_{1,j} (T^{1,j}) &= [p's'] = [p'q'] \sqcup [q'r'] \sqcup [r's']\\
&= [pq] \sqcup \eta([qr]) \sqcup [rs] & \text{by \eqref{eq:gr_cactus_12}, \eqref{eq:gr_cactus_14} and \eqref{eq:gr_cactus_16}}\\
&= \eta_{j-l+1,j-k+1} ([ps])\\
&= \eta_{j-l+1,j-k+1} (T^{1,j}).
\end{align*}
\end{proof}

\section*{Acknowledgements}
The author is grateful to her supervisors O. Azenhas and M. M. Torres, and acknowledges the hospitality of the Department of Mathematics of University of Coimbra. 
Thanks also to Seung-Il Choi and Thomas Lam for pointing to O. Azenhas the references  \cite{CNO17,CNOrev} and \cite{CGP16}, respectively, during the FPSAC18 and the 82nd Séminaire Lotharingien de Combinatoire. 

The author is partially supported by the LisMath PhD program (funded by the Portuguese Science Foundation). This research was made within the activities of the Group for Linear, Algebraic and Combinatorial Structures of the Center for Functional Analysis, Linear Structures and Applications (University of Lisbon), and was partially supported by the Portuguese Science Foundation, under the project UIDB/04721/2020.

\bibliographystyle{siam}
\bibliography{bibliography}

\end{document}